\documentclass[11pt]{amsart}

\renewcommand{\Re}{\operatorname{Re}}
\renewcommand{\Im}{\operatorname{Im}}

\newcommand{\defeq}{\stackrel{\rm{def}}{=}}

\usepackage{amssymb}
\usepackage{amsthm}
\usepackage{amsxtra}
\usepackage{graphicx}
\usepackage{color}
\usepackage{amsmath}
\usepackage{listings}
\usepackage{verbatim}
\usepackage{float}
\usepackage{cases}
\usepackage{threeparttable}
\usepackage{dcolumn}
\usepackage{multirow}
\usepackage{booktabs}

\newcommand{\ds}{\displaystyle}

\newtheorem{theorem}{Theorem}[section]
\newtheorem{definition}[theorem]{Definition}
\newtheorem{proposition}{Proposition}[section]
\newtheorem{lemma}[proposition]{Lemma}
\newtheorem{corollary}[proposition]{Corollary}

\newtheorem{conjecture}{Conjecture}

\theoremstyle{remark}
\newtheorem{remark}[proposition]{Remark}

\numberwithin{equation}{section}

\title[Blow-up in {\lowercase{g}}Hartree equation]
{Stable blow-up dynamics in the $L^2$-critical and\\ $L^2$-supercritical generalized Hartree equation}

\author[K. Yang]{Kai Yang}
\address{Department of Mathematics  \& Statistics\\Florida International University,  Miami, FL, USA}
\curraddr{}
\email{yangk@fiu.edu}

\author[S. Roudenko]{Svetlana Roudenko}
\address{Department of Mathematics \& Statistics\\Florida International University,  Miami, FL, USA}
\curraddr{}
\email{sroudenko@fiu.edu}

\author[Y. Zhao]{Yanxiang Zhao}
\address{Department of Mathematics\\George Washington University,  Washington, DC, USA}
\curraddr{}
\email{yxzhao@gwu.edu}

\subjclass[2010]{35Q55, 35Q40, 65M70, 65N35} 

\keywords{Hartree equation, Choquard equation, nonlocal potential, convolution nonlinearity, dynamic rescaling, log-log blow-up, multi-bump profile, adiabatic regime}

\begin{document}
\begin{abstract}
We study stable blow-up dynamics in the generalized Hartree equation with radial symmetry, which is a Schr\"odinger-type equation with a nonlocal, convolution-type nonlinearity: $iu_t+\Delta u +\left(|x|^{-(d-2)} \ast |u|^{p} \right) |u|^{p-2}u = 0$, 
$x \in \mathbb{R}^d$. 
First, we consider the $L^2$-critical case in dimensions $d=3, 4, 5, 6, 7$ and obtain that a generic blow-up has a self-similar structure and exhibits not only the square root blowup rate $\| \nabla u(t)\|_{L^2} \sim (T-t)^{-\frac{1}{2}}$, but also the {\it log-log} correction (via asymptotic analysis and functional fitting), thus, behaving similarly to the stable blow-up regime in the $L^2$-critical nonlinear Schr\"{o}dinger (NLS) equation. 
In this setting we also study blow-up profiles and show that generic blow-up solutions converge to the rescaled $Q$, a ground state solution of the elliptic equation $-\Delta Q+Q- \left(|x|^{-(d-2)} \ast |Q|^p \right) |Q|^{p-2} Q =0$. 

We also consider the $L^2$-supercritical case in dimensions $d=3, 4$.  
We derive the profile equation for the self-similar blow-up and establish the existence and local uniqueness of its solutions. As in the NLS $L^2$-supercritical regime, the profile equation exhibits branches of non-oscillating, polynomially decaying (multi-bump) solutions.  
A numerical scheme of putting constraints into solving the corresponding ODE is applied during the process of finding the multi-bump solutions. Direct numerical simulation of solutions to the generalized Hartree equation by the dynamic rescaling method indicates that the $Q_{1,0}$  is the profile for the stable blow-up. In this supercritical case, we obtain the 
blow-up rate without any correction. This blow-up happens at the focusing level $10^{-5}$, and thus, numerically observable (unlike the $L^2$-critical case). 

In summary, we find that the results are similar to the behavior of stable self-similar blowup solutions in the corresponding settings for the nonlinear Schr\"odinger equation. Consequently, one may expect that the form of the nonlinearity in the Schr\"odinger-type equations is not essential in the stable formation of singularities.  
\end{abstract}

\maketitle
\tableofcontents
\section{Introduction}
We consider the Cauchy problem of the generalized Hartree (gHartree) equation:
\begin{align}\label{gHartree0}
\begin{cases}
iu_t + \Delta u +\left( \dfrac{1}{|x|^{b}} \ast |u|^p \right) |u|^{p-2} u =0, \quad (t,x) \in \mathbb{R}\times \mathbb{R}^d,  \\
u_0=u(x,0) \,  \in H^1(\mathbb{R}^d).
\end{cases}
\end{align}
Here, the $\ast$ represents the convolution in $\mathbb{R}^d$ with the convolution power $0<b<d$ and the nonlinearity power typically $p\geq 2$, though we will also consider cases with $p>1$ (details below).
When $p=2$, the equation \eqref{gHartree0} is the well-known Hartree equation
\begin{equation}\label{E2}
iu_t+\Delta u +\left( \dfrac{1}{|x|^b}*|u|^2\right)u=0,
\end{equation}
which arises, for example, in the description of dynamics in Bose-Einstein condensates (BEC) with long-range attractive interaction, proportional to $1/|x|^b$ and arbitrary angular dependence, for example, see \cite{Lushnikov2010}, \cite{Le2009}. It appears as the mean field limit of quantum Bose gases \cite{FL2003}, and is also used to describe a certain type of a trapped electron \cite{Li1977}, see also \cite{GV1979aa}, \cite{GV2000}, \cite{Li1980}, \cite{Li2003}. 
Within the pseudo-relativistic setting, if the Laplacian in \eqref{gHartree0} is replaced by $\sqrt{m^2 - \Delta}$ and $b=1$, then the equation 
\begin{equation}\label{E3}
i u_t+ \sqrt{m^2 -\Delta} \, u +\left( \dfrac{1}{|x|}*|u|^2\right)u=0, \quad x \in \mathbb{R}^3,
\end{equation}
appears in the description of boson stars, see \cite{FJL2007}. 


The well-posedness theory of the equation \eqref{E2}
is obtained by Ginibre and Velo in \cite{GV1979aa} (see also \cite{Ca2003}). 
For a general nonlinearity $p \geq 2$, the $H^1$ well-posedness is obtained in \cite{KR2019}, for $\dot{H}^s$ well-posedness, see \cite{AR2019b}. 
Let $(T_-, T_+)$ denote the maximal time interval of existence of solutions to (\ref{gHartree0}), that is, for given initial data $u_0 \in H^1(\mathbb{R}^d)$, one has $u(t) \in (\mathbb{C}(T_-,T_+), H^1(\mathbb{R}^d))$. Without loss of generality, we consider the solutions in forward time $T>0$. We say that the solution to the equation  (\ref{gHartree0}) is locally well-posed if $T<\infty$, and it is globally well-posed if $T=\infty$. If $T<\infty$, then we say that the solution blows up in finite time. In the energy-subcritical cases, this means $\lim_{t \nearrow T}\|u(x,t)\|_{H^1_x}=\infty$. We discuss other cases later in the paper.

During their lifespan, solutions of (\ref{gHartree0}) conserve mass and energy (Hamiltonian):
\begin{align*}
& M[u(t)]:=\int_{\mathbb{R}^d} |u(x,t)|^2dx=M[u_0], \quad&( \mathrm{Mass})\\
& E[u(t)]:=\frac{1}{2} \int_{\mathbb{R}^d} |\nabla u|^2 dx-\frac{1}{2p}\int_{\mathbb{R}^d} \left(\frac{1}{|x|^b} \ast |u|^{p}\right)|u|^p dx=E[u_0]. \quad & (\mathrm{Energy})
\end{align*}
Since we only consider radial solutions, we omit conservation of momentum.

The equation (\ref{gHartree0}) has scaling invariance similar to the nonlinear Schr\"{o}dinger (NLS) equation.  Let $u(x,t)$ be the solution to (\ref{gHartree0}), then one can see that $u_{\lambda}(x,t)=\lambda^{\frac{d-b+2}{2(p-1)}}u(\lambda x, \lambda^2 t)$ is also a solution to (\ref{gHartree0}). 

The criticality comes from the scaling invariance of $\dot{H}^s$ norm, i.e., 
$\|u(x,t)\|_{\dot{H}^s_x}=\|u_{\lambda}(x,t)\|_{\dot{H}^s_x}$. The direct calculation leads us to
\begin{equation}\label{criticality}
s=\frac{d}{2}-\frac{d-b+2}{2(p-1)}.
\end{equation}

If $s=0$, the equation \eqref{gHartree0} is referred to as the $L^2$-critical (or mass-critical as it preserves the mass, $L^2$-norm). If $s=1$, the equation is $\dot{H}^1$-critical (or energy-critical as it preserves the energy). If $0<s< 1$, the equation is mass-supercritical and energy-subcritical (or inter-critical), and finally, it is energy-supercritical if $s>1$.

When $s \geq 0$, solutions can blow-up in finite time (for example, if initial data has negative energy and finite initial variance $\mathcal{V}(0) < \infty$, where $\mathcal{V}(t)= \int_{\mathbb{R}^d}|x|^2|u(x,t)|^2dx$), while there are globally-existing-in-time solutions as well (i.e., global well-posedness holds for some set of solutions, see \cite{GV1979aa}, \cite{Ca2003}, \cite{KR2019}, \cite{AR2019b}).

In this paper, we restrict our attention to the power $b=d-2$ and the dimensions $d > 2$. There are two reasons for that. The first one is that this is exactly the case when the convolution is the fundamental solutions of the Poisson equation, and thus, the nonlocal term can be written as 
$$
\frac{1}{|x|^{d-2}} \ast |u|^p=\alpha(d)(-\Delta)^{-1}|u|^p,
$$
where $\alpha(d)$ is the dimensional constant. In this case, the criticality \eqref{criticality} becomes
\begin{equation}\label{criticality2}
s=\frac{d}{2}-\frac{2}{p-1}.
\end{equation}
The dimensional constant $\alpha(d)$ can be removed by scaling, thus, the equation (\ref{gHartree0}) is reduced to
\begin{align}\label{gHartree1}
\begin{cases}
iu_t + \Delta u +\left( (-\Delta)^{-1} |u|^p \right) |u|^{p-2} u =0, \, (t,x) \in[0,T)\times \mathbb{R}^d,  \\
u_0=u(x,0)\in H^1(\mathbb{R}^d).
\end{cases}
\end{align}

The second reason for choosing $b=d-2$, is the solitary wave solutions to \eqref{gHartree1}. 
Similar to the NLS equation, when $b=d-2$ and $p<1+\frac4{d-2}$ ($s<1$), we consider standing wave solutions to \eqref{gHartree1} of the form $u(x,t) = e^{it}Q(x)$ with $Q$ being the positive-vanishing-at-infinity solution of 
\begin{equation}\label{GS-conv}
-\Delta Q+Q- \left( \frac{1}{|x|^{d-2}} \ast |Q|^p \right) |Q|^{p-2} Q =0, \end{equation}
or equivalently,
\begin{equation}\label{GS}
-\Delta Q+Q- \left( (-\Delta)^{-1} |Q|^p \right) |Q|^{p-2} Q =0.
\end{equation}

The existence and uniqueness of the real, positive, vanishing at infinity solution to \eqref{GS-conv}, or \eqref{GS}, are obtained for $p=2$ in \cite{Li1977} ($d=3$), \cite{KLR2009} ($d=4$), \cite{KR2019} ($2<d<6$); for $p=2+\epsilon$ in \cite{Xiang2016},
otherwise, it is not known; the existence with decay and other properties in a general case is investigated in \cite{MS2013}, see also an excellent review in \cite{MS2017}. 

This solution is known as the ground state solution, which we also denote by $Q$. Note that the ground state solution is radially symmetric $Q=Q(r)$ and is exponentially decaying at infinity for $p \geq 2$, see for example, \cite{MS2013}.
While there is no explicit formula for the ground state solution $Q$, we can obtain the profiles numerically (e.g., via the renormalization method similar to the NLS in \cite[Chapter 28]{F2015}, see Appendix).

In this paper, we are interested in studying stable blow-up dynamics of solutions to the equation \eqref{gHartree1} in the $L^2$-critical case ($p=1+ \frac{4}{d}$) and in the $L^2$-supercritical case ($p > 1+ \frac{4}{d}$). As in the NLS equation, in the $L^2$-critical case, some blow-up solutions (of minimal mass) can be obtained via the pseudo-conformal transformation. However, these blow-up solutions are unstable. We are interested in stable blow-up solutions of \eqref{gHartree1}, at least in those solutions, which can be observed numerically from a generic initial data (such as Gaussian initial conditions). 
The scaling invariance is the underlying mechanism for the dynamic rescaling method that we use to simulate the blow-up solutions (this is in the spirit of \cite{LPSS1988}, \cite{SS1999}, also \cite{YRZ2018}, \cite{YRZ2019}, see Section \ref{Section-DR} for details).
In particular, we will investigate the blow-up rate and blowup profiles of singular solutions to the gHartree equation \eqref{gHartree1} in the critical and supercritical settings. 

We first recall the definition of the blow-up rate (e.g. from \cite{FMR2006}, \cite{FGW2007}, or \cite{F2015}), which is used in the standard NLS equation.
\begin{definition}\label{D: blowup-rate}
The blow-up rate is the function $f(t)$ (e.g., $f(t)=(T-t)^{-\frac{1}{2}}$) such that 
\begin{align}
\lim_{t\nearrow T} \dfrac{ \| \nabla u(t)\|_{L^2_x}}{f(t)} =C,
\end{align}
where $C$ is a constant.
\end{definition}

The above definition 
uses the $\dot{H}^{1}$ norm, note that due to scaling invariance when $s_c=1$, the norm $\|u(t)\|_{\dot{H}^1}$ becomes constant, and when $s_c>1$, then $\|u(t)\|_{\dot{H}^1}$ decreases to zero (see \eqref{blowup rate L} with $L(t) \to 0$). On the other hand, in the numerical simulations we observe that the solution is concentrating to a point with its amplitude growing to infinity in finite time.
Thus, instead of tracking the $\dot{H}^1$ norm, one can also study the blow-up rate in terms of the $L^{\infty}$ norm.
\begin{definition}\label{D: blowup-rate2}
The blow-up rate is the function $f(t)$ (e.g., $f(t)=(T-t)^{-\frac{1}{2}}$) such that 
\begin{align}
\lim_{t\nearrow T} \dfrac{ \| u(t)\|_{L^{\infty}_x}}{f(t)} =C,
\end{align}
where $C$ is a constant.
\end{definition}
For the $L^2$-critical NLS equation, it is known that Definitions \ref{D: blowup-rate} and \ref{D: blowup-rate2} are equivalent, see \cite{Ca2003}, \cite{FGW2005}. The study of blow-up rates go back to 1970's, mainly in the two-dimensional cubic NLS ($L^2$-critical) and, in part, for the 3d cubic NLS ($L^2$-supercritical) equations, see \cite{BZS1975}, \cite{GRH1980}, \cite{SSP1984}. From scaling and local well-posedness it follows that the lower bound on the blow-up rate is $(T-t)^{-\frac{1}{2}}$. In 1986, McLaughlin, Papanicolaou, C. Sulem and P. Sulem in \cite{MPSS1986} introduced the dynamic rescaling method to track the blow-up profile and the rate, and suggested that there should be a correction terms to the rate $(T-t)^{-\frac{1}{2}}$. Previously, Talanov (1978), Wood (1984) and Rypdal and Rasmussen (1986) suggested the rate $(|\ln(T-t)|/(T-t))^{\frac{1}{2}}$ from a different approach (see \cite{VPT1978}, \cite{Wood1984}, \cite{RR1986}). Using the far asymptotics of the ground state, and considering a slightly supercritical equation by treating the dimension $d$ as a continuous parameter, Landman, Papanicolaou, C. Sulem, and P. Sulem in \cite{LPSS1988}, and also LeMesurier, Papanicolaou, C. Sulem and P. Sulem in \cite{LePSS1988} (see also an earlier work of Fraiman \cite{Fr1985}) concluded that the rate of the stable blow-up is of the form $\left( \ln|\ln(T-t)|/(T-t) \right)^{\frac{1}{2}}$, now commonly referred to as the {\it log-log} law, see also \cite{LLePSS1989}, \cite{DNPZ1992} and books \cite{SS1999}, \cite{F2015}. We note that numerically it is not possible (at least with the current computational power) to observe such a double log correction, however, the asymptotic analysis (e.g., as in \cite{SS1999}) produces such a correction; numerically, it is only possible to do the functional fitting and examine stabilization properties of the convergence (see \cite{ADKM2003} and \cite{YRZ2018}, also Section 4.2 below). This {\it log-log} rate holds extremely close to the blow-up time, and before the singularity formation gets into the {\it log-log} regime, it goes through the adiabatic phase, which has been described by Malkin or Fibich adiabatic laws (see \cite{Ma1993} and \cite{FP1998}); the rate in that penultimate regime is proportional to $(|\ln(T-t)|^\gamma/(T-t))^{\frac{1}{2}}$, see \cite{ADKM2003} for the 2d cubic NLS or our work \cite{YRZ2018} for various other dimensions in the $L^2$-critical setting.

Theoretical studies of stable self-similar blow-up dynamics, including rates, in the $L^2$-critical NLS-type equations have been going on since 2000's (starting with Galina Perelman's work for the 1d quintic NLS \cite{Pe2001}, followed by a series of works by Merle \& Raphael \cite{MR2005a}-\cite{MR2005}.
Various perturbations of nonlinearity have been studied as well, tracking the blow-up rates for various singular solutions, for example, see \cite{LeCoMR2016}, \cite{MR2015}, though most of those works are not stable blow-up solutions as breaking radial symmetry or other perturbations will break the set-up and force to blow-up in the {\it log-log} regime, provided enough mass is available. 
While various perturbations of nonlinearities have been considered in the literature (for example, in the $L^2$-critical setting), it is far from being understood how blow-up dynamics depends on the form of the nonlinearity (for example, if a nonlinearity has a significant influence on the stable blow-up rate). This work is a step in that direction. We study how a nonlocal nonlinearity affects the stable blow-up dynamics. This is also important in connection with understanding gravitational collapse of \eqref{E3}, where currently only the existence of blow-up is known, see \cite{FL2003} and also \cite{KLR2009}.

In this paper we investigate the following conjectures: 
\begin{conjecture}[$L^2$-critical gHartree] \label{C:1}
A stable blow-up solution to the $L^2$-critical gHartree equation has a self-similar structure and comes with the rate
$$ 
\lim_{t \rightarrow T}\|\nabla u(\cdot,t) \|_{L^2_x} =\left( \frac{\ln|\ln(T-t)|}{2\pi(T-t)} \right)^{\frac{1}{2}} \quad \mbox{as} \quad {t \to T},
$$
known as the {\it log-log} rate. Thus, the solution blows up in a self-similar regime with profile converging to a rescaled profile $Q$, which is a ground state solution of \eqref{GS-conv}, namely,
$$ 
u(x,t) \sim \dfrac{1}{L(t)^{\frac{d}{2}}} Q\left(\frac{x-x(t)}{L(t)}\right) e^{i\gamma(t)} 
$$
for some parameter $\gamma(t)$. The stable blow-up dynamics in the $L^2$-critical gHartree equation is similar to the stable blow-up dynamics in the $L^2$-critical NLS equation. 
\end{conjecture}
 
\begin{conjecture}[$L^2$-supercritical gHartree] \label{C:2}
A stable blow-up solution for the $L^2$-supercritical gHartree equation is of the self-similar form
\begin{align}\label{self-similar sol}
u(x,t) \sim \dfrac{1}{L(t)^{\frac{2}{p-1}}} Q\left(\frac{x-x(t)}{L(t)}\right) \exp \left({i \theta + \frac{i}{2a}\log \frac{T}{T-t}} \right),
\end{align}
where the blow-up profile  $Q$ is the $Q_{1,0}$ solution of the profile equation \eqref{Q eqn}, with the specific constant $a$ and rate $L(t)=(2a(T-t))^{1/2}$ (see Section \ref{Section-Existence} for the notation and details). Consequently,
$$
\| \nabla u(\cdot, t)\|_{L_x^2} \sim \left(2a(T-t) \right)^{-\frac{1}{2}(1-s_c)} \quad \mbox{as} \quad t\rightarrow T
$$
by a direct calculation. This dynamics is similar to the stable blow-up dynamics in the $L^2$-supercritical NLS equation.
\end{conjecture}

We prove existence of profiles $Q$ to \eqref{Q eqn} abd find their decay before we numerically investigate the above Conjectures. We give numerical confirmation to Conjecture 1 in dimensions $d=3,4,5,6,7$ and to Conjecture 2 in dimensions $d=3,4$. In particular, we show that the rates in the stable blow-up dynamics do not depend on the local or non-local type of nonlinearity in the NLS-type equation, at least in the radial case. The profile in the $L^2$-critical case is a ground state solution of \eqref{GS} and in the $L^2$-supercritical regime, the profile equation \eqref{Q eqn} exhibits branches of slowly oscillating multi-bump solutions. 

To study the blow-up solutions, we adapt the dynamic rescaling method to the generalized Hartree equation and use it in both critical and supercritical cases. For the $L^2$-critical case, we find that generic blow-up happens with the rate $\left( \ln|\ln(T-t)|/(T-t) \right)^{\frac{1}{2}}$, which we also refer to as the {\it log-log} blow-up rate, with the self-similar blow-up profiles converging to $Q$ up to rescaling, that is, $|u(x,t)| \sim | {L^{-\frac{d}{2}}(t)} Q({x}/{L(t)}) |$, where $L(t) \approx \left( \ln|\ln(T-t)|/(T-t) \right)^{\frac{1}{2}} $. 
For the $L^2$-supercritical case, we obtain that the blow-up rate is {$\| \nabla u(t)\|_{L_x^2} \sim (T-t)^{-\frac{1}{2}(1-s_c)}$}, and we observe that it also blows up with the self-similar profile $Q$, which is different from a ground state solution of \eqref{GS-conv}. We show the existence and the ``local uniqueness" of such self-similar profile $Q$ for the case $0<s_c<2$. Numerically, we find that such $Q$ can have multiple slowly decaying solutions. Similar to the NLS $L^2$-supercritical case \cite{SS1999}, the existence of complex solutions of the rescaled static states $Q$ and the slow decay (not in $L^2$) makes it challenging to analyze the supercritical blow-up dynamics. Nevertheless, we do find the blow-up profile and the blow-up rate in this case, see Section \ref{Section-Supercritical}.

This paper is organized as follows. In Section \ref{S:profiles} we discuss existence and decay of profiles. In Section \ref{Section-DR}, we describe the dynamic rescaling method for the gHartree equation. In Section \ref{Section-Critical}, we discuss the $L^2$-critical case, obtaining the square root blow-up rate and the {\it log-log} correction. Besides numerical and asymptotic investigations, we also discuss the adiabatic regime occurring prior to the {\it log-log} regime. We also observe that blow-up profiles converge to the rescaled ground state $Q$ in our numerical simulations. In Section \ref{Section-Supercritical}, we discuss the $L^2$-supercritical cases (including $s_c>1$). 
Numerically, we obtain the profile $Q_{1,0}$, and justify that the blow-up solutions do converge to that blow-up profiles. We also obtain the blow-up rates, with the precision of $10^{-5}$ to the predicted blow-up rates. We finish with the Appendix discussing the computation of $Q$ via the renormalization method. 

\textbf{Acknowledgments}.
KY and YZ would like to acknowledge Yongyong Cai, who hosted their visit to the Computational Science Research Center (CSRC) in Beijing during Summer 2017. KY would thank Anudeep Kumar Arora for his help on the background of the gHartree equation and further discussions and clarifications.
S.R. was partially supported by the NSF CAREER grant DMS-1151618/1929029 and DMS-1815873/1927258 
as well as part of the KY's research and travel support to work on this project came from the above grants. SR would also like to thank the AROOO (`A Room of Ones's Own')
initiative for focused research time for this project.
YZ was partially supported by the Simons Foundation Grant No. 357963.

\section{Preliminaries on ground states and profiles}\label{S:profiles}

We start with applying the scaling invariance property to finite time existing solutions of (\ref{gHartree1}), which makes solutions of the rescaled equation exist globally in time. For consistency with literature we write the power $p=2\sigma+1$ and set (here, $r = |x|$)
\begin{equation}\label{rescaled-variables}
u(r,t)=\frac{1}{L^{1/\sigma}(t)} v(\xi, \tau), \qquad \xi=\frac{r}{L(t)}, \quad \mbox{and} \quad \tau=\int_0^t\frac{ds}{L^2(s)}.
\end{equation}
The direct calculation of this substitution into \eqref{gHartree1} yields
\begin{equation}\label{RgHartree}
iv_{\tau}+ia(\tau)\left(\frac{v}{\sigma}+\xi v_{\xi}\right)+\Delta v + \left((-\Delta)^{-1}|v|^{2\sigma+1} \right)|v|^{2\sigma-1} v=0, 
\end{equation}
where
\begin{equation}\label{a form}
a(\tau)=-L \frac{dL}{dt}=-\frac{d \ln L}{d\tau}.
\end{equation}
As in the NLS case, studying the parameter $L(t)$ will clarify the blow-up rate of the solutions, which differs for the $L^2$-critical vs. the $L^2$-supercritical cases, exactly because of the asymptotic behavior of the parameter $a(\tau)$ (we show that in the gHartree equation it will tend to zero in the $L^2$-critical case and to a nonzero constant in the supercritical case). Therefore, we study those cases separately.

Before that, we discuss some preliminaries on the profile equation and suitable solutions for the blow-up profiles. For that we assume that $a(\tau) \to a$, some specific constant, which we will obtain later numerically. 

We note that  the behavior of solutions as $t\rightarrow T$ in the original equation (\ref{gHartree1}) can be reconstructed from those to the rescaled equation (\ref{RgHartree}) as $\tau \rightarrow \infty$.

\subsection{Profile equation} \label{Section-Existence}
We separate variables $v(\xi,\tau)=e^{i\tau}Q(\xi)$ in \eqref{RgHartree} and obtain 
\begin{equation}\label{Q eqn-tau}
\Delta_{\xi} Q -Q+ia(\tau)\left( \dfrac{Q}{\sigma}+\xi Q_{\xi} \right) +((-\Delta)^{-1}|Q|^{2\sigma+1})|Q|^{2\sigma-1}Q=0,
\end{equation}
here, $\Delta_{\xi}:=\partial_{\xi \xi}+\frac{d-1}{\xi} \partial_{\xi}$ denotes the Laplacian with radial symmetry. Assuming that $a(\tau)$ converges to a constant $a$, instead of \eqref{Q eqn-tau} in this section we study the following problem
\begin{align}\label{Q eqn}
\begin{cases}
\Delta_{\xi} Q -Q+ia\left( \dfrac{Q}{\sigma}+\xi Q_{\xi} \right) +((-\Delta)^{-1}|Q|^{2\sigma+1})|Q|^{2\sigma-1}Q=0,\\
Q_{\xi}(0)=0,\qquad Q(0)=\mathrm{real}, \qquad Q(\infty)=0.
\end{cases}
\end{align}
The first condition for $Q$ indicates that the local maximum is at zero. The second condition on $Q$ shows that we fix the phase of the solutions, since the equation is phase invariant; the last condition means that $Q(\xi) \to 0$ as $\xi \to \infty$. Moreover, we will seek for solutions, which have $|Q(\xi)|$ decreasing monotonically with $\xi$, without oscillations as $\xi \to \infty$.  

Understanding solutions of the stationary equation in \eqref{Q eqn} leads to a set of possible profiles, one of which corresponds to the profile of stable blow-up. For the $L^2$-critical case this equation is simplified (due to $a$ being zero), however, we still ought to investigate the $L^2$-supercritical case (with nonzero $a$ but asymptotically approaching zero), since the correction in the blow-up rate comes exactly from that. 
We refer to the above equation as the {\it profile equation} and discuss the existence and local uniqueness theory of its solutions.

\subsection{Existence theory for profile solutions} 
Several properties of solutions to \eqref{Q eqn} are established in the following lemmas. We mention that while the statements are similar to the ones in the NLS case (see \cite{YRZ2019}), the calculations differ and often have extra terms and assumptions, compared to the pure power case.  
\begin{lemma}\label{L:1}
Let $s_c=\frac{d}{2}-\frac{1}{\sigma}$. Assume $d> 2$ and $\sigma \geq \frac{1}{2}$. If $Q(\xi)$ is the solution of the equation \eqref{Q eqn}, then
\begin{align}\label{Q identity 1}
\dfrac{\xi^{d-2}}{2} \bigg|\xi Q_{\xi}+ \dfrac{Q}{\sigma}\bigg|^2+ & \dfrac{\xi^d}{2} |Q|^2 \left( \left(\frac{1}{2\sigma+1}(-\Delta)^{-1}|Q|^{2\sigma+1}\right)|Q|^{2\sigma-1} -\frac{1}{\sigma^2 \xi^2}\right) \nonumber\\
&\qquad\qquad\qquad\qquad + \dfrac{2-s_c}{2\sigma+1}\int_0^{\xi} V(Q)\xi^{d-1} 
= (1-s_c)\int_0^{\xi} |Q_{\xi}|^2\xi^{d-1}, 
\end{align}
and
\begin{align}\label{Q identity 2}
2\Im(\xi Q_{\xi}\bar{Q})+2(d-2)\Im \int_0^{\xi}Q_{s}\bar{Q}+2a\left( \dfrac{1}{\sigma}-1 \right)\int_0^{\xi}s|Q|^2+a|\xi|^2|Q|^2=0,
\end{align}
where 
\begin{align}\label{E: V(Q)}
V(Q)\defeq ((-\Delta)^{-1}|Q|^{2\sigma+1})|Q|^{2\sigma+1}.
\end{align}
\end{lemma}
\begin{proof}
Multiply (\ref{Q eqn}) by $\Delta \bar{Q} \xi^{d-1}$, take the imaginary part and integrate from $0$ to $\xi$. This gives 
\begin{align}\label{Q identity 10}
a \Re \int_0^{\xi} \left(\xi Q_{\xi}+\dfrac{Q}{\sigma}\right) \Delta \bar{Q} \xi^{d-1}+ \Im  \int_0^{\xi} ((-\Delta)^{-1}|Q|^{2\sigma+1})|Q|^{2\sigma-1}Q \Delta \bar{Q} \xi^{d-1} =0.
\end{align} 
The first part is equivalent to
\begin{align}\label{Q identity 11}
a\left( \xi^{d-2}\Re\left(\xi \bar{Q}_{\xi} \dfrac{Q}{\sigma}\right) +\dfrac{\xi^d}{2}|Q_{\xi}|^2 +\left(\frac{d}{2}-\frac{1}{\sigma}-1 \right) \int_0^{\xi} |Q_{\xi}|^2 \xi^{d-1}   \right),
\end{align}
and, by using \eqref{Q eqn} to express $\Delta \bar{Q}$, the second part of \eqref{Q identity 10} yields
\begin{align}\label{Q identity 12}
a\left( \dfrac{\xi^d}{2(2\sigma+1)} V(Q)-\frac{1}{2\sigma+1} \left( \frac{d}{2}-\frac{1}{\sigma}-2 \right) \int_0^{\xi} V(Q) \, \xi^{d-1} \right).
\end{align}
Putting together these two parts, gives the identity (\ref{Q identity 1}).

The second identity \eqref{Q identity 2} is obtained by multiplying $2\xi \bar{Q}$, integrating from $0$ to $\xi$, and then taking the imaginary part.
\end{proof}

\begin{lemma}\label{Q boundedness}
Suppose $Q(\xi)$ is the $C^2[0,\infty)$ solution of the equation \eqref{Q eqn} for $d>2$ and $\sigma \geq \frac12$. If $0<s_c \leq 2$\footnote{The reason for the restriction $s_c<2$ is to keep the third term in \eqref{Q identity 1} positive.}, then $|Q(\xi)|$ and $|Q_{\xi}(\xi)|$ are bounded.
\end{lemma}
\begin{proof}
Since the equation \eqref{Q eqn} has two derivatives, both $Q(\xi)$ and $Q_{\xi}(\xi)$ are continuous. Thus, both $Q(\xi)$ and $Q_{\xi}(\xi)$ are bounded in the interval $\xi \in [0,M]$ for any $M>0$. Then, it suffices to consider the case when $\xi \rightarrow \infty$.

For that, from (\ref{Q identity 1}), we claim that if $|Q_{\xi}|$ is bounded, so is $|Q|$. To the contrary, suppose that $|Q_{\xi}|$ is bounded but $|Q|$ is not bounded as $\xi \rightarrow \infty$. We consider two cases: $s_c \geq 1$ and $s_c<1$.


For $s_c \geq 1$, the RHS of \eqref{Q identity 1} is not positive, while the LHS of  \eqref{Q identity 1} is strictly positive for sufficiently large $\xi$. We reach a contradiction immediately. 

Now we consider $s_c<1$. By dropping the first and third terms (which are positive) in \eqref{Q identity 1}, for sufficiently large $\xi$, we have
\begin{align}
0\leq \dfrac{\xi^d}{2} |Q|^2 \left( \dfrac{1}{2\sigma+1}((-\Delta)^{-1}|Q|^{2\sigma+1})|Q|^{2\sigma-1} -\dfrac{1}{\sigma^2 \xi^2}\right) \leq \textrm{LHS of (\ref{Q identity 1})} = \textrm{RHS of (\ref{Q identity 1})} \leq c \, \xi^d. 
\end{align}
This implies that $|Q|$ must be bounded, contrary to our assumption.

Now we show that $|Q_{\xi}|$ is bounded when $\xi \rightarrow \infty$. We prove the boundedness of $|Q_{\xi}|$ by contradiction based on the argument from \cite{BCR1999} and \cite{YRZ2019}. Suppose $|Q_{\xi}|$ is not bounded, i.e., $\limsup_{\xi\rightarrow \infty}Q(\xi)=\infty$. Then, there exists a monotonically increasing sequence $\lbrace \xi_j\rbrace_0^{\infty}$ for both $\xi_j$ and $Q(\xi_j)$ such that $|Q_{\xi}(\xi_j)|\rightarrow \infty$ as $\xi_j \rightarrow \infty$, and $|Q_{\xi}(\xi_j)|>|Q_{\xi}(\xi_k)|$ for $j>k$. For any large fixed number $M>0$, there exists an index $j$ such that $\xi_j>M$. 

Now we again split the cases for $s_c \geq 1$ and $s_c<1$. When $s_c \geq 1$, the RHS of \eqref{Q identity 1} is nonpositive, while the LHS of \eqref{Q identity 1} will be strictly positive for $\xi_j$ with $j$ sufficiently large, leading to a contradiction.

Now we consider the case $s_c<1$. From the RHS of the identity (\ref{Q identity 1}), we have 
\begin{align}\label{Q bound 1}
(1-s_c)\int_0^{\xi_j} |Q_{\xi}|^2\xi^{d-1} d\xi\leq \dfrac{1-s_c}{d}|Q_{\xi}(\xi_j)|^2 \xi_j^d.
\end{align}
We choose $\delta > 0$ such that $0<\delta<1-\frac{2-2s_c}{d}$, or equivalently, $\frac{1-\delta}{2}>\frac{1-s_c}{d}$. Then  
\begin{align}\label{Q bound 2}
\dfrac{1-\delta}{2} |Q_{\xi}(\xi_j)|^2 \xi_j^d < \textrm{LHS of (\ref{Q identity 1})} \leq  \dfrac{1-s_c}{d}|Q_{\xi}(\xi_j)|^2 \xi_j^d.
\end{align}
We reach the contradiction in \eqref{Q bound 2} for $\xi_j$ sufficiently large, since $\frac{1-\delta}{2}>\frac{1-s_c}{d}$. Therefore, we conclude that $|Q_{\xi}|$ is bounded and so is $|Q|$.
\end{proof}

We next discuss the existence theory for
\eqref{Q eqn}.
\begin{theorem}[Existence of $Q$]
Define $s_c=\frac{d}{2}-\frac{1}{\sigma}$. If $0< s_c \leq 2$, $d>2$ and $\sigma\geq \frac{1}{2}$, for any given initial value $Q(0) \in \mathbb{R}$ and constant $a>0$, the equation (\ref{Q eqn}) has a unique solution in ${C}^2[0,\infty)$.
\end{theorem}
\begin{proof}
The problem is equivalent to the Volterra integral equation:
\begin{align}\label{Q int}
Q(\xi)& = Q(0)-ia\int_0^{\xi} sQ(s)ds +\dfrac{1}{d-2} \times \nonumber \\
& \int_0^{\xi} \left[ 1+ia\left(d-\frac{1}{\sigma}\right)-((-\Delta)^{-1}|Q(s)|^{2\sigma+1})|Q(s)|^{2\sigma-1}\right]Q(s)\left(s-\frac{s^{d-1}}{\xi^{d-2}} \right)ds, ~ Q(\infty)=0.
\end{align}
The equation (\ref{Q int}) is of the form 
\begin{align}\label{Q form}
Q(\xi)=Q(0)+\int_0^{\xi} g(s,\xi,Q(s))ds, \quad Q(\infty)=0.
\end{align}
From the theory of Volterra integral equation (see exact statements in \cite{YRZ2019} as well as the application in the NLS case, which is following \cite[Theorem 3.2.2]{Bu1983}), the equation \eqref{Q form} has a unique solution on the interval $\xi \in [0,M]$ for some fixed $M>0$, since $g(s,\xi,Q(s))$ is continuous. This result can be extended to $M=\infty$, since $|Q(\xi)|$ is bounded (see \cite[Theorem 2.1]{YRZ2019} and \cite[Theorem 3.3.6]{Bu1983}). From Lemma \ref{Q boundedness}, it follows that the integral equation \eqref{Q int} has a unique solution. We next note that $Q$ is the solution not only to the equation \eqref{Q int} or \eqref{Q form}, but also to the differential equation \eqref{Q eqn}, and thus, differentiating $Q$ twice classically, it gives $Q \in \mathbb{C}^2[0,\infty)$, finishing the proof.
\end{proof}

\begin{remark}
If $s_c=0$ (thus, $a=0$) the equation in \eqref{Q int} reduces to 
$$
Q(\xi) = Q(0) +\frac1{d-2} \int_0^{\xi} \left[ 1-((-\Delta)^{-1}|Q(s)|^{2\sigma+1})|Q(s)|^{2\sigma-1}\right]Q(s)\left(s-\frac{s^{d-1}}{\xi^{d-2}} \right)ds,
$$ 
and given the initial value of $Q(0)$, the uniqueness holds from a similar argument due to Volterra integral theory. The values of $Q(0)$ are unknown {\it a priori} in the $L^2$-critical case, nevertheless, our numerical solver converges to the same $Q$ regardless of initial condition, see Appendix. 
\end{remark}

\begin{corollary}\label{C: Q decay}
For $d > 2$ and $s_c>0$, if $\sigma=1$, then $|Q(\xi)| \lesssim \xi^{-1}$ for $\xi$ large enough (recall that $\xi$ is radial variable here, and thus, nonnegative).
\end{corollary}
\begin{proof}
When $\sigma=1$, the term $2a\left(\frac{1}{\sigma}-1 \right) \int_0^{\xi}s|Q|^2ds$ in \eqref{Q identity 2} cancels. Then, the rest of the proof is the same as in \cite[Theorem 2.2]{BCR1999} and \cite[Corollary 2.7]{YRZ2019}.
\end{proof}
\begin{remark}
For other values of $\sigma$, one would obtain the decay of $Q$ as $|\xi|^{-1/\sigma}$, which can be proved in various ways: as in the NLS (see \cite[Theorem 2.2]{YRZ2019}), or by examining the asymptotic (large distance) behavior as in \cite[Section 3.1]{LePSS1988}, which we will do in the next subsection. 
\end{remark}
\begin{remark}
The reason for the lower bound $s_c>0$ is indeed necessary, since the equation \eqref{Q eqn} does not have ``admissible" solutions as we prove below in Proposition \ref{$a=0$}. 
\end{remark}

\subsection{Asymptotic behavior of the $L^2$-supercritical profile}
We further investigate the large distance behavior of profile solutions following \cite[Prop.7.1]{SS1999}. 
\begin{proposition} \label{P: Q1 Q2}
As $\xi \to \infty$ solutions of \eqref{Q eqn} behave asymptotically as  $Q=\alpha Q_1+\beta Q_2$, where
\begin{equation}\label{E:Q1-Q2}
Q_1(\xi) \approx |\xi|^{-\frac{i}{a}-\frac{1}{\sigma}}, \qquad 
Q_2(\xi) \approx e^{-\frac{ia\xi^2}{2}}|\xi|^{-\frac{i}{a}-d+\frac{1}{\sigma}}, \quad \alpha, \beta \in \mathbb C.
\end{equation}
\end{proposition}
\begin{proof}
Substituting $Q(\xi) = e^{-ia\xi^2/4} \xi^{(1-d)/2} \, Z(\xi)$ into \eqref{Q eqn}, we obtain
\begin{align*}
-Z''  +  \left(-\frac{a^2}4 \xi^2 +1-ia \, s_c  +\frac{(d-1)(d-3)}{4 \xi^2} 
  - ((-\Delta)^{-1}|\xi|^{\frac{(1-d)}2(2\sigma+1)}|Z|^{2\sigma+1})|\xi|^{\frac{(1-d)}2(2\sigma-1)}|Z|^{2\sigma-1}\right)Z = 0.
\end{align*}
Writing $Z(\xi) = e^{w(\xi)}$, yields
\begin{equation} \label{E:w}
w''+(w')^2 + \frac{a^2}4 \xi^2 -1 +ia \, s_c -\frac{(d-1)(d-3)}{4} {\xi^{-2}} = O(e^{2\sigma w} |\xi|^{-2\sigma(d-1)+2}).
\end{equation}
Now, for $s_c \geq \frac12$, we can drop
the nonlinear term $((-\Delta)^{-1}|Q(s)|^{2\sigma+1})|Q(s)|^{2\sigma-1}Q$ to compute the asymptotics, which gives two linear independent solutions 
\begin{align*}
w_1 &\sim i a \frac{\xi^2}4 - \frac{i}{a} \ln |\xi| - \bigg(\frac12-s_c\bigg) \, \ln |\xi|,\\
w_2 &\sim - i a \frac{\xi^2}4 + \frac{i}{a} \ln |\xi| + \bigg(\frac12-s_c\bigg) \, \ln |\xi|.
\end{align*}
Returning back to the notation of $Q$, we get \eqref{E:Q1-Q2}.

We note that if $s_c<\frac12$, then the term with $\xi^{-2}$ in \eqref{E:w} is not dominant compared with the right side. 
 For conciseness we only consider $s_c \geq \frac12$, it is also the setting we use in our numerical study below.
\end{proof}

We note that $Q_2$ is the fast oscillating solution as $\xi \to \infty$, which we should exclude from $Q$ (or require that $\beta = 0$), since we are interested in complex-valued solutions $Q$, which have monotonically decreasing amplitude $|Q|$, of the form $\alpha Q_1$. Such solutions are typically referred to as ``admissible solutions". More importantly, excluding the span of $Q_2$ gives us solutions with finite Hamiltonian.

\begin{proposition}\label{P: finite energy}
If $Q$ is a solution of \eqref{Q eqn} with $Q_{\xi} \in L^2(\mathbb{R}^d)$ and $Q\in L^{\frac{2d(2\sigma+1)}{(d+2)}}(\mathbb{R}^d)$, and $s_c \neq 0$, its Hamiltonian is a non-zero constant, i.e.,
\begin{align}\label{E: Q Hamiltonian}
\int \left( |Q_{\xi}|^2 -\frac{1}{2\sigma+1} V(Q) \right) \xi^{d-1} d\xi=const,
\end{align}
where $V(Q)$ is defined in \eqref{E: V(Q)}. 

Equivalently, taking $P=e^{ia\xi^2/4}Q$ yields
\begin{align}\label{E: P Hamiltonian}
\int \left( |P_{\xi}|^2 -\frac{1}{2\sigma+1} V(Q) +a \Im(\xi P \bar{P}_{\xi}) +\frac{a^2\xi^2}{4}|P|^2 \right) \xi^{d-1} d\xi=const.
\end{align}
\end{proposition}

\begin{proof}
As in the proof of Lemma \eqref{L:1}, multiply \eqref{Q eqn} by $\Delta Q$ and apply 
\eqref{Q identity 11}, \eqref{Q identity 12} with $\xi \rightarrow \infty$. Note that $Q \in L^{\frac{2d(2\sigma+1)}{(d+2)}}(\mathbb{R}^d)$, since $Q(\xi) \sim \xi^{-1/\sigma}$, and the 
Hardy-Littlewood inequality implies $V(Q) \in L^1_{rad}(\mathbb{R}^d)$. Since 
\begin{align}\label{E: H(Q)00}
a\left[ \left( \frac{d}{2}-\frac{1}{\sigma}-1 \right) \left( \int \Big( |Q_{\xi}|^2 -\frac{1}{2\sigma+1} V(Q) \Big) \xi^{d-1} d\xi \right)- \frac{1}{2\sigma+1} \int V(Q) \xi^{d-1}d\xi    \right]=0
\end{align}
with the last term being a constant, we obtain that the first term is also a constant, provided $s_c \neq 1$, completing the proof. 
\end{proof}

\begin{remark}\label{Nonzero-EQ}
From the identity \eqref{E: H(Q)00}, we notice that the energy of $Q$ is not necessarily zero when $0<s_c<1$. Our numerical calculations show that, for example, in the 3d gHartree case with $\sigma=1$ ($p=3$) we get $E[Q] \approx 0.96$. This is different from the NLS case. However, we will show that this does not affect obtaining the \textit{log-log} blow-up rate in the $L^2$-critical case (see Section 4).
\end{remark}

\subsubsection{Admissible solutions to \eqref{RgHartree}}
To discuss what happens with admissible solutions in the case when $s_c = 0$ (more precisely, $s_c \searrow 0$), we allow  flexibility by letting the dimension $d$ vary continuously (as in \cite{SS1999}) so that the equation \eqref{Q eqn} becomes slightly $L^2$-supercritical. 
The reason for this flexibility is to investigate existence of solutions to \eqref{Q eqn} when $s_c=0$; in particular, if we stay rigid in this case with the nonlinearity $\sigma=2/d$, then the equation \eqref{Q eqn} does not have reasonable solutions when $a \neq 0$. 

\begin{proposition}\label{$a=0$}
The equation (\ref{Q eqn}) with $\sigma = \frac2{d}$ ($s_c=0$) has no admissible solutions when $a \neq 0$ and $a$ is finite.
\end{proposition}
\begin{proof}
We split $Q$ into the real amplitude and phase by writing $Q=We^{i\theta}$. The equation (\ref{Q eqn}) produces the following system for functions $W(\xi)$ and $\theta(\xi)$:
\begin{align}
&\Delta W -W + \left((-\Delta)^{-1}|W|^{2 \sigma +1} \right)|W|^{2 \sigma -1} W - \theta_{\xi}(a\xi+\theta_{\xi}) =0, \label{W} \\
&\dfrac{\partial}{\partial \xi}\left( \xi^{\frac{2}{\sigma}-1}W^2(\theta_{\xi}+\frac{a}{2} \xi)  \right)+\dfrac{\sigma d-2}{\sigma}\xi^{2/ \sigma-2}\theta_{\xi}W^2=0. \label{theta}
\end{align}
Note that the nonlinearity only shows up in the first equation, while the existence of admissible solutions comes from examining the second equation, where the second term vanishes when $\sigma=2/d$, giving $\theta(\xi)= -a \xi^2/4$. Now using the large distance behavior from \eqref{E:Q1-Q2}, and giving the same argument as in the NLS case \cite[Section 8.1.1]{SS1999}, the conclusion that there are no admissible solutions for $a\neq 0$ in the $L^2$-critical case $\sigma=\frac{2}{d}$ (or when dimension $d=\frac2{\sigma}$) follows. 
\end{proof}

\begin{remark}
This seems to be the feature for any $L^2$-critical NLS-type equation with {\it any} nonlinear term (as long as $\sigma  = \frac2{d}$). 
\end{remark}

\begin{remark}
If $a=0$, then \eqref{Q eqn} becomes $\Delta Q - Q + ((-\Delta)^{-1}|Q|^{2\sigma+1})|Q|^{2\sigma-1}Q = 0$, which is exactly \eqref{GS}. Thus, the solutions of \eqref{RgHartree} convergence in some sense to ground state solutions of \eqref{RgHartree}. 
\end{remark}

We are now ready to investigate the behavior of blow-up solutions, and in particular, behavior of the parameter $a(\tau)$. We start with the description of the dynamic rescaling method needed for the gHartree equation.

\section{The dynamic rescaling method} \label{Section-DR}

The dynamic rescaling method, which was first introduced in \cite{MPSS1986} in 1986, has proven to be an efficient way to simulate the blow-up phenomena for the NLS equation. Since the generalized Hartree has scaling symmetry, we apply a similar approach and study \eqref{Q eqn}, in particular, 
we recall the parameter $L(t)$ from \eqref{a form}.
We note that the proper choice for representing $L(t)$ will provide the global existence of the rescaled equation (\ref{RgHartree}) on $\tau$. 
Recall that the blow-up rate is defined, for example, as $\|\nabla u(t)\|_2 \sim f(t)$ for some function $f(t)$. Direct calculation by the chain rule from (\ref{rescaled-variables}) shows
\begin{align}\label{blowup rate L}
\|\nabla u(t)\|_2 =\dfrac{1}{L(t)^{\frac{1}{2}(1-s_c)}}\|\nabla v(\tau)\|_2,
\end{align}
and thus, the behavior of $L(t)$ describes the rate of the blow-up. As we discussed in \cite{YRZ2018}, one intuitive choice for $L(t)$ is to restrict the norm $\|\nabla v\|_2$ to be constant in time, i.e.,
\begin{equation*}
L(t)=\left(\frac{\|\nabla v_0\|^2_2}{\|\nabla u(t) \|_2^2}\right)^{\beta}.
\end{equation*}
The direct calculation leads to
\begin{equation*}
\beta=\frac{1}{2/\sigma+2-d},
\end{equation*}
and
\begin{equation}\label{a1}
a(\tau)=-\frac{2 \beta}{\|\nabla v_0\|_{2}^2} \mathrm{Im} \left( \int_0^{\infty} \left((-\Delta)^{-1}|v|^{2\sigma+1} \right)|v|^{2\sigma-1} \bar{v} \Delta v \xi^{d-1} d \xi \right).
\end{equation}

An alternative choice for $L(t)$ (from Definition \ref{D: blowup-rate2}) is to restrict the $L^{\infty}$ norm of the solution to the rescaled equation $v(\tau)$ to be constant, say $\|v(\tau)\|_{L^{\infty}}=1$ (as we mentioned in the introdution, the blow-up rate in the $L^\infty$ norm is equivalent to the blow-up in the $\dot{H}^1$ norm, see also \cite{Ca2003},\cite{FGW2007}). By setting
\begin{align}
L(t)=\left(\frac1{ \|u(t)\|_{\infty}}\right)^{\sigma},
\end{align}
one has
\begin{align}\label{a inf}
a(\tau)=-\sigma \, \Im (\bar{v}\Delta v)(0,\tau).
\end{align}

In this work we fix $\|v(\tau)\|_{\infty}\equiv 1$ instead of $\| \nabla v \|_2$, since computing the last norm involves the integral $\int_0^{\infty} \cdots \xi^{d-1} d\xi$ for $d>1$ (we will consider $d=3,4,5,6,7$); when the dimension $d$ becomes higher, say $d=7$, the values of the term $\int_0^{\infty} \cdots \xi^{d-1} d\xi$ in (\ref{a1}) become very large. 
If we fix $\|v\|_{L^{\infty}}$, then there will be no influence on $a(\tau)$ from the dimension $d$. Actually, both options lead to the same results in the cases $d=3$ and $d=4$ (lower dimensions). For the $L^2$-supercritical case, which we consider in Section 4, we choose to fix the value $\|v(\tau)\|_{\infty}$ to be constant; this is in part because when $s_c>1$, Definition \ref{D: blowup-rate} has to be replaced with the blow-up rate defined with respect to the $\dot{H}^s$ norm for $s>s_c$, i.e.,
$$\lim_{t\nearrow T} \dfrac{\|u(t)\|_{\dot{H}^s}}{f(t)}=C$$
for some function $f(t)$. 

We return to the equation (\ref{RgHartree}), which is of the form
\begin{align} \label{DRNLSF}
iv_{\tau} +\Delta v +\mathcal{N}(v)=0,
\end{align}
where $\mathcal{N}(v)=ia(\tau)\left( \xi v_{\xi} +\frac{v}{\sigma} \right)+  \left((-\Delta)^{-1}|v|^{2\sigma +1} \right)|v|^{2\sigma-1} v$.

The equation (\ref{DRNLSF}) is of the same form as the one we studied in \cite{YRZ2018} and \cite{YRZ2019}. It is given on the whole space $\xi \in [0,\infty)$, and for numerical purposes,  we ought to map the spatial domain $[0, \infty)$ onto some finite interval, for example, onto $[-1,1)$. For that, we choose the mapping from \cite{MPSS1986} by setting $\xi = l\, \frac{1+z}{1-z}$. Here, $l$ is a constant indicating the half number of the collocation points assigned on the interval $[0,l]$ and $z$ is the Chebyshev-Gauss-Lobatto collocation points from $[-1,1]$ (see \cite{STL2011}). We impose the homogeneous Dirichlet boundary condition, $v(\infty)=0$, on the right, and thus, we remove the last Chebyshev point, and, consequently, delete the last row and the last column of the matrix $\mathbf{M}$ in \eqref{E: AB solve} { below}. 
The Laplacian operator can be discretized from the Chebyshev-Gauss-Lobatto differentiation matrix (refer to \cite{STL2011} and \cite{Tr2000} for details). We denote the discretized Laplacian with $N+1$ collocation points by the matrix $\Delta_N$.
The non-local operator $(-\Delta)^{-1}$ can now be approximated by the matrix $(-\Delta_N)^{-1}$, which is the inverse of the matrix $-\Delta_N$ with the first row replaced by the first row of the Chebyshev differential matrix because of the Neumann homogeneous boundary condition, $\varphi_{\xi}(0)=0$, for the equation $ -\Delta \varphi =|v|^{2\sigma +1}$ for the non-local term. This also avoids the singularity of the Laplacian at $\xi=0$. The matrix $(-\Delta)^{-1}_N$ needs to be calculated only once by numerically taking the inverse of the matrix $-\Delta_N$ and then storing it to be used later to calculate the time evolution.


To discuss the time evolution, we use the following notation for $v$ as the semi-discretization in time variable $\tau$: let $v^{(m)}{\approx v(\xi, m\cdot\Delta \tau)}$ {be the approximation of $v$ at the time $m\cdot\Delta \tau$}, where $\Delta \tau$ is the time step and $m$ is the number of iterations. The time evolution of \eqref{DRNLSF} can be approximated by the second order Crank-Nicolson-Adam-Bashforth method, i.e.,
\begin{align}\label{AB}
 i\frac{v^{(m+1)}-v^{(m)}}{\Delta \tau}+\frac{1}{2}\left({\Delta} v^{(m+1)}+{\Delta} v^{(m)} \right)+\frac{1}{2}\left( 3\mathcal{N}(v^{(m)})-\mathcal{N}(v^{(m-1)})\right)=0.
\end{align}
We rewrite (\ref{AB}) as
\begin{align}\label{AB_simplified}
\left(\frac{i}{\Delta \tau}+\frac{1}{2} {\Delta}\right)v^{(m+1)}=\left( \frac{i}{\Delta \tau}-\frac{1}{2} {\Delta} \right)v^{(m)}- \frac{1}{2}\left( 3\mathcal{N}(v^{(m)})-\mathcal{N}(v^{(m-1)})\right).
\end{align}
With the Laplacian operator $\Delta$ replaced by the matrix {$\Delta_N$}, and also the term ${\frac{i}{\Delta \tau}}$ replaced by the diagonal matrix $\mathbf{diag}({\frac{i}{\Delta \tau})}$, the equation \eqref{AB_simplified} is equivalent to the following linear system:
\begin{align}\label{E: AB solve}
\mathbf{M}v^{(m+1)}=\mathbf{F}(v^{(m)},v^{(m-1)}).
\end{align}
Therefore, each time step is updated by
\begin{align*}
v^{(m+1)}=\mathbf{M}^{-1}\mathbf{F}(v^{(m)},v^{(m-1)}).
\end{align*}
Again, the inverse of the matrix $\mathbf{M}$ can be calculated and stored only once in the beginning, since $\mathbf{M}=\left(\mathbf{diag}(\frac{i}{\Delta\tau})+\frac{1}{2} \Delta_N \right)$ stays the same. 

The boundary conditions are imposed similar to \cite{MPSS1986}, \cite{STL2011}, \cite{Tr2000} and \cite{YRZ2018} as follows: For the homogeneous Neumann boundary condition on the left, $v(0) = 0$, we substitute the first row of the matrix $\mathbf{M}$ by the first row of the first order Chebyshev differential matrix, and change the first element of the vector $\mathbf{F}$ to 0. Because of the homogeneous Dirichlet boundary condition $v(\infty) = 0$ on the right, we delete the last row and column of $\mathbf{M}$ as well as the last element of the vector $\mathbf{F}$.

This discretization gives us spectral accuracy in space. Figure \ref{Coe} shows that the coefficients reach the machines accuracy ($10^{-16}$) within 200 grid points. To utilize the Fast Fourier Transform ({FFT}) efficiently, we use $N=256$ grid points (instead of $200$).
\begin{figure}[ht]
\includegraphics[width=0.45\textwidth]{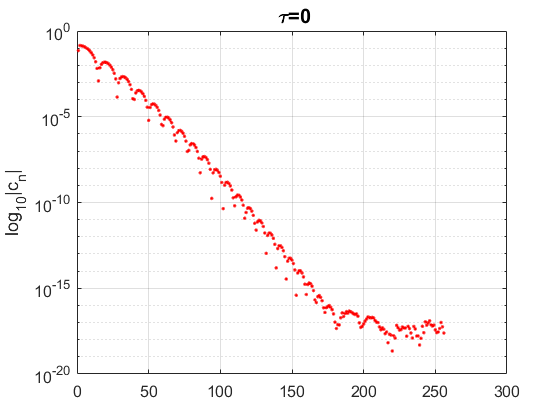}
\includegraphics[width=0.45\textwidth]{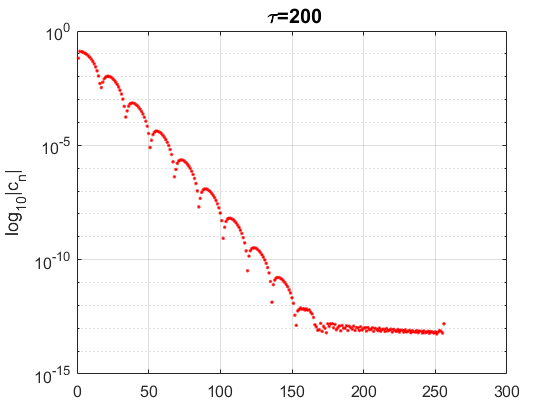}
\caption{The 4d case: the Chebyshev coefficients for the solution $v(\xi,\tau)$ at $\tau=0$ (left) and $\tau=200$ (right).}
\label{Coe}
\end{figure}

After each $v^{(m+1)}$ is obtained, the terms $a^{(m+1)}$ and $\ln L(\tau_{m+1})$ can be updated by the {trapezoidal} rule: 
\begin{align}\label{Hartree ln L}
\ln L(\tau_{m+1})=\ln L(\tau_m) + \frac{{\Delta \tau}}{2}(a^{(m+1)}+a^{(m)}).
\end{align}

To determine the blow-up rate, we track the quantity $T-t$ (and then compute $\ln(T-t)$) in a similar way as we did in \cite{YRZ2018} and \cite{YRZ2019}. The right-hand side of \eqref{Hartree ln L} produces from the $m$th step the value $\ln L(\tau_{m+1})$ on the left-hand side; exponentiating it, we get $\exp(\ln L(\tau_{m+1}))$, which is the value of $L(\tau_{m+1})$. Now, denoting $\Delta t_{m+1}:=t_{m+1}-t_m$, we obtain this difference from the last equation of \eqref{rescaled-variables}
\begin{align}
\Delta t_{m+1}= {\Delta} \tau L^2(\tau_{m+1}).
\end{align}
Hence, starting from $t_0=0$, the mapping for the rescaled time $\tau$ back to the real time $t$ is calculated as 
\begin{align}
t(\tau_{m+1}) = \sum_{j=1}^{m+1} \Delta t_j = {\Delta} \tau \sum_{j=1}^{m+1}  L(\tau_j)^2.
\end{align}
Note that as time evolves, the time difference $T - t(\tau_n)$ will become smaller and smaller, and eventually reach saturation level (with little change), therefore, we treat the stopping time $t(\tau_{\text{end}}) = t(\tau_{M})$ as the blow-up time $T$, where $M$ is the total number of iterations when reaching the stopping condition ($L<10^{-24}$). Then, we can take
\begin{align}
T = t(\tau_{\text{end}})= {\Delta} \tau \sum_{j=1}^M  L(\tau_j)^2.
\end{align}
Consequently, for any $t_i$, we calculate $T-t_i$ as 
\begin{align}\label{T-t}
T-t_i=\sum_{j=i+1}^M \Delta t_j = {\Delta} \tau \sum_{j=i+1}^M  L(\tau_j)^2.
\end{align}
This indicates that instead of recording the cumulative time $t_i$, we only need to record the elapsed time between the two recorded data points, i.e., $\Delta t_{i+1} =t_{i+1}-t_{i}$. By doing so, we avoid the loss of significance when adding a small number onto a large one. 

Since the mapped-Chebyshev collocation method may suffer from the under-resolution issue (when the solution is far away from the origin), we also use the finite difference method with the uniform mesh size on a bounded domain. This involves constructing the artificial boundary conditions to approximate $v(\infty)=0$ as well as the nonlocal term $\left((-\Delta)^{-1}|v|^{2\sigma+1} \right)|v|^{2\sigma-1} v$ at $\xi=\infty$. Similar to the argument in \cite{SS1999}, we know that the terms $\Delta v$ and $\left((-\Delta)^{-1}|v|^{2\sigma+1} \right)|v|^{2\sigma-1} v$ are of the higher order compared with the remaining linear terms in \eqref{RgHartree}. When $\xi \gg 1$, these two terms can be negligible and the equation  \eqref{RgHartree} reduces to 
\begin{align}\label{DRgHartree_L}
v_{\tau}+a(\tau)\left( \alpha v +\xi v_{\xi}\right)=0 ~~\mbox{at} ~~ \xi=K,
\end{align}
where $K$ is our computational domain length taken to be large enough. The equation \eqref{DRgHartree_L} can be solved exactly 
\begin{align}
v(\xi,\tau)= v_0\bigg(\xi \frac{L(\tau)}{L(0)}\bigg) \left( \frac{L(\tau)}{L(0)}\right)^{\frac{1}{\sigma}}.
\end{align} 
This suggests that at $\xi=K$, we have
\begin{align}
v(K,\tau_{m+1})=v\bigg(K \frac{L(\tau_{m+1})}{L(\tau_m)}\bigg) \left( \frac{L(\tau_{m+1})}{L(\tau_m)}\right)^{\frac{1}{\sigma}}.
\end{align}
The $L(\tau_{m+1})$ can be approximated by the second order central difference 
$$
L(\tau_{m+1})=L(\tau_{m-1})+2{\Delta} \tau L_{\tau}(\tau_m).
$$
Note that $a^{(m)}=-\frac{L_{\tau}(\tau_m)}{L(\tau_m)}$, and $\frac{L(\tau_{m})}{L(\tau_{m-1})}$ can be approximated by $\frac{L(\tau_{m})}{L(\tau_{m-1})}=e^{-\frac{{\Delta} \tau}{2}(a^{(m-1)}+a^{(m)})} +O({\Delta} \tau^3)$. Therefore, the right side boundary condition is approximated with second order accuracy 
\begin{align}\label{abc_RgHartree}
v(K,\tau_{m+1})=v\left(K (e^{{+}\frac{{\Delta} \tau}{2}(a^{(m-1)}+a^{(m)})}-2{\Delta} \tau a^{(m)})\right) \left( (e^{{+}\frac{{\Delta} \tau}{2}(a^{(m-1)}+a^{(m)})}-2{\Delta} \tau a^{(m)})\right)^{\frac{1}{\sigma}}.
\end{align}
We also discretize the equation \eqref{DRNLSF} by a uniform mesh with the boundary condition \eqref{abc_RgHartree}. Let $v(\xi_j,\tau) \approx v(jh,\tau)$ to be the semi-discretization in space, where $h=\xi_{j+1}-\xi_{j}$ is the { spatial grid} size, the derivatives are approximated by the sixth order central difference:
\begin{align*}
v_{\xi}(jh,\tau) & \approx D^{(1)}_6 v_j  = \frac{1}{60h}[-v_{j-3} + 9 v_{j-2} - 45 v_{j-1} + 45 v_{j+1} - 9 v_{j+2} + v_{j+3}],\\
v_{\xi \xi}(jh,\tau) &\approx D^{(2)}_6 v_j =  \frac{1}{180h^2}[2v_{j-3} - 27 v_{j-2} + 270 v_{j-1} -490 v_j + 270 v_{j+1} - 27 v_{j+2} + 2v_{j+3}], 
\end{align*}
and the Laplacian operator is approximated by
\begin{align}
\Delta v(jh, \tau) \approx \Delta_h v_j = v_{\xi \xi}(jh,\tau)+\dfrac{d-1}{jh} v_{\xi}(jh,\tau).
\end{align}
In fact, we also tested our approach with the second order and fourth order central difference method, and obtained the { consistent} result. The reported results are obtained by the sixth order central difference method.

When the grid points beyond the right side computational domain are needed, we set up the fictitious points obtained by extrapolation 
$$
v_{N+2}=8v_{N+1}-28v_N+56v_{N-1}-70v_{N-2}+56v_{N-3}-28v_{N-4}+8v_{N-5}-v_{N-6}.
$$
For the grid points beyond the left side computational domain, note that $v(\xi)$ is radially symmetric, and thus, we use the fictitious points $v_{-j}=v_{j}$. The singularity at $\xi=0$ in the Laplacian term $\Delta_h$ is eliminated by the L'Hospital's rule
$$
\lim_{\xi \rightarrow 0} \dfrac{d-1}{\xi} v_{\xi} = (d-1)v_{\xi \xi}.
$$

As in \cite{YRZ2018} for the NLS equation, we have an alternative way to approximate the time evolution by introducing a {\it predictor-corrector} scheme (see also \cite{F2015}):
\begin{align}\label{predictor}
i\frac{v^{(m+1)}_{\mathrm{pred}{,j}}-v^{(m)}_{j}}{{\Delta} \tau}+\frac{1}{2}\left({\Delta}_N v^{(m+1)}_{\mathrm{pred}{,j}}+ {\Delta}_N v^{(m)}_{j} \right)+\frac{1}{2}\left( 3\mathcal{N}(v^{(m)}_{j})-\mathcal{N}(v^{(m-1)}_{j})\right)=0, ~~ (\mathrm{P})
\end{align}
\begin{align}\label{corrector}
i\frac{v^{(m+1)}_{j}-v^{(m)}_{j}}{{\Delta} \tau}+\frac{1}{2}\left( {\Delta}_N v^{(m+1)}_{j}+ {\Delta}_N v^{(m)}_{j} \right)+\frac{1}{2}\left( \mathcal{N}(v^{(m+1)}_{\mathrm{pred}{,j}})+\mathcal{N}(v^{(m-1)}_{j})\right)=0. ~~ (\mathrm{C})
\end{align}

Both approaches \eqref{AB} and \eqref{predictor}--\eqref{corrector} lead to similar results. Numerical tests suggest that (\ref{predictor})--(\ref{corrector}) is slightly more accurate than the scheme (\ref{AB}), though it is still a second order scheme in time and it doubles the computational time, therefore, we mainly use the {\it predictor-corrector} scheme (\ref{predictor})--(\ref{corrector}) in our simulation.

We next remark about the term $|v|^{p-2}$. The power $p-2$ may become negative in the $L^2$-critical case when $d\geq 5$ (since $p=1+\frac{4}{d} < 2$). Numerically, this may cause problems as the singularities may occur if $v(\xi_0,\tau)=0$ at certain points $\xi_0$. To avoid this issue, we write $v=|v|e^{i\theta}$, hence, the outside nonlinear term becomes
\begin{align}\label{Nonlinearity}
|v|^{p-2}v = |v|^{p-2} |v| e^{i \theta} = |v|^{p-1} e^{i \theta}.
\end{align}
Note that with (\ref{Nonlinearity}), we can deal with $p \geq 1$, in particular, $p=1+\frac{4}{d}$. Furthermore, we can also use $v$ when it is zero, by defining
\begin{align}
\left((-\Delta)^{-1}|v|^{p} \right)|v|^{p-2} v = \begin{cases}
\left((-\Delta)^{-1}|v|^{p} \right)|v|^{p-2} v & \text{if $|v|>0$} \\
0 & \text{if $|v|=0$,} \end{cases}
\end{align}
since zeros occurring on the term $|v|^{p-2}v$ are of the higher order term compared with the zeros occurring on $v$.

We set the rescaled initial value $\|v_0\|_{\infty}=1$. We choose $N=256$ collocation points, the mapping parameter $l=256 $ and ${\Delta} \tau=2\times10^{-3}$, if we apply the mapped collocation spectral method to work on the entire space. Alternatively, we choose $h=0.1$, $K=120$ and ${\Delta} \tau=10^{-4}$, if we use the finite difference method and apply the artificial boundary condition \eqref{abc_RgHartree}. Again, these two discretizations lead to similar results. Initially, we only have $v^{(0)}=v_0$. The next time step $v^{(1)}$ can be obtained by the standard second order explicit Runge-Kutta method (RK2).

\section{The $L^2$-critical case} \label{Section-Critical}

In this section, we only consider the $L^2$-critical case, i.e., $\sigma=2/d$ and $$i u_t + \Delta u + \left((-\Delta)^{-1}|u|^{1+\frac4{d}} \right)|u|^{\frac4{d}-1} u=0, \quad d \geq 3.$$

\subsection{Preliminary investigation of rates and profile}
\subsubsection{Initial data}
Similar to the NLS equation in \cite{YRZ2018} and \cite{MPSS1986}, we use the Gaussian-type initial data $u_0(r)=Ae^{-r^2}$, which leads to the self-similar blow-up solutions concentrated at the origin. 
As the amplitude $A_0$ can become very large in higher dimensions, we normalize the exponent and work with data $u_0=A_0e^{-\frac{r^2}{d}}$, since the normalization term $\frac{r^2}{d}$ keeps $A_0$ reasonably small. Table \ref{Initial} lists the mass of the ground state $Q$, the mass of $e^{-\frac{r^2}{d}}$, and we also list the threshold of the amplitude $\tilde{A}$ for the finite time blow-up solutions vs. globally existing solutions. The amplitude $A_0$ is one of the examples from our simulation (one could choose any $A_0>\tilde{A}$). We remark that we drop the dimensional constant $\alpha(d)$ in our calculations.
{\small \begin{table}
\begin{tabular}{|c|l|l|l|l|}
\hline
$d$&$A_0$&$\tilde{A}$ (threshold)&$\|e^{-\frac{r^2}{d}}\|_2^2$&$\|Q\|_2^2$\\
\hline
$3$&$4$&$1.9878$&$0.81405$&$3.2167$\\
\hline
$4$&$5$&$2.774$&$2$&$15.3898$\\
\hline
$5$&$6$&$3.7019$&$6.5683$&$90.0122$\\
\hline
$6$&$7$&$4.8631$&$27$&$638.5311$\\
\hline
$7$&$8$&$6.3399$&$133.2859$&$5357.3174$\\
\hline
\end{tabular}
\linebreak
\linebreak
\caption{Various values for the initial condition $u_0=A_0e^{-\frac{r^2}{d}}$ depending on the dimension $d$. Here, $\|Q\|_2^2$ is the mass of the ground state, the value $\tilde{A}$ gives  the threshold for the finite time blow-up vs. globally-existing solutions, $A_0$ is an example of the amplitude used. For reference, the $L^2$-norm of $e^{-\frac{r^2}{d}}$ is also listed. Note that $\tilde{A}^2\cdot \| e^{-\frac{r^2}{d}} \|_2^2 \approx \|Q\|_2^2$. (All of the $L^2$-norms are calculated without the dimensional constant $\alpha(d)$.)}
\label{Initial}
\end{table}
}
	
To demonstrate the precision of our calculations, we check the following quantity $\|v(\tau) \|_{L^{\infty}_{\xi}}$, which is supposed to be conserved in time $\tau$. Table \ref{Consistency} shows how this quantity $\|v\|_{L^{\infty}_{\xi}}$ varies in the rescaled time $\tau$, i.e., $\mathcal{E}=\displaystyle\max_{\tau} (\|v(\tau)\|_{L^{\infty}_{\xi}})-\displaystyle\min_{\tau} (\|v(\tau)\|_{L^{\infty}_{\xi}}),$ which is at least on the order of $10^{-7}$.
\begin{table}[ht]
\begin{tabular}{|c|c|c|c|c|c|}
\hline
$d$&$3$&$4$&$5$&$6$&$7$\\
\hline
$\mathcal{E}$&$8e-7$&$7e-9$&$4e-9$&$2e-9$&$2e-9$\\
\hline
\end{tabular}
\linebreak
\linebreak
\caption{The error for the conserved quantity $\|v(\tau)\|_{\infty}$ in $\tau$ by using the predictor-corrector method with $\delta \tau=2\times 10^{-3}$ with respect to the dimension $d$.}
\label{Consistency}
\end{table}


\subsubsection{Blow-up rate} In this part we investigate behavior of $L(t)$. We plot the slope of $\ln L$ vs. $\ln(T-t)$ and dependence of $a(\tau)$ on $\tau$ in dimensions $d=3,\cdots,7$ in Figures \ref{3dL}--\ref{7dL}. The subplots on the left show that the slope is approximately $\frac{1}{2}$ (for example, the slope of $\ln(T-t)$ vs. $\ln L$ is $0.50171$ in 3d). 
The subplots on the right show a (slow) decay of the coefficient $a(\tau)$, recall this coefficient from \eqref{RgHartree} and \eqref{a form}. Note that $a(\tau)$ decays very slowly in $\tau$, this is similar to the decay of the corresponding $a(\tau)$ in the $L^2$-critical NLS equation in \cite{MPSS1986} and \cite{YRZ2018}. We also plot the dependence of $a(\tau)$ vs. $1/(\ln(\tau)+3 \ln \ln \tau)$, and observe that it fits the straight line very well (see Figure \ref{a fit}), here we are using the same expression in the denominator for the consistency with the NLS computations and fittings (see more discussion on this below). 

Because of the second term in the above fitting (in Figure \ref{a fit}), one might expect that the correction term in the blow-up rate should be the {\it log-log} correction. We may also expect that the self-similar blow-up solution converges to the ground state profile $Q$ (up to scaling) as $a(\tau) \rightarrow 0$ from the slow decay of $a(\tau)  \sim 1/(\ln(\tau)+3 \ln \ln \tau)$. In the next two subsections, we confirm these implications, i.e., the convergence of blow-up profiles to the ground state $Q$, and also provide justifications to the {\it log-log} correction. 
\begin{figure}[ht]
\includegraphics[width=0.4\textwidth]{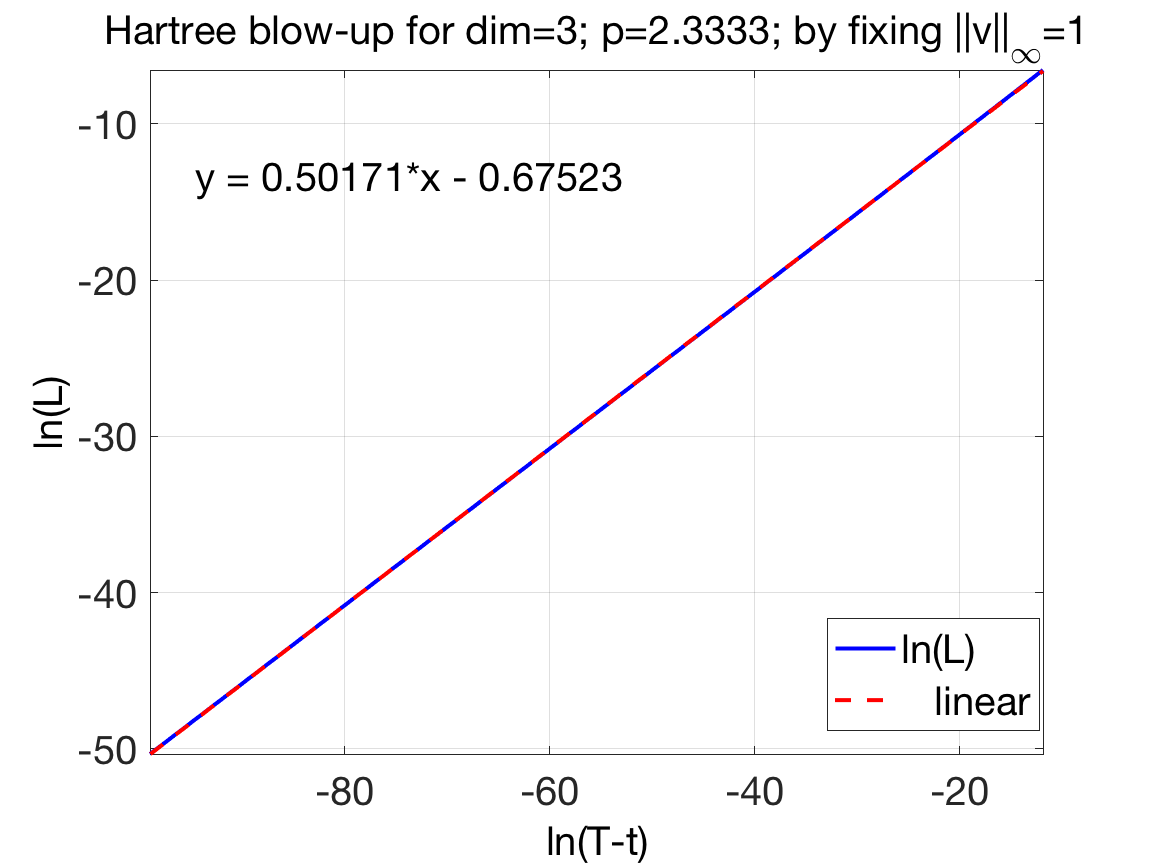}
\includegraphics[width=0.4\textwidth]{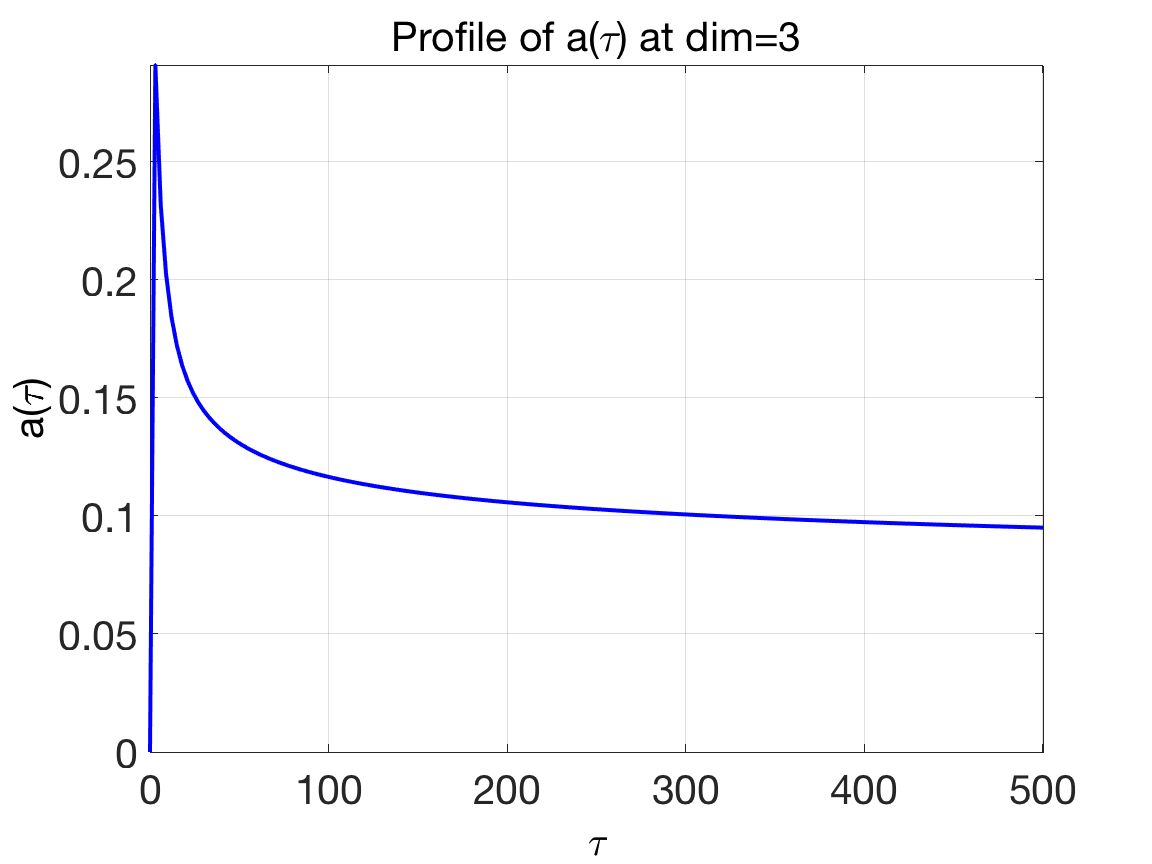}
\caption{3d ($p=\frac73$): the slope of $L(t)$ vs. $T-t$ on a {\it log} scale, which shows the slope of $\frac12$ (left); the behavior of $a(\tau)$ - very slow decay (right).}
\label{3dL}
\end{figure}

\begin{figure}[ht]
\includegraphics[width=0.38\textwidth]{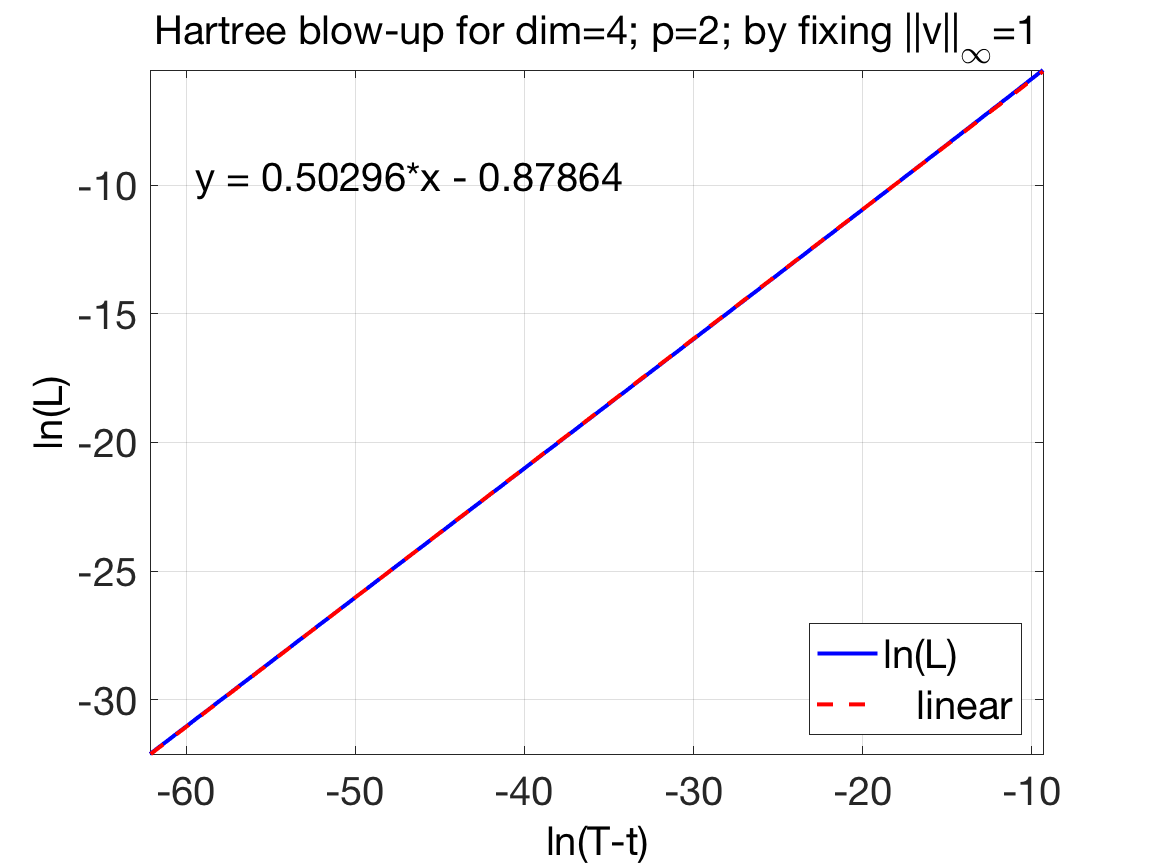}
\includegraphics[width=0.38\textwidth]{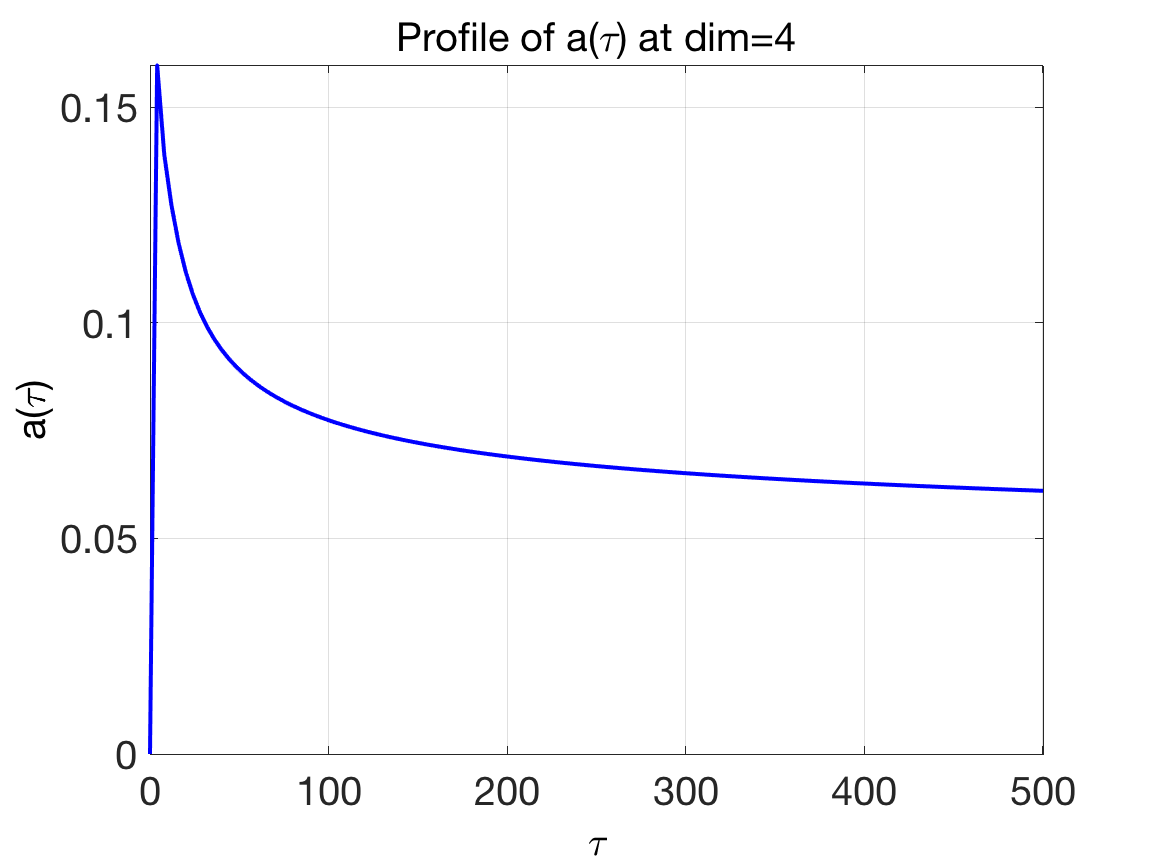}
\includegraphics[width=0.38\textwidth]{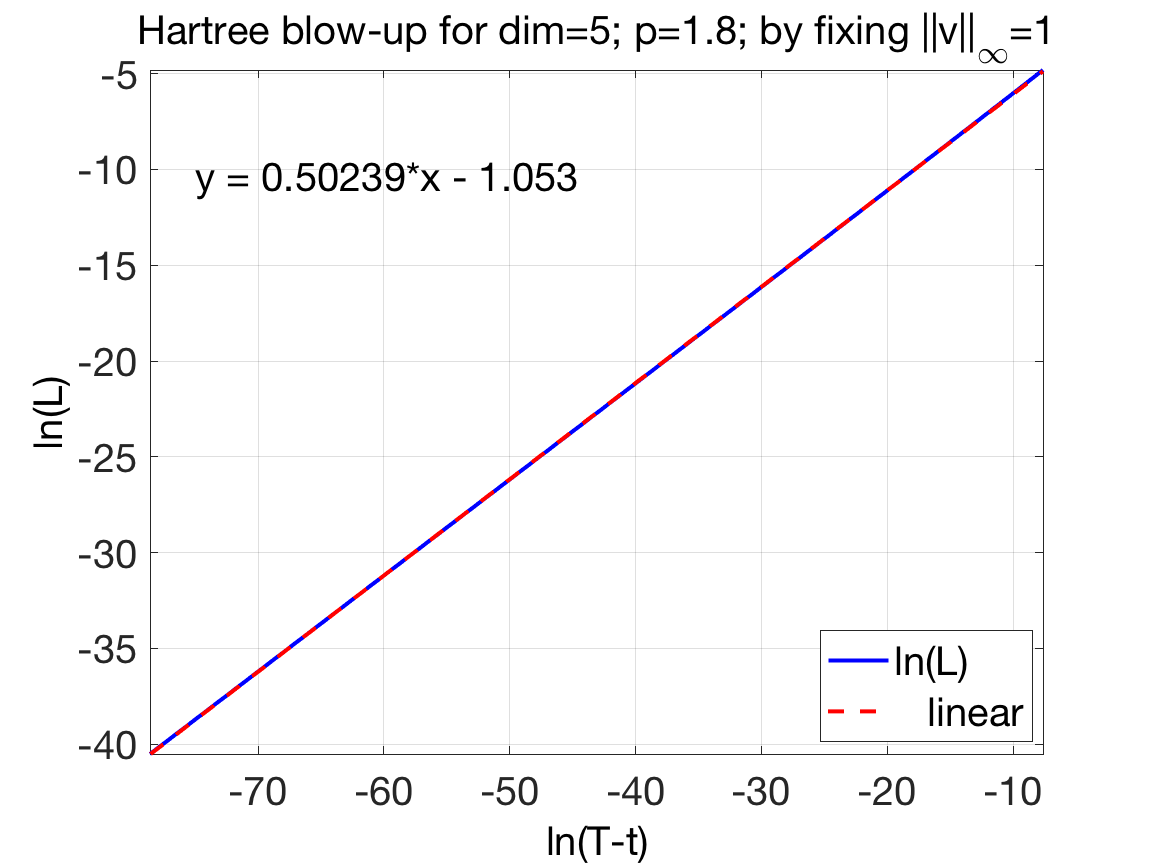}
\includegraphics[width=0.38\textwidth]{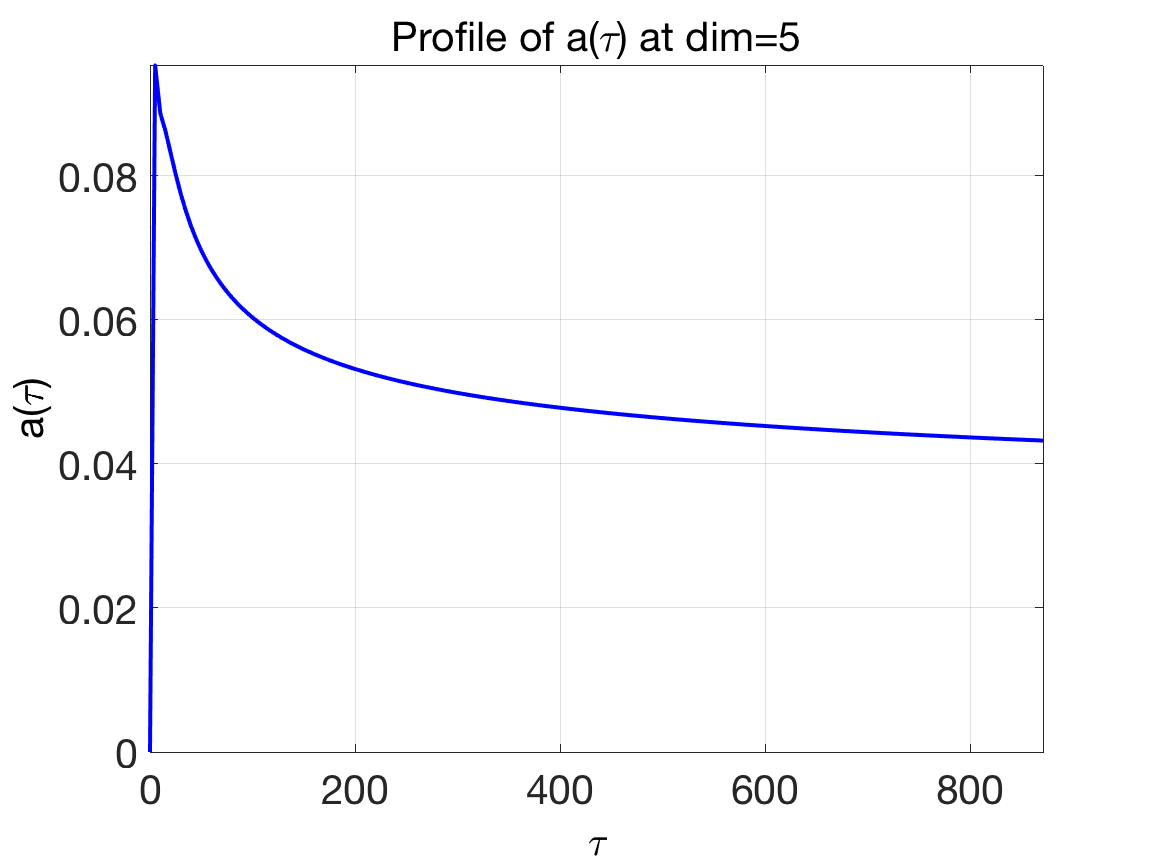}
\includegraphics[width=0.38\textwidth]{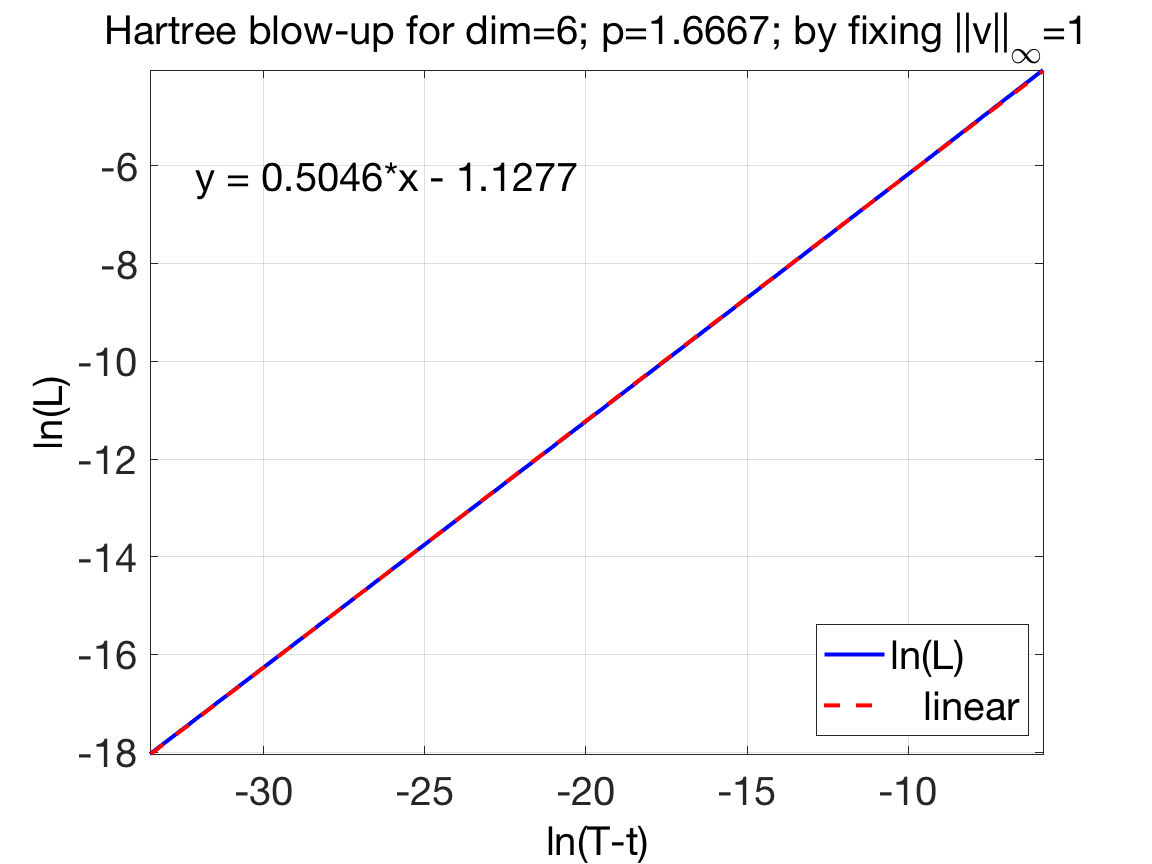}
\includegraphics[width=0.38\textwidth]{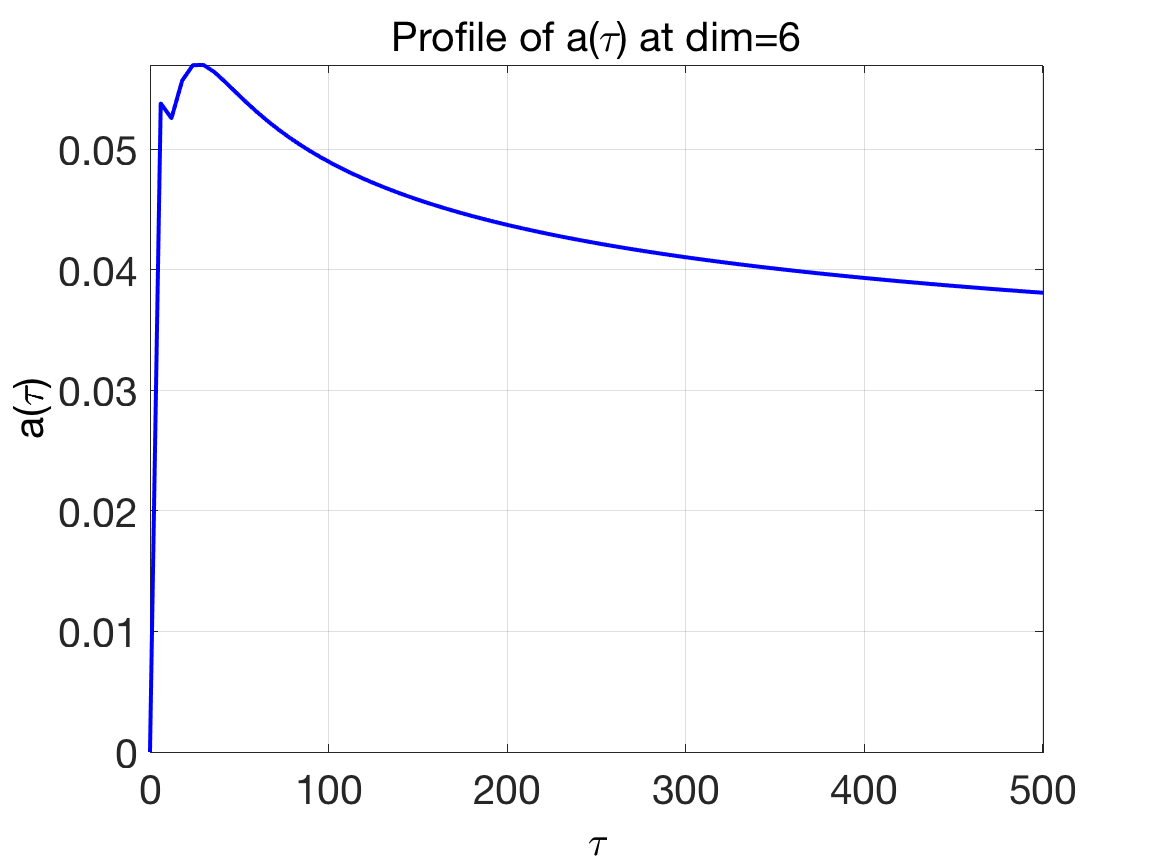}
\includegraphics[width=0.38\textwidth]{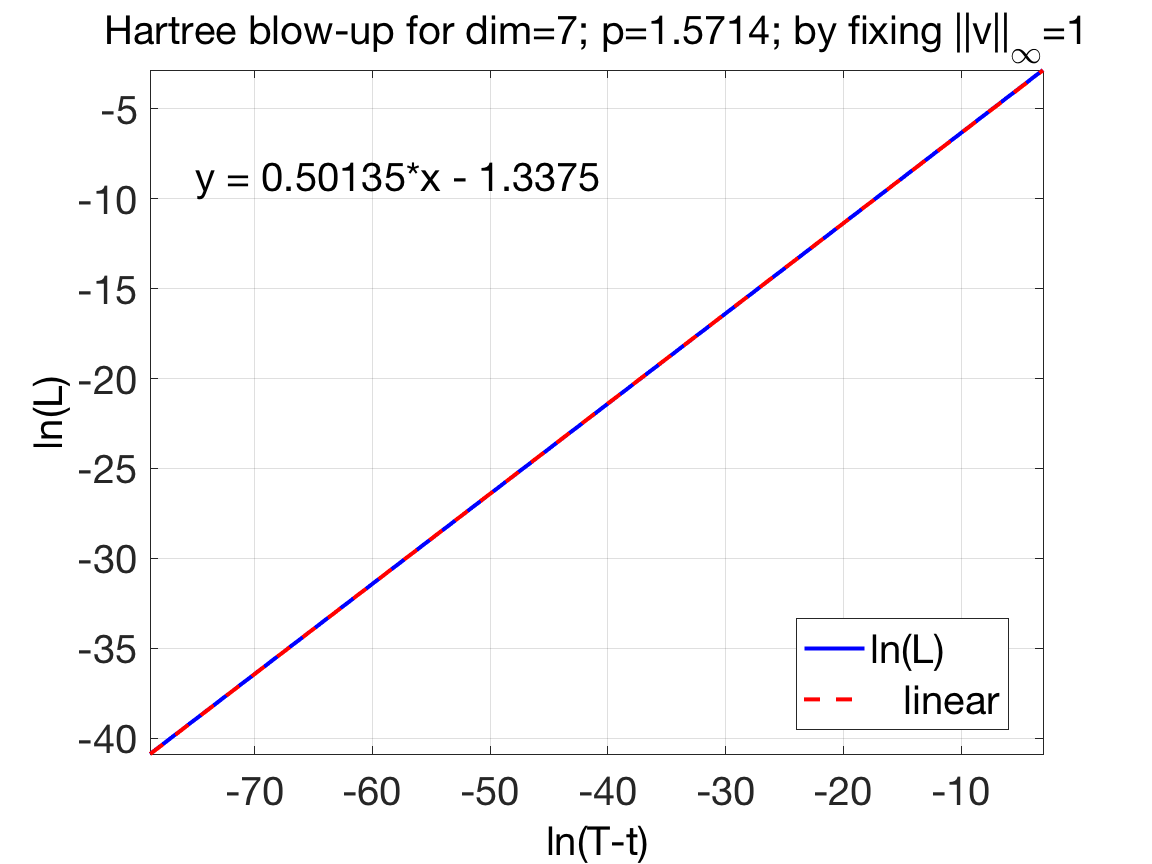}
\includegraphics[width=0.38\textwidth]{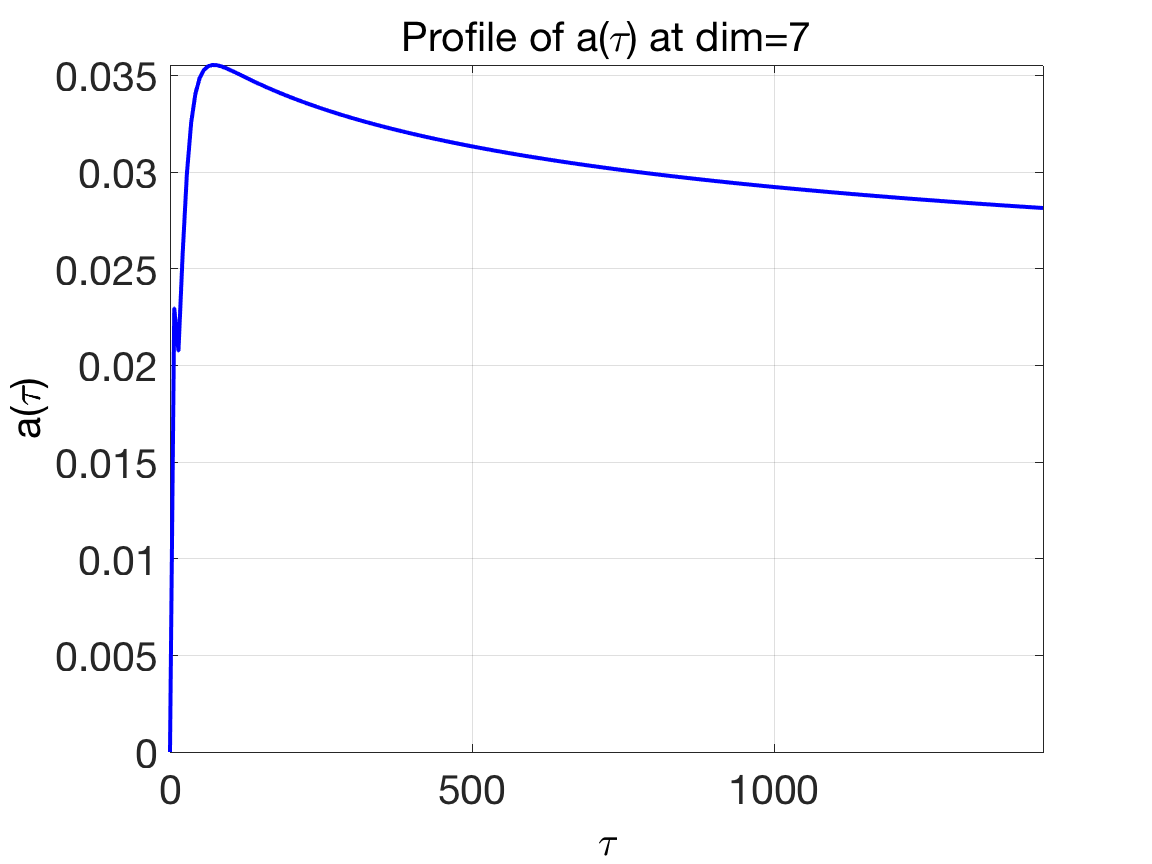}
\caption{The slope of $L(t)$ vs. $T-t$ on a {\it log} scale (left); the slow decay of $a(\tau)$ (right). 
Top: 4d ($p=2$); Top middle: 5d ($p=\frac95$); Bottom middle: 6d ($p=\frac53$); Bottom: 7d ($p=\frac{11}7$). }
\label{7dL}
\end{figure}

\begin{figure}[ht]
\includegraphics[width=0.39\textwidth]{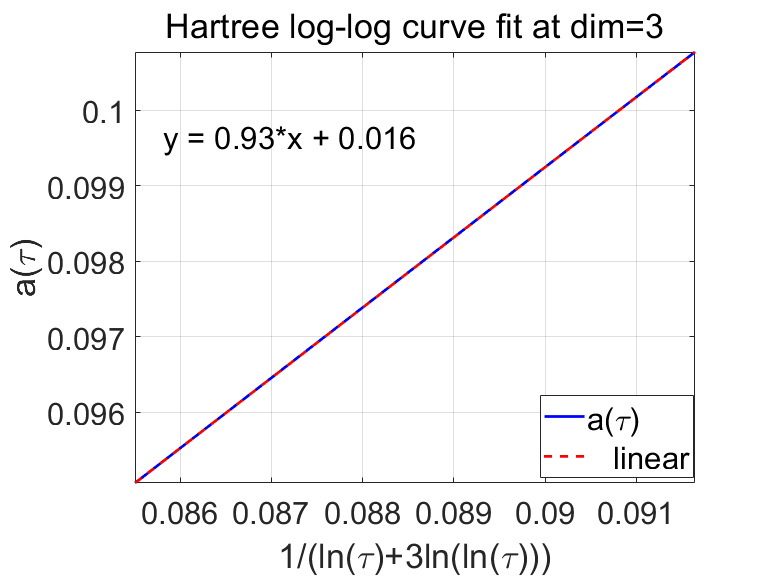}
\includegraphics[width=0.39\textwidth]{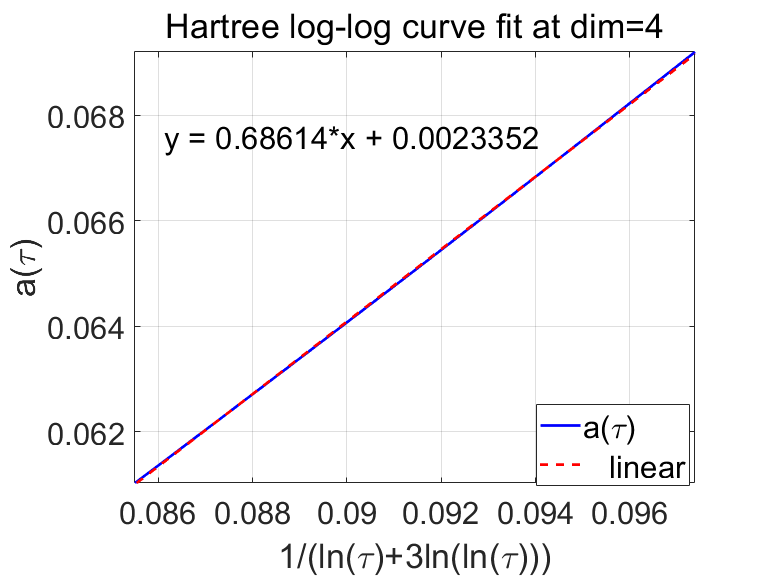}
\includegraphics[width=0.3\textwidth]{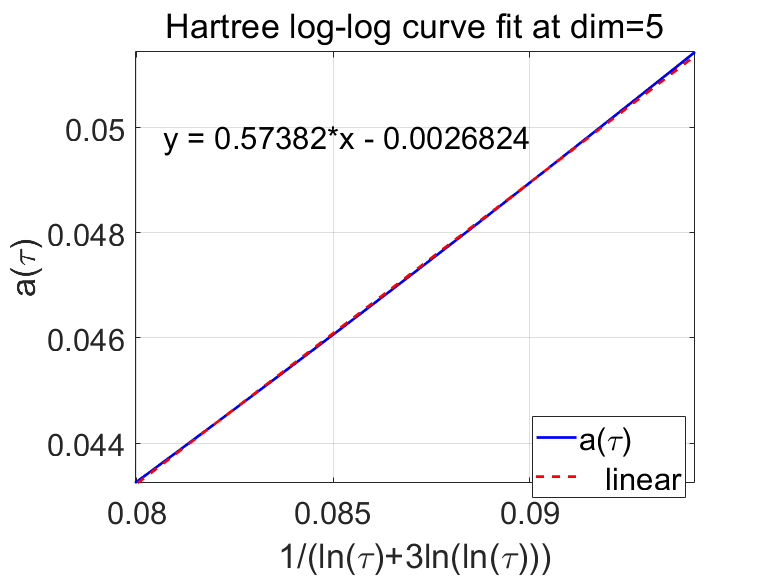}
\includegraphics[width=0.3\textwidth]{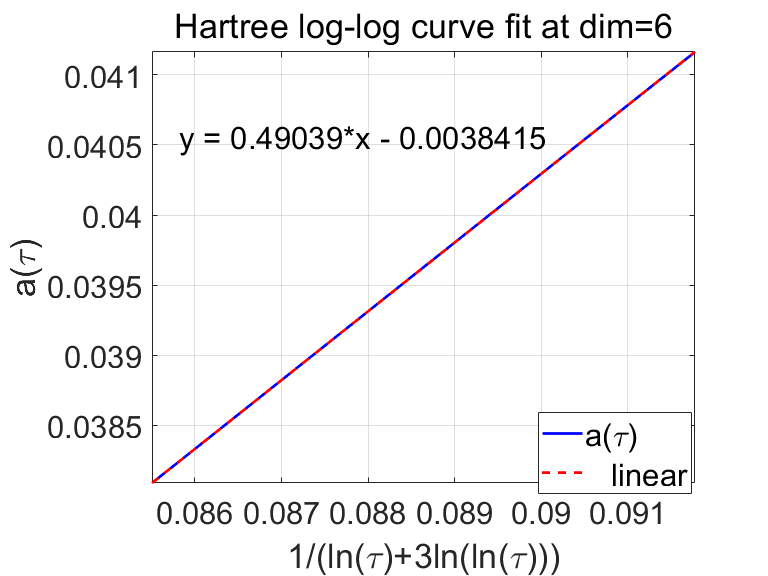}
\includegraphics[width=0.3\textwidth]{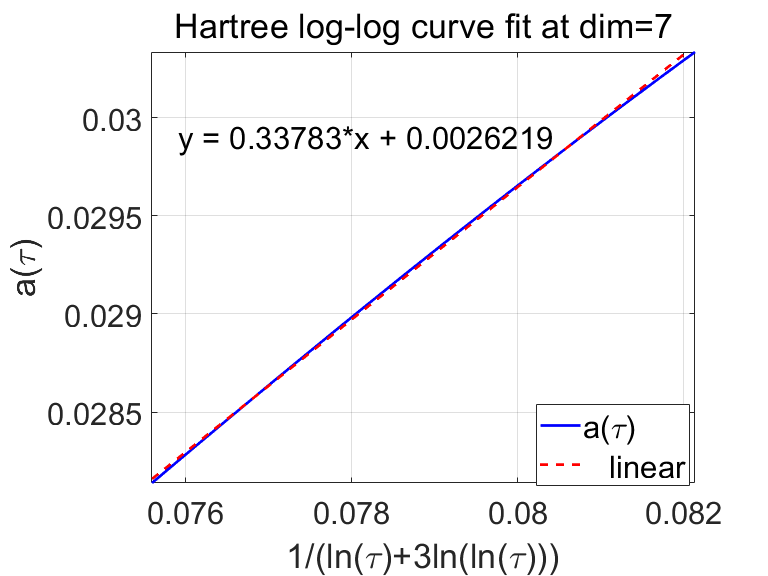}
\caption{Fitting for $a(\tau)$, which suggests $a(\tau) \sim 1/(\ln\tau+3\ln \ln(\tau))$.}
\label{a fit}
\end{figure}

\begin{remark}
In the following subsections, we study the decay rate of $a(\tau)$ via the asymptotic analysis and at the first leading order, as $\tau \rightarrow \infty$, $a(\tau)$ decays at the rate $\frac{\pi}{\ln \tau}$, i.e., slower than any polynomial rate. When more corrective terms are retained, then we have $a(\tau) \sim \frac{\pi}{\ln\tau +c \cdot \ln \ln \tau}$, and from $a_{\tau} \sim -a^{-1}e^{-\pi/a}$ one concludes that $c=3$. This is why we only show the figures $a(\tau)$ vs. $1/\left( \ln \tau +3\ln \ln \tau  \right)$ in this paper. However, we observe that when we change the constant $c$ dependence in the second term of $1/\left( \ln \tau + c\ln \ln \tau  \right)$ with different values of $c$, including zero and large constants (we tried with $c=0,1,3,100,1000$), we find that the slope does not change, which only confirms that such a correction is difficult to track numerically as it only happens at the very high focusing levels. 

We next note that the slope of the line is not $\frac{1}{\sqrt{2\pi}}$ as expected from the asymptotic analysis, where the correction term for $a(\tau)$ is given by $q(t) \sim ( (2\pi)/\ln\ln(\frac{1}{T-t}) )^{1/2}$. This is because at the time we { terminate} our simulations, which we are forced to do as the maximal current numerical precision is reached, the values of $a(\tau)$ are still far from $0$, similar to the computations for the NLS equation in \cite{LPSS1988}, \cite{LePSS1988} and \cite{YRZ2018}. These facts indicate that the {\it log-log} regime occurs only when the amplitude is very large (say $\gg 10^{200}$).
\end{remark}

\subsubsection{Blow-up profiles} 
In this part we investigate the profiles of blow-up solutions at the time $\tau=2, 40, 200, 400$ in dimension $d=3,\cdots,7$. Figures \ref{p3d}--\ref{p7d} show plots of $|v(\xi, \tau)|$ and $|u(r,t)|$ next to each other as we plot in pairs  different times $\tau$ and { the corresponding} $t$, 
recalling that $v$ is the solution to the rescaled equation \eqref{rescaled-variables} 
and $u$ is the solution to the gHartree equation \eqref{gHartree1} reconstructed from  $v$. These figures demonstrate that the blow-up solutions $v=v(\xi,\tau)$ converge to the rescaled ground state profile $Q$ from \eqref{GS}, or \eqref{GS-conv}, in all dimensions we simulated ($d=3,\cdots,7$). (Appendix A explains the computation of $Q$ via the renormalization method.) Table \ref{Converge} shows that $\| \, v(\tau)-Q\|_{\infty} \rightarrow 0$ as $\tau \rightarrow \infty$, but very slowly, which matches our hypothesis about slow decay of $a(\tau)$. This confirms that the blow-up profile $u(x,t)$ converges to $Q$, i.e., $\| \, u(t)-Q\|_{\infty} \rightarrow 0$ as $t\rightarrow T$, up to scaling.
{\small 
\begin{table}[ht]
\begin{tabular}{|c|c|c|c|c|c|c|}
\hline
$d$&$\tau=0$&$\tau=10$&$\tau=50$&$\tau=100$&$\tau=200$&$\tau=400$\\
\hline
$3$&$0.3267$&$0.0325$&$0.0136$&$0.0106$&$0.0086$&$0.0072$\\
\hline
$4$&$0.2402$&$0.03331$&$0.0122$&$0.0089$&$0.0070$&$0.0057$\\
\hline
$5$&$0.1878$&$0.0290$&$0.0120$&$0.0085$&$0.0064$&$0.0051$\\
\hline
$6$&$0.1497$&$0.0218$&$0.0119$&$0.0087$&$0.0065$&$0.0051$\\
\hline
$7$&$0.1188$&$0.0118$&$0.0102$&$0.0082$&$0.0065$&$0.0052$\\
\hline
\end{tabular}
\linebreak
\linebreak
\caption{The values of $\| \, v(\tau)- Q\|_{\infty}$, where $v(\tau)$ is the solution to the rescaled equation \eqref{RgHartree}. 
Note that the values are decreasing as $\tau \rightarrow \infty$, or equivalently, $t \rightarrow T$.}
\label{Converge}
\end{table}
}

\begin{figure}
\includegraphics[width=0.45\textwidth]{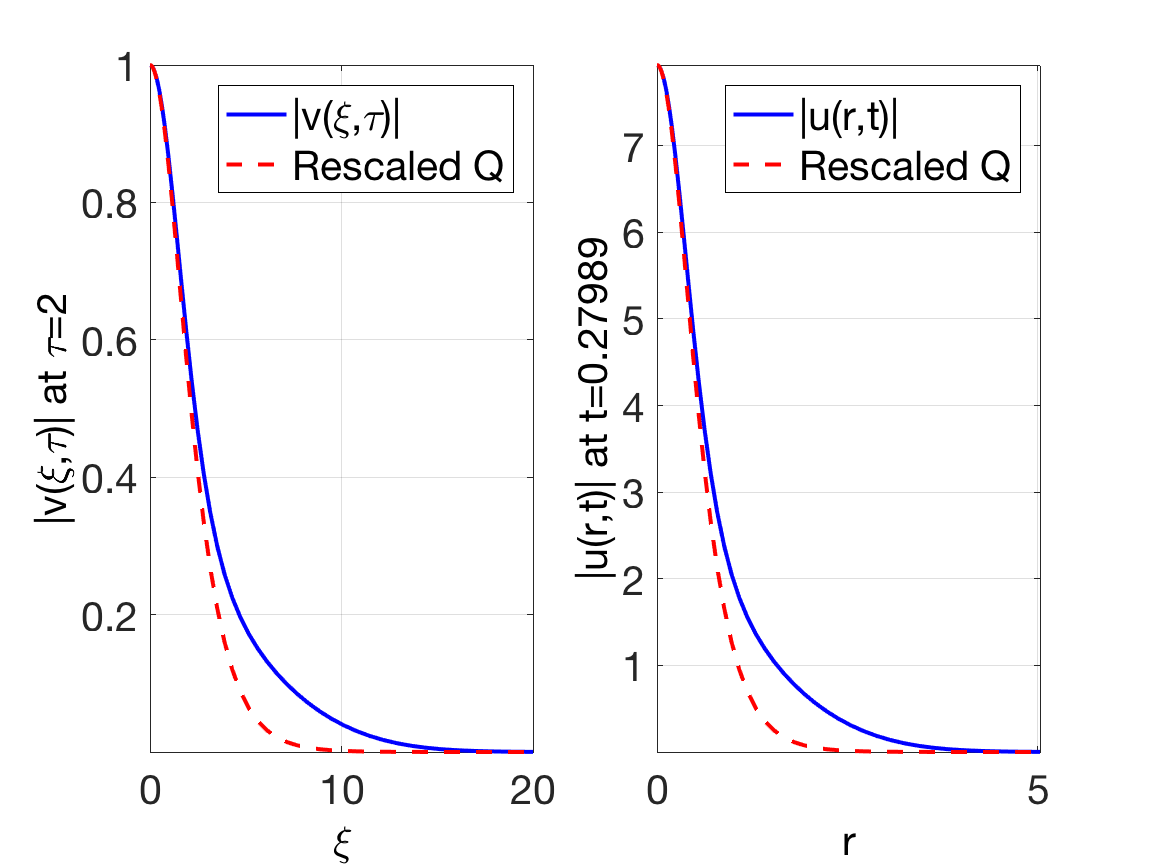}
\includegraphics[width=0.45\textwidth]{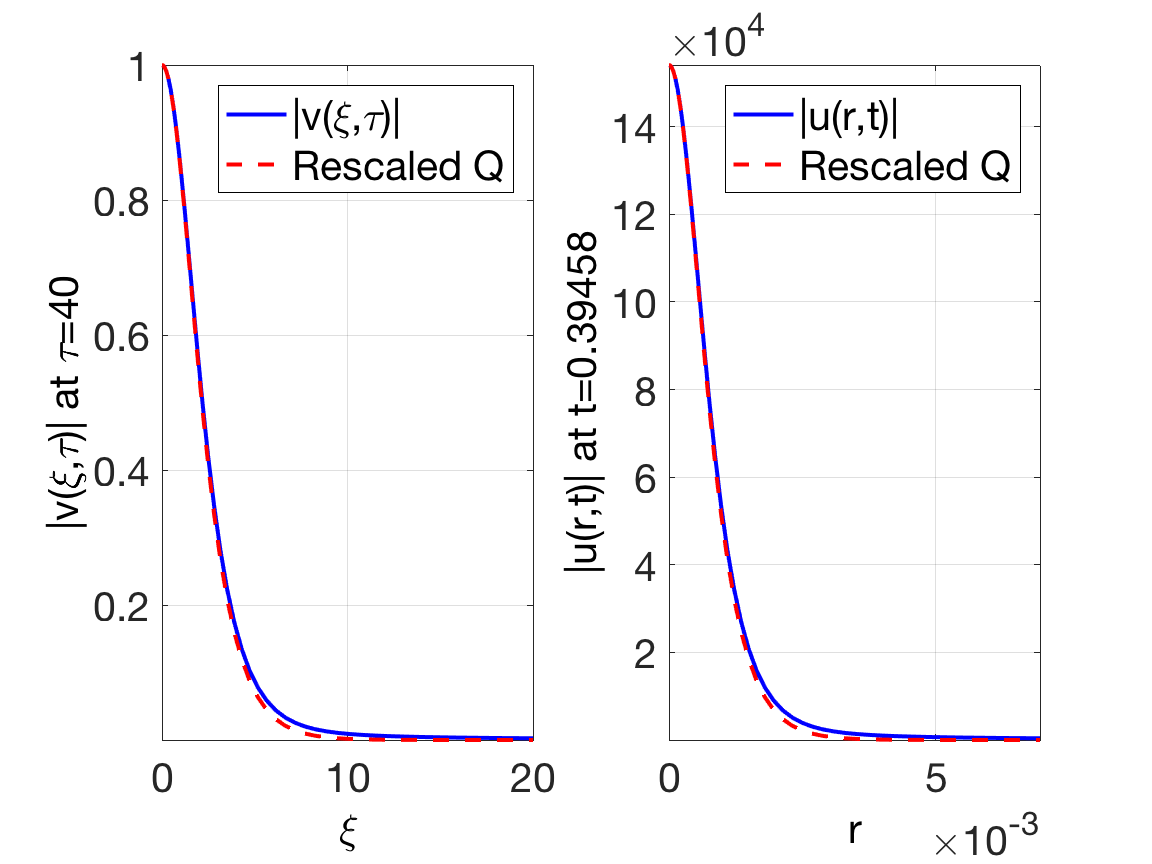}
\includegraphics[width=0.45\textwidth]{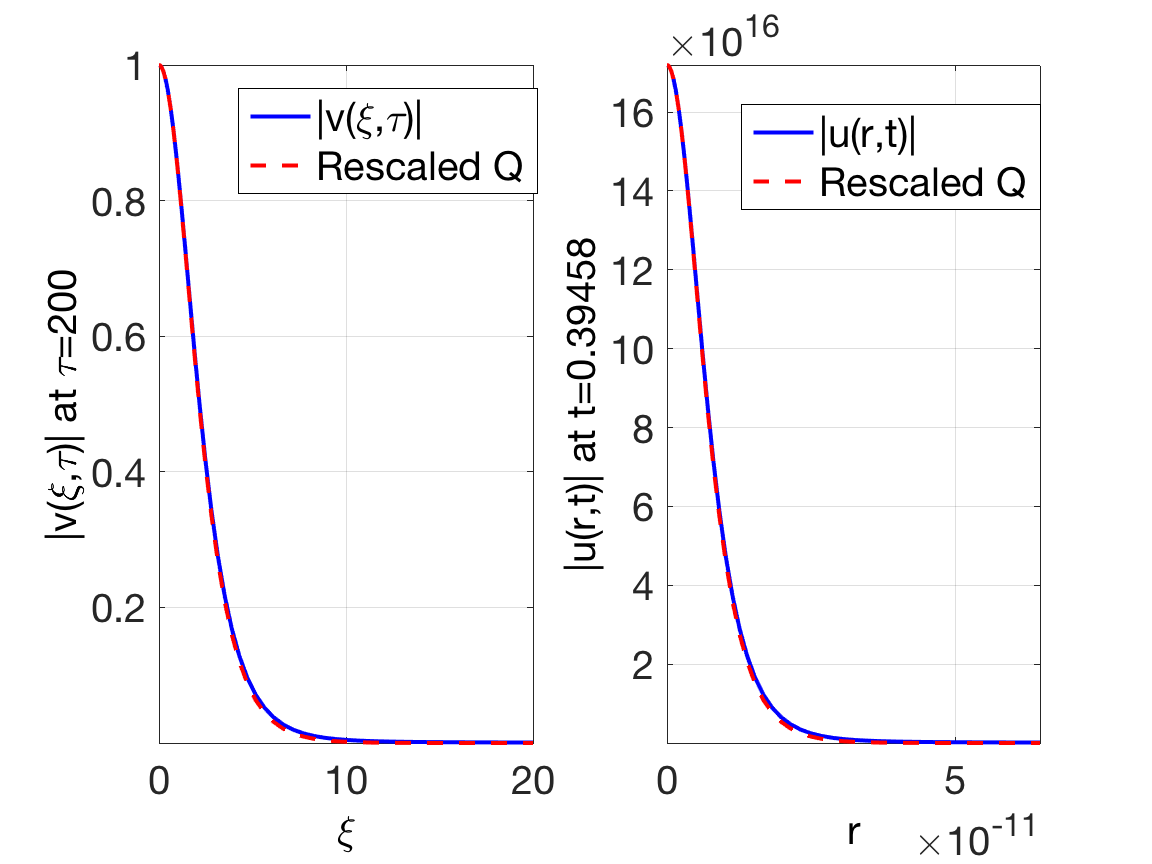}
\includegraphics[width=0.45\textwidth]{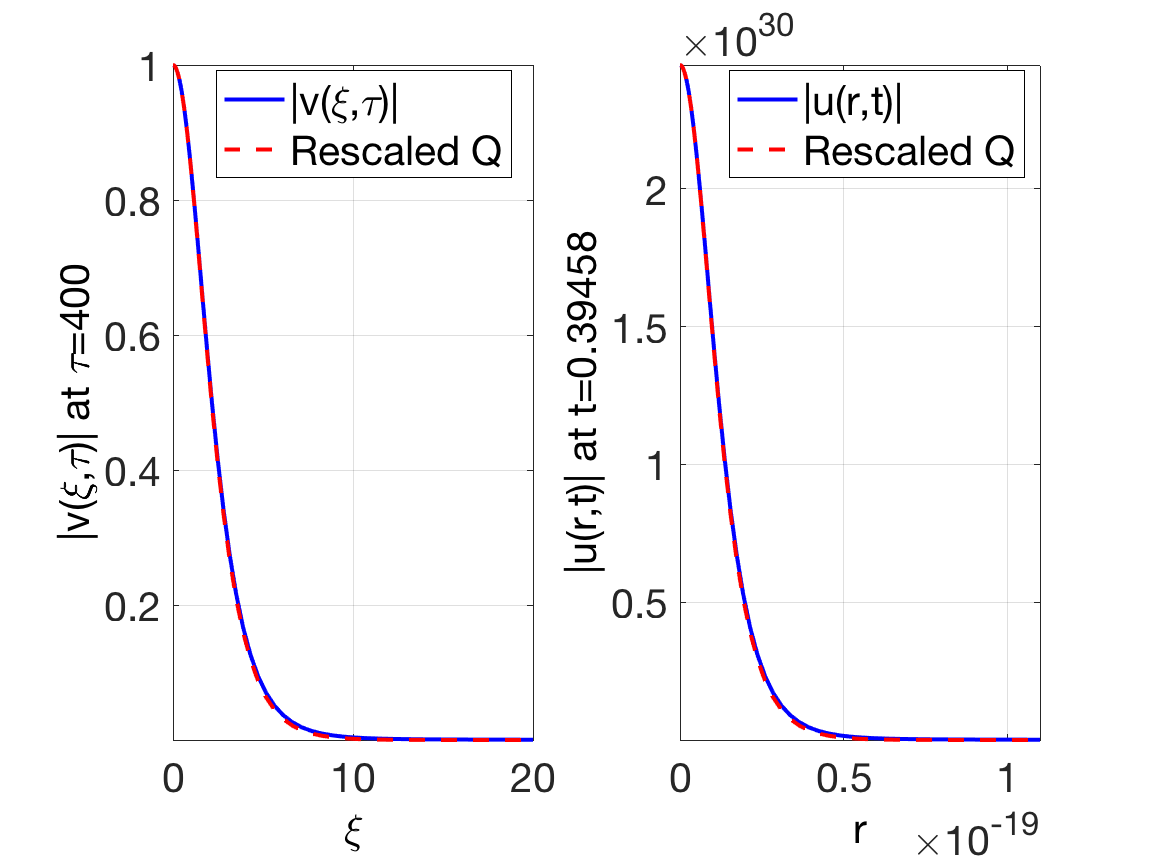}
\caption{Convergence of blow-up profile in the dimension $d=3$. Next to each other we plot the rescaled profile $|v|$ in the rescaled time $\tau$ and the original solution $|u|$ in the actual time $t$. As $\tau \rightarrow \infty$ (here, $\tau=2, 40, 200, 400$), the profile (blue dots) approaches the rescaled ground state $Q$, i.e., $Q_{rescaled}=\frac{1}{L^{\alpha}(t)} Q(\frac{r}{L(t)})$.}
\label{p3d}
\end{figure}

\begin{figure}
\includegraphics[width=0.32\textwidth]{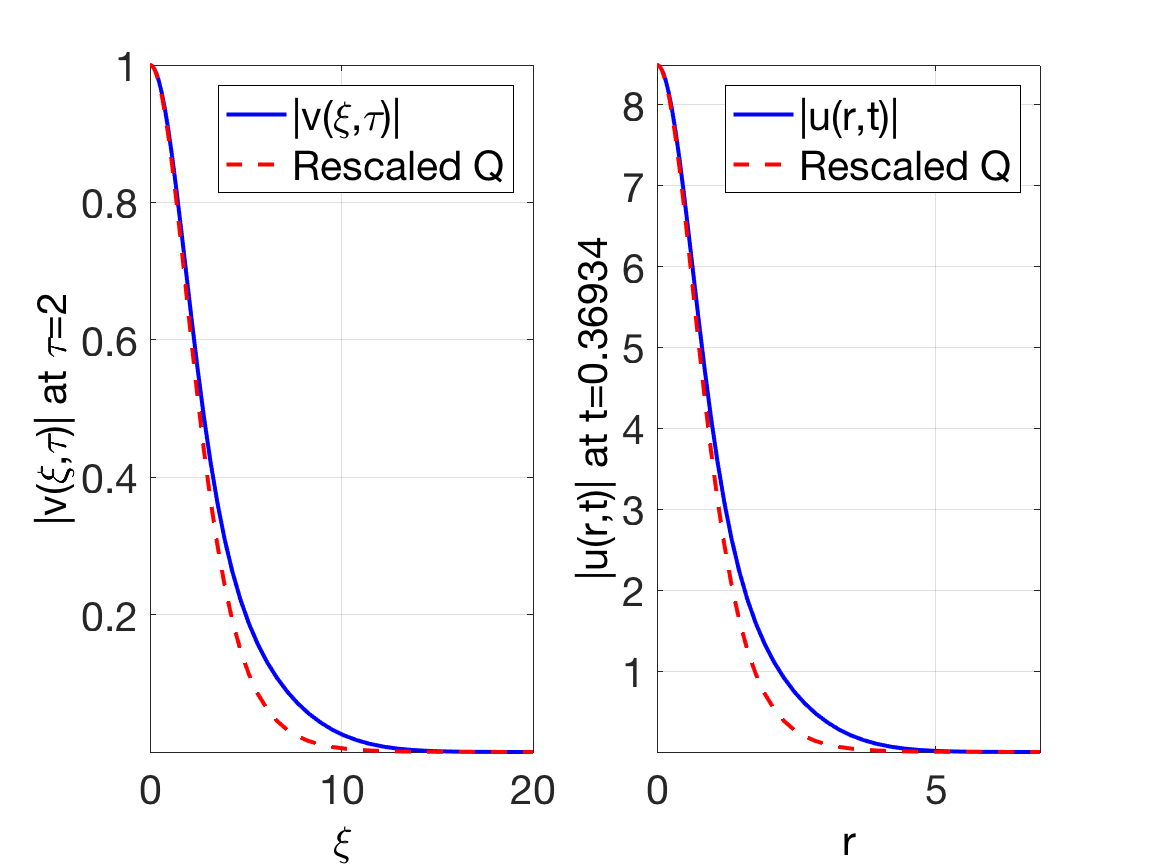}
\includegraphics[width=0.32\textwidth]{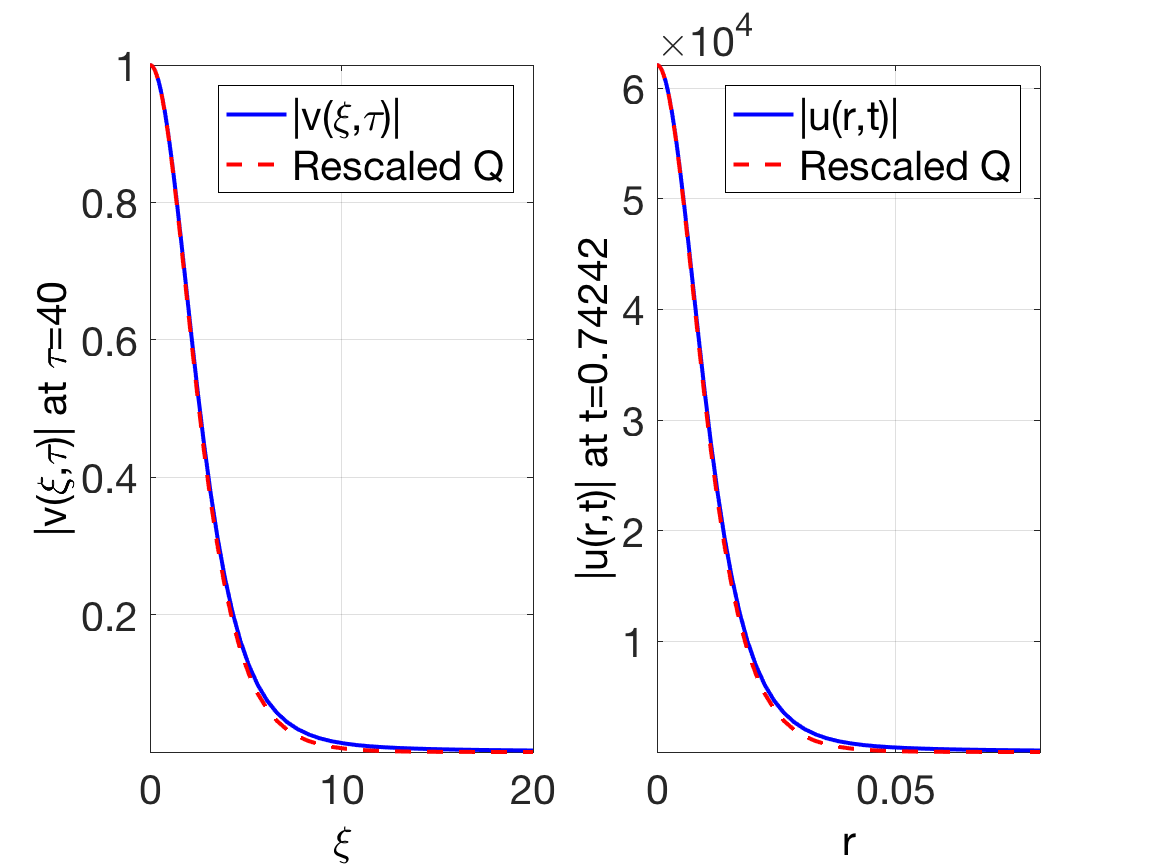}
\includegraphics[width=0.32\textwidth]{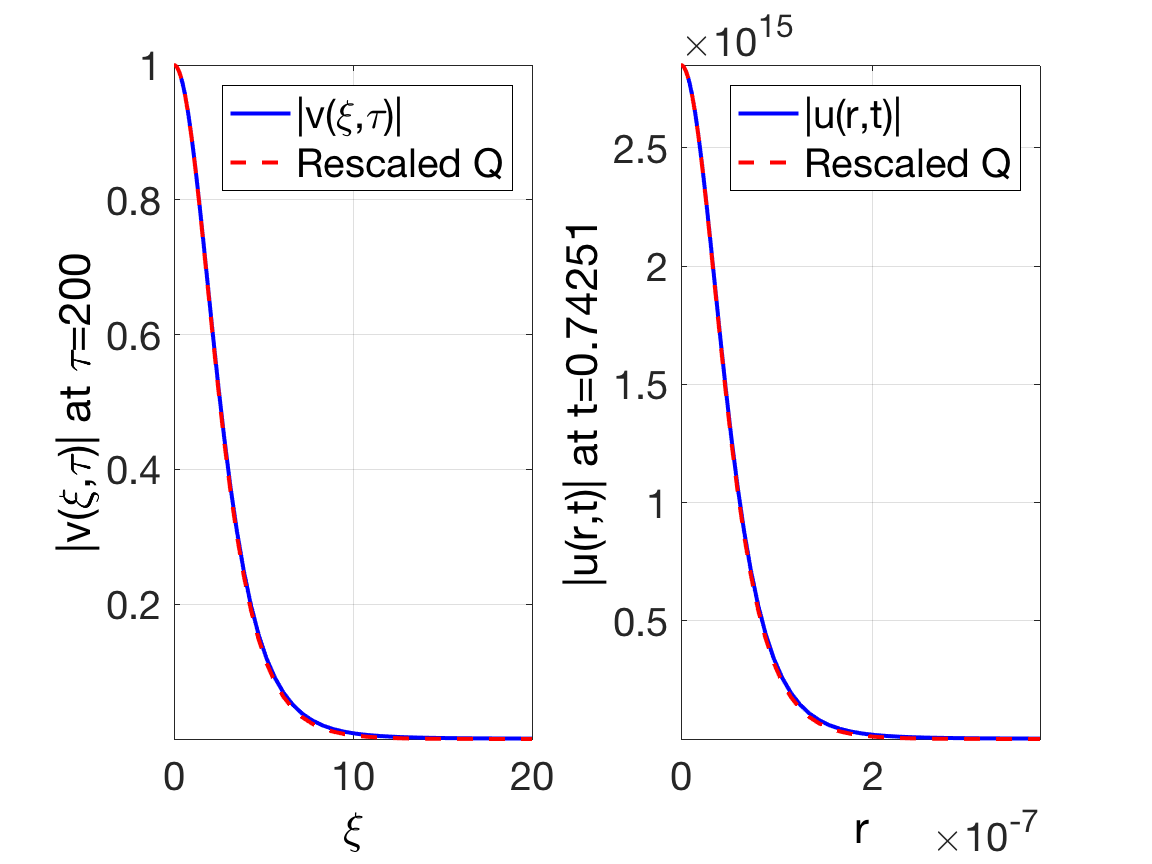}
\includegraphics[width=0.32\textwidth]{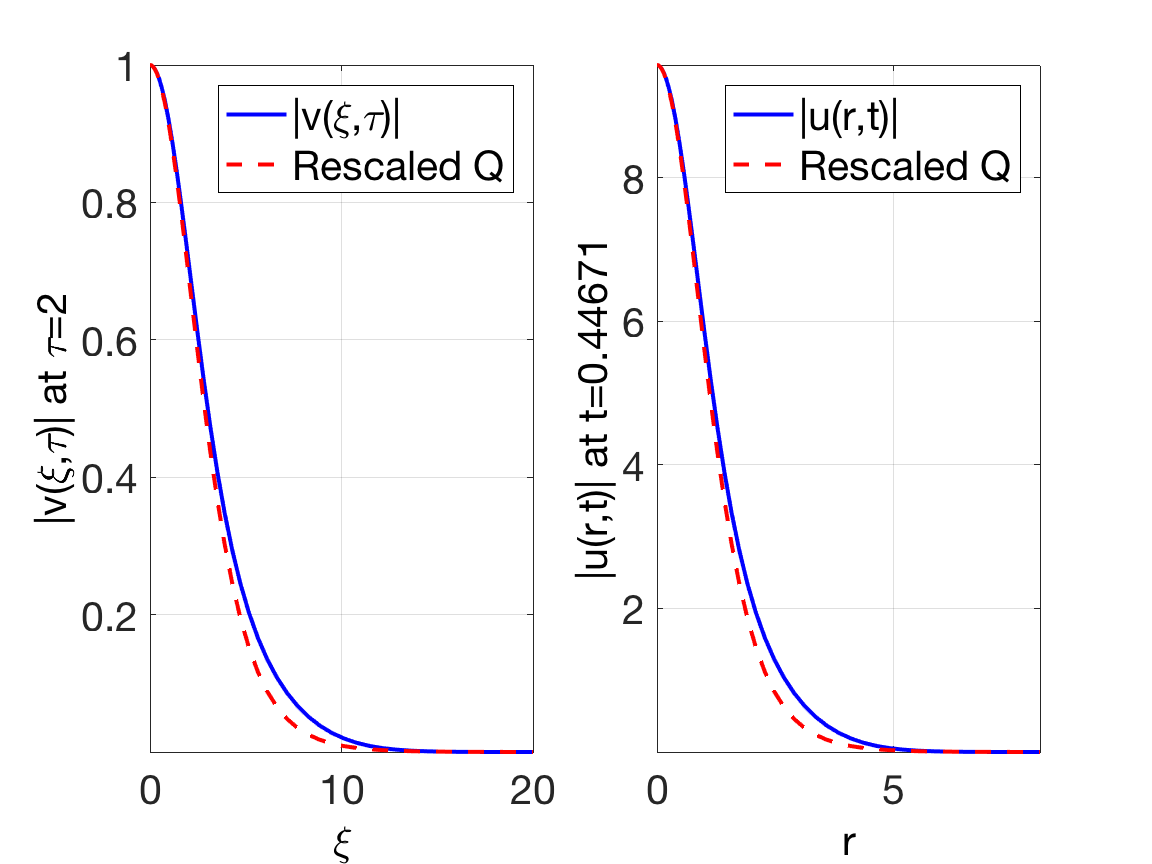}
\includegraphics[width=0.32\textwidth]{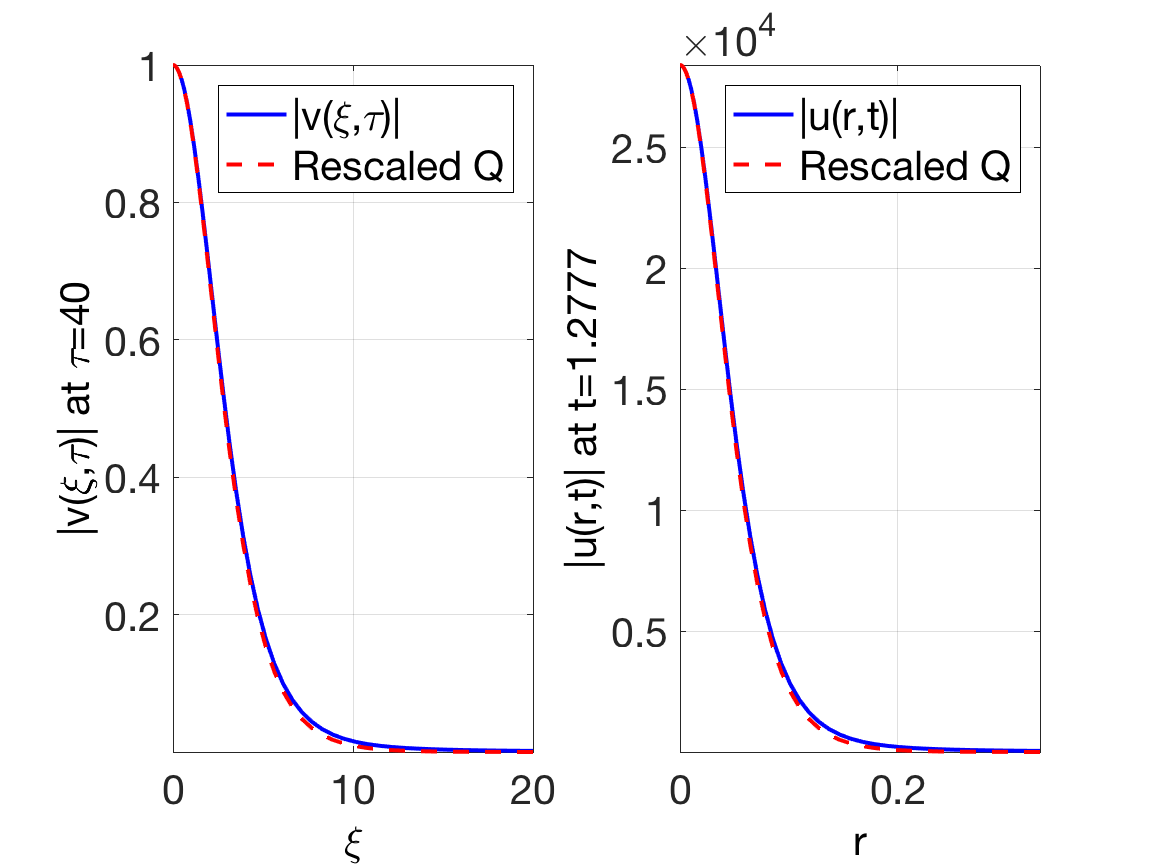}
\includegraphics[width=0.32\textwidth]{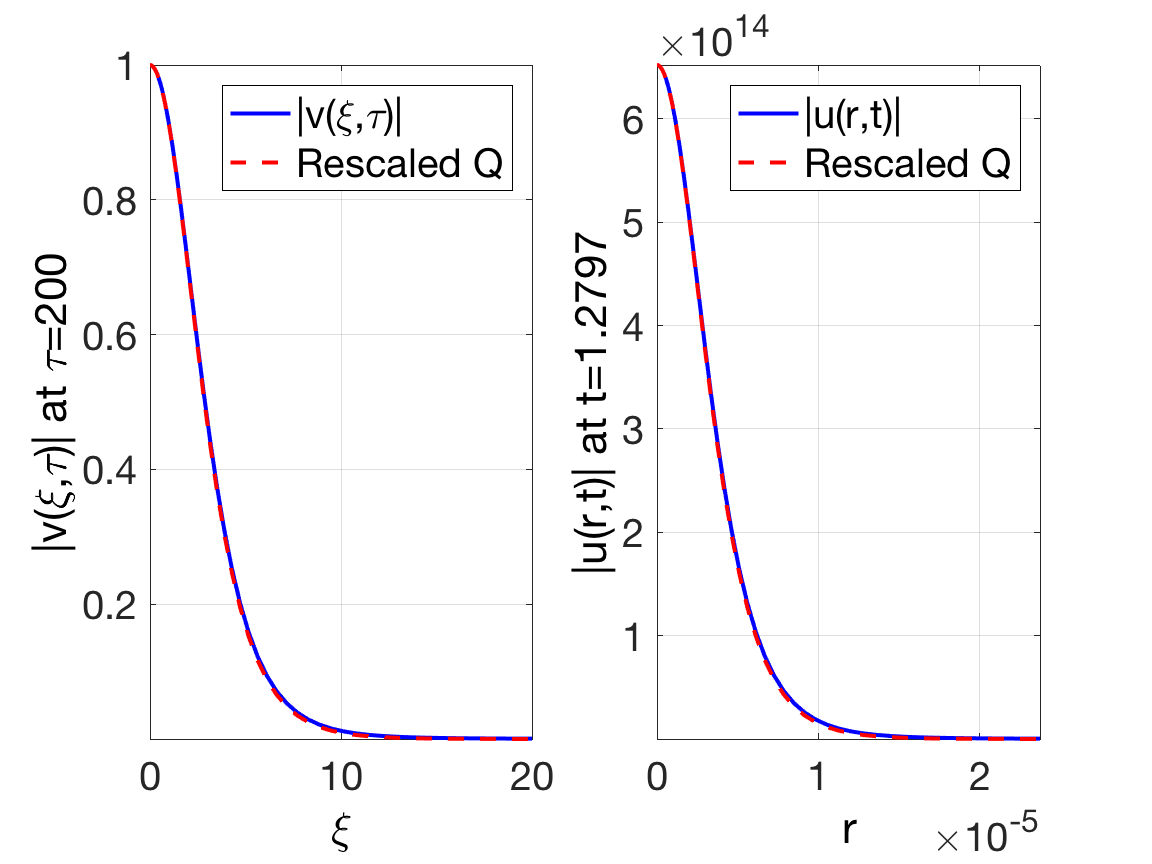}
\includegraphics[width=0.32\textwidth]{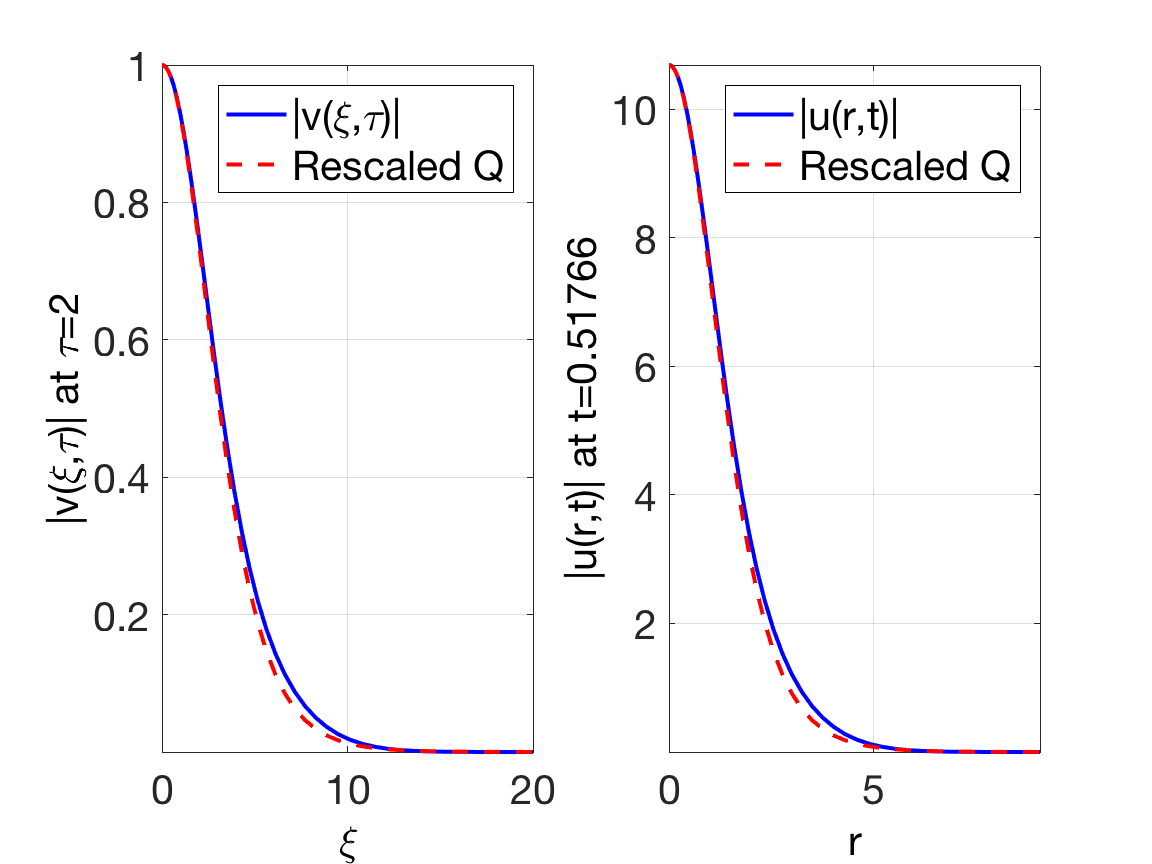}
\includegraphics[width=0.32\textwidth]{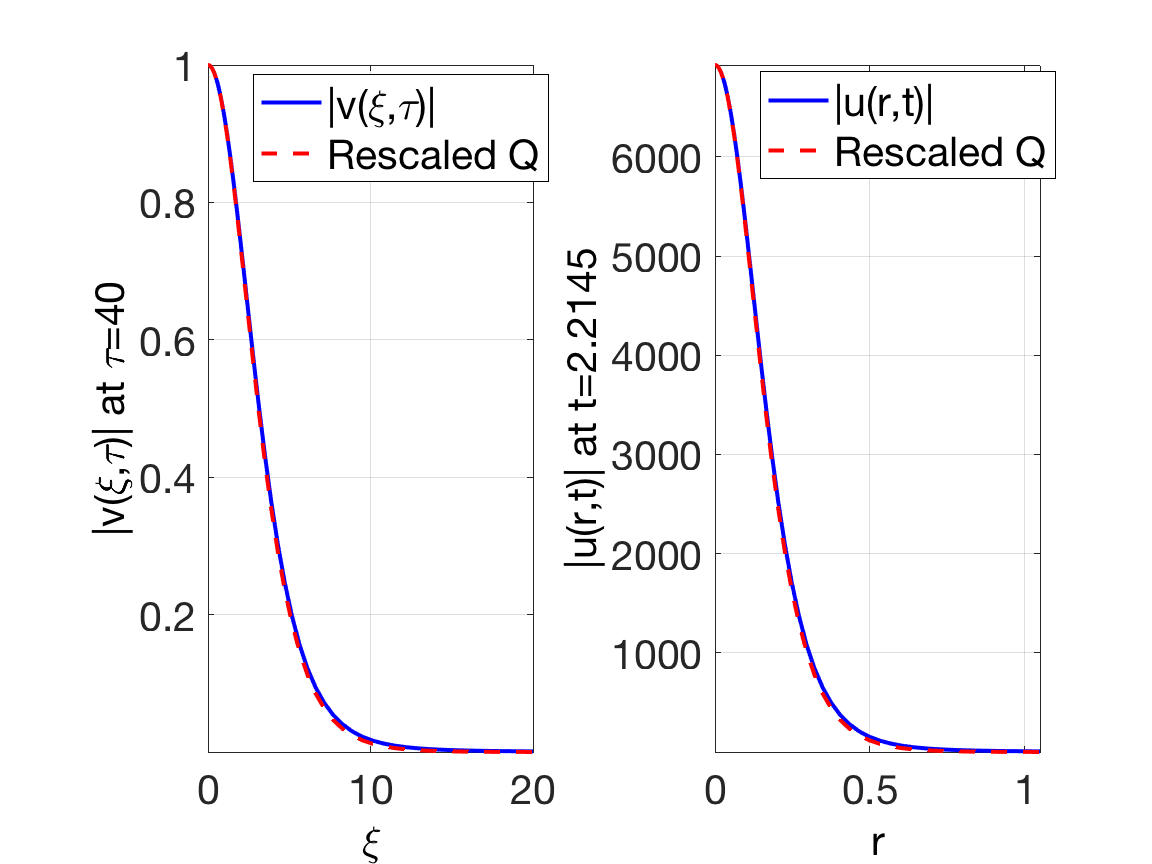}
\includegraphics[width=0.32\textwidth]{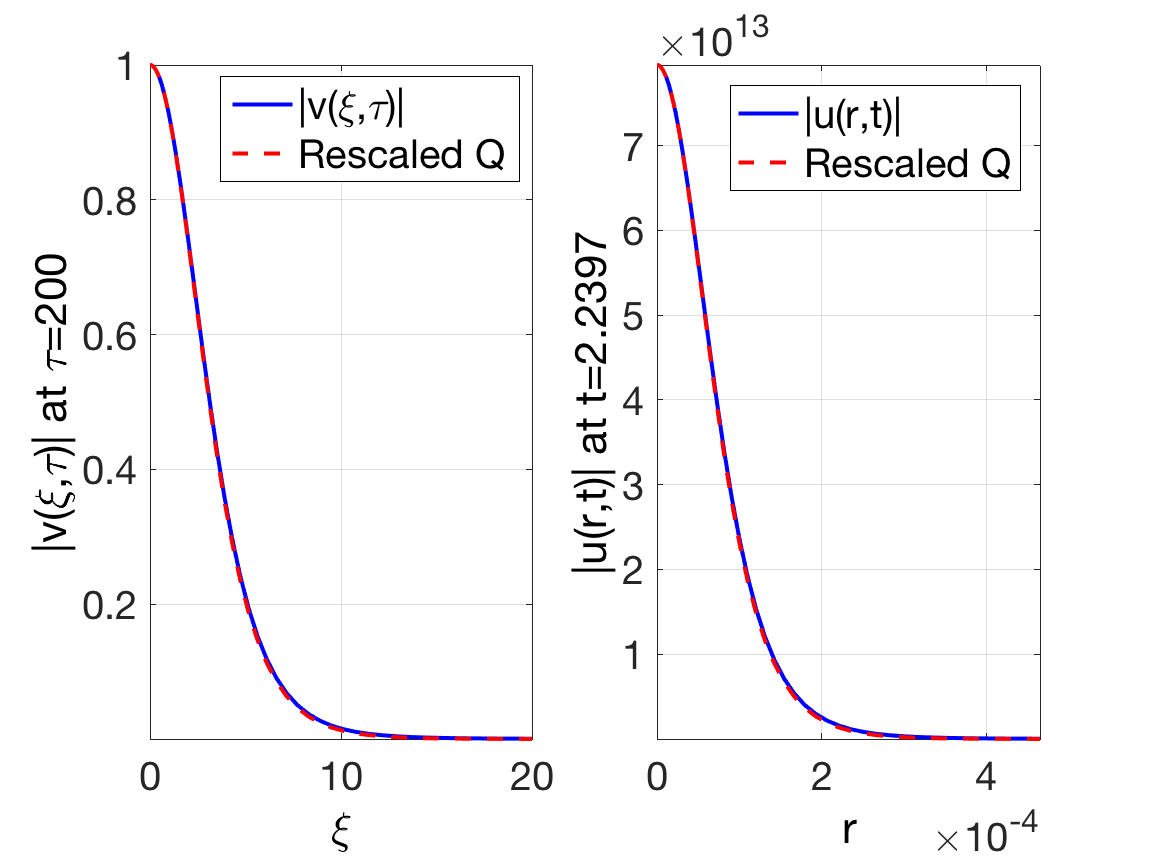}
\includegraphics[width=0.32\textwidth]{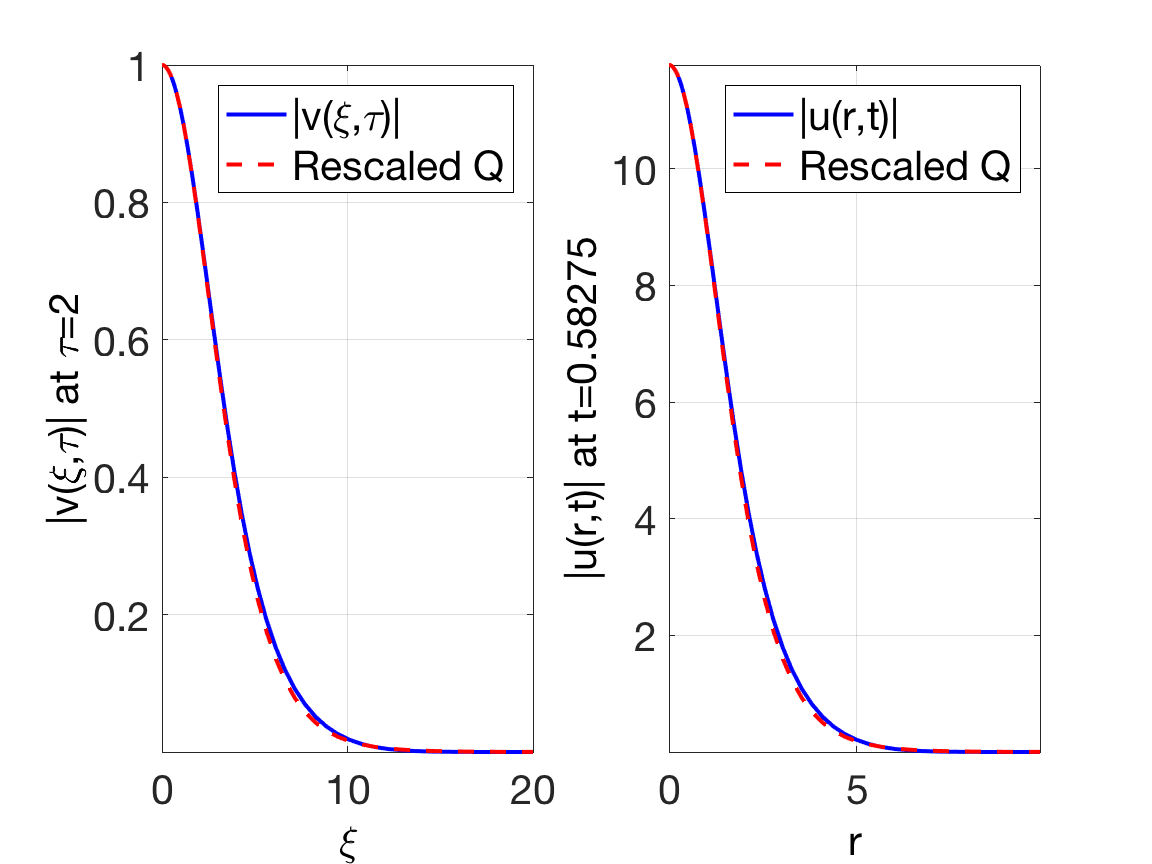}
\includegraphics[width=0.32\textwidth]{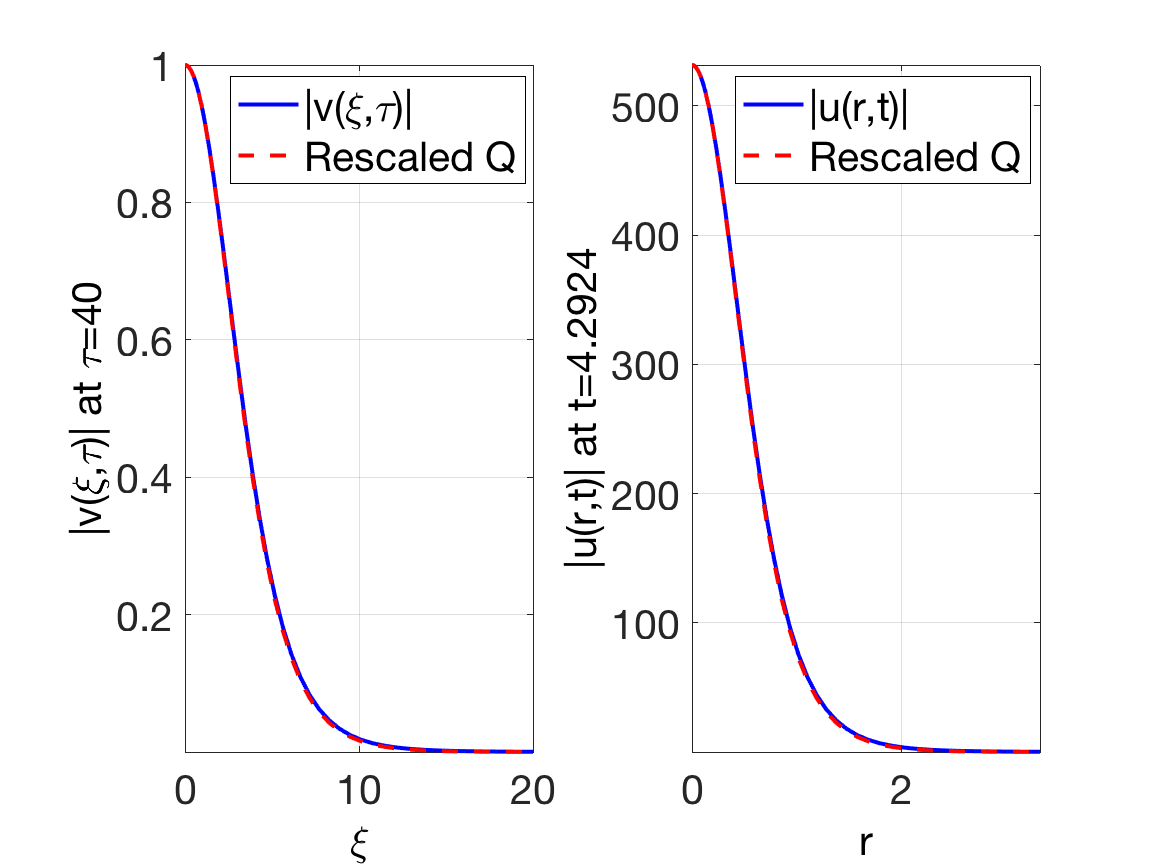}
\includegraphics[width=0.32\textwidth]{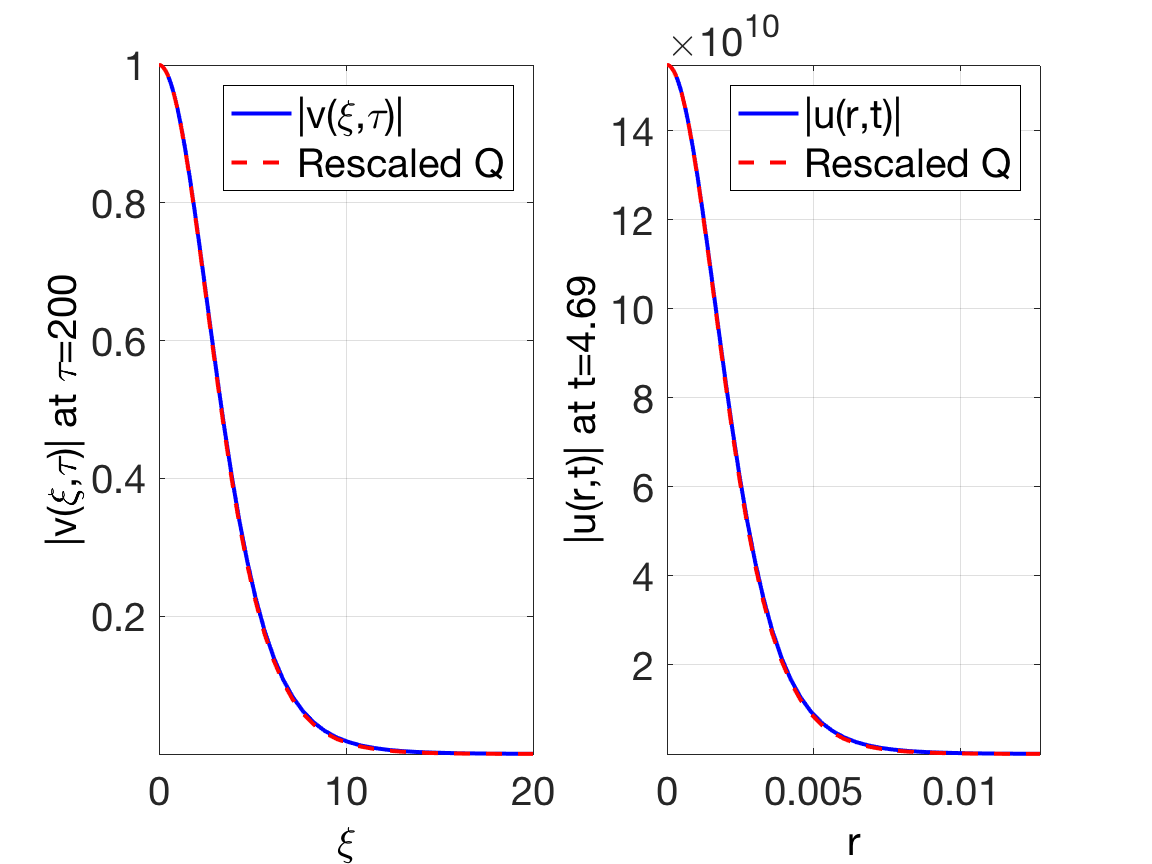}
\caption{Convergence of blow-up profile in $d=4$ (top), $d=5$ (middle top), $d=6$ (middle bottom), $d=7$ (bottom). We plot next to each other the rescaled profile $|v|$ in the rescaled time $\tau$ and the original solution $|u|$ in the actual time $t$. As $\tau \rightarrow \infty$ (here, $\tau=2, 40, 200$),
the profile (blue dots) approaches the rescaled ground state $Q$ ($Q_{rescaled}=\frac{1}{L^{\alpha}(t)} Q(\frac{r}{L(t)})$).}
\label{p7d}
\end{figure}


\subsection{First attempt to obtain the correction term in the blow-up rate}\label{Fitting}
In the NLS equation, the {\it log-log} regime is reached when the amplitude of the solution is extremely large ($\gg 10^{200}$), which is currently impossible to observe numerically. In \cite{ADKM2003} the functional form testing was suggested and the authors succeeded in showing that among all tested functional forms, the {\it log-log} form minimizes the errors in the fitting the best. This method has also been proven to be efficient in  checking the {\it log-log} correction for the NLS equation in higher dimensions, see \cite{YRZ2018}. In this paper, we also use this approach for the correction term in the blow-up rate in the gHartree equation. We write the rate as 
\begin{align}\label{FL}
\dfrac{1}{L(t)} \sim \left( \dfrac{F(T-t)}{T-t} \right)^{\frac{1}{2}},
\end{align}
where $F(s)= \left( \ln(s^{-1}) \right)^{\gamma}$ (for example, we consider $\gamma=1,0.6,0.5,0.4,0$) or $F(s)=\ln \ln s^{-1}$. 
We compute $\frac{1}{L(t_i)}$ at each time $t_i$, and also we check the following approximation parameter
\begin{align}\label{rho}
\rho_i=\dfrac{L(t_i)}{L(t_{i+1})} \Big/ \ln \left( \dfrac{F_{i+1}/(T-t_{i+1})}{F_{i}/(T-t_{i})} \right), ~~ \mathrm{where} ~~ F_{i} =F(T-t_i). 
\end{align}
Due to the leading square root decay, $\rho$ is expected to be $\frac12$, and thus, we check how fast the parameter $\rho_i$ converges to $\frac{1}{2}$ and which choice of $F(s)$
gives the best approximation. In \cite{ADKM2003} and \cite{YRZ2018} it was shown that $F(s)=\ln \ln s^{-1}$ provides the fastest convergence to $\frac12$ as well as the best parameter $\rho_i$ stabilization. Furthermore, $F(s)=\ln \ln s^{-1}$ gave the optimal quantity in the standard deviation $\epsilon$: for computational purposes we define it on each subinterval of values of $1/L(t)$ (denoting by $I_j$ the range of values, see for example, Table \ref{diviation 3d})
\begin{align}\label{derivation}
\epsilon=\left(\frac1{\#|I_i|} \sum_{j \in I_i}
\left(\frac{1}{2}-\rho_j\right)^2 \right)^{\frac12}.
\end{align}

We provide the results from our computations for the best fitting $\rho_i$ and standard deviation $\epsilon_i$ in Tables \ref{diviation 3d} - \ref{diviation 7d} for dimensions $d = 3, ..., 7$. One can notice that the {\it log-log} correction does the best minimization of the error in the fitting. We also find the optimal value of $\gamma$, denoted by $\gamma^{\star}$, such that $F(s)=(\ln s^{-1})^{\gamma^{\star}}$ gives the rate $\rho_i$ to be exactly $\frac{1}{2}$. The parameter $\gamma^{\star}$ was introduced in \cite{ADKM2003}, we also used it in \cite{YRZ2018} and found that even this parameter is decreasing, which also indicates that the correction should be weaker than $(\ln s^{-1})^{\gamma}$ for any $\gamma$. 
To compute $\gamma^\star$ we directly calculate $\rho_i$ from \eqref{rho}, which gives 
\begin{align}\label{gamma star}
\dfrac{1}{\rho_i(\gamma)}=\dfrac{\gamma}{\rho_i(1)}+\dfrac{1-\gamma}{\rho_i(0)},
\end{align}
together with 
\begin{align}\label{rho 0.5}
\rho_i(\gamma)=\dfrac{1}{2}.
\end{align}
Then we obtain the optimal $\gamma^{\star}$. Observe that the $\gamma^{\star}$ is decreasing as the magnitude of the range $I_i$ is increasing, this is similar to the behavior and results in the $L^2$-critical NLS equation. 

Some of the results are recorded in Tables \ref{diviation 3d}-\ref{diviation 7d}, where we tested $F(s)=1$, $F(s)=\left(\ln s^{-1} \right)^{\gamma}$, $\gamma=1, 0.5, 0.25$, and $F(s)=\ln \ln s^{-1}$; we also list the values of the optimal $\gamma^{\star}$ at the increasing magnitude range. As we mentioned above, the $\gamma^{\star}$ decreases as the magnitude increases (and as the solution approaches the blow-up time $T$). This indicates that none of $(\ln s^{-1})^{\gamma}$ corrections are good choice. Therefore, a weaker correction than $(\ln s^{-1})^{\gamma}$ is needed, giving more support to the {\it log-log} correction. 
Besides the forms of the functional fitting corrections already discussed, we also include the results $F_{mal}$ for the ``Malkin adiabatic" law $L(t) \approx (2\sqrt{b}(T-t))^{1/2}$ for a comparison (details about the adiabatic laws are in subsection \ref{S:Adiabatic}).

{\small
\begin{table}[ht]
\begin{tabular}{|c|c|c|c|c|c|c|c|c|}
\hline
\multicolumn{9}{|c|}{The fitting power $\rho_i$ from different corrections}\\
\hline
$i$ &$\frac{1}{L(t)}$ range 
& $F(s)=1$ 
& $F(s)=\ln \frac{1}{s}$ 
& $F(s)=\left( \ln \frac{1}{s} \right)^{0.5}$ 
& $F(s)=\left( \ln \frac{1}{s} \right)^{1/4}$ 
& $F(s)=\ln \ln \frac{1}{s}$  
& $\gamma^{\star}$
& $F_{mal}$ \\
\hline
$0$&$5e8 \sim 4e9 $&$0.5028$&$0.4908$&$0.4967$&$0.4997$&$0.4995$&$0.2291$&$0.5007$ \\
\hline
$1$&$4e9 \sim 3e10 $&$0.5024$&$0.4916$&$0.4969$&$0.4997$&$0.4995$&$0.2198$&$0.5006$ \\
\hline
$2$&$3e10 \sim 3e11 $&$0.5021$&$0.4922$&$0.4971$&$0.4996$&$0.4996$&$0.2120$&$0.5006$ \\
\hline
$3$&$3e11 \sim 2e12 $&$0.5019$&$0.4927$&$0.4973$&$0.4996$&$0.4996$&$0.2055$&$0.5004$ \\
\hline
$4$&$2e12 \sim 2e13 $&$0.5017$&$0.4932$&$0.4974$&$0.4996$&$0.4996$&$0.1998$&$0.5003$ \\
\hline
$5$&$2e13 \sim 1e15 $&$0.5016$&$0.4936$&$0.4975$&$0.4996$&$0.4996$&$0.1948$&$0.5004$ \\
\hline
$6$&$1e15 \sim 7e15 $&$0.5014$&$0.4939$&$0.4977$&$0.4995$&$0.4996$&$0.1904$&$0.5004$ \\
\hline
$7$&$7e15 \sim 5e16 $&$0.5013$&$0.4942$&$0.4978$&$0.4995$&$0.4997$&$0.1865$&$0.5003$ \\
\hline
\end{tabular}
\linebreak
\linebreak
\begin{tabular}{|c|c|c|c|c|c|c|c|}
\hline
\multicolumn{8}{|c|}{The $\epsilon_i$ from different corrections}\\
\hline
$i$ &$\frac{1}{L(t)}$ range 
& $F(s)=1$ 
& $F(s)=\ln \frac{1}{s}$ 
& $F(s)= \left( \ln \frac{1}{s}\right)^{0.5}$ 
& $F(s)= \left( \ln \frac{1}{s}\right)^{1/4}$ 
& $F(s)=\ln \ln \frac{1}{s}$
& $F_{mal}$ \\
\hline
$0$&$5e8 \sim 4e9 $&$0.0028$&$0.0092$&$0.0033$&$2.53e-4$&$4.87e-4$&$7.12e-4$ \\
\hline
$1$&$4e9 \sim 3e10 $&$0.0026$&$0.0088$&$0.0032$&$2.95e-4$&$4.77e-4$&$6.70e-4$ \\
\hline
$2$&$3e10 \sim 3e11 $&$0.0025$&$0.0085$&$0.0031$&$3.27e-4$&$4.67e-4$&$6.46e-4$ \\
\hline
$3$&$3e11 \sim 2e12 $&$0.0023$&$0.0082$&$0.0030$&$3.50e-4$&$4.57e-4$&$5.98e-4$ \\
\hline
$4$&$2e12 \sim 2e13 $&$0.0022$&$0.0079$&$0.0029$&$3.67e-4$&$4.47e-4$&$5.56e-4$ \\
\hline
$5$&$2e13 \sim 1e15 $&$0.0021$&$0.0077$&$0.0028$&$3.83e-4$&$4.37e-4$&$5.32e-4$ \\
\hline
$6$&$1e15 \sim 7e15 $&$0.0021$&$0.0075$&$0.0028$&$3.93e-4$&$4.28e-4$&$5.14e-4$\\
\hline
$7$&$7e15 \sim 5e16 $&$0.0020$&$0.0073$&$0.0027$&$4.01e-4$&$4.19e-4$&$4.91e-4$ \\
\hline
\end{tabular}
\linebreak
\linebreak
\caption{3d case. Top table: comparison of curve fitting for various choices of correction terms $F(s)$. Here, ``$\frac{1}{L(t)}$ range" means values are in the range $\frac{1}{L(t_i)} \sim \frac{1}{L(t_{i+1})}$. 
Bottom table:
standard deviation $\epsilon_i$ for different corrections $F(s)$ from the top table. 
The {\it log-log} correction produces the minimal error in $\epsilon_i$.}
\label{diviation 3d}
\end{table}
}

{\small
\begin{table}[ht]
\begin{tabular}{|c|c|c|c|c|c|c|c|c|}
\hline
\multicolumn{9}{|c|}{The fitting power $\rho_i$ from different corrections}\\
\hline
$i$ &$\frac{1}{L(t)}$ range 
& $F(s)=1$ 
& $F(s)=\ln \frac{1}{s}$ 
& $F(s)=\left( \ln \frac{1}{s} \right)^{0.5}$ 
& $F(s)=\left( \ln \frac{1}{s} \right)^{1/4}$ 
& $F(s)=\ln \ln \frac{1}{s}$  
& $\gamma^{\star}$ 
& $F_{mal}$ \\
\hline
$0$&$4e8 \sim 3e9 $&$0.5029$&$0.4906$&$0.4966$&$0.4997$&$0.4995$&$0.2281$&$0.5005$ \\
\hline
$1$&$3e9 \sim 3e10 $&$0.5025$&$0.4913$&$0.4968$&$0.4996$&$0.4995$&$0.2179$&$0.5000$ \\
\hline
$2$&$3e10 \sim 2e11 $&$0.5022$&$0.4920$&$0.4970$&$0.4996$&$0.4995$&$0.2095$&$0.5013$ \\
\hline
$3$&$2e11 \sim 2e12 $&$0.5019$&$0.4925$&$0.4972$&$0.4995$&$0.4995$&$0.2024$&$0.4997$ \\
\hline
$4$&$2e12 \sim 1e13 $&$0.5017$&$0.4930$&$0.4973$&$0.4995$&$0.4995$&$0.1963$&$0.5008$ \\
\hline
$5$&$1e13 \sim 1e14 $&$0.5016$&$0.4934$&$0.4975$&$0.4995$&$0.4996$&$0.1910$&$0.5004$ \\
\hline
$6$&$1e14 \sim 7e14 $&$0.5014$&$0.4938$&$0.4976$&$0.4995$&$0.4996$&$0.1863$&$0.5000$ \\
\hline
$7$&$7e14 \sim 5e15 $&$0.5013$&$0.4941$&$0.4977$&$0.4995$&$0.4996$&$0.1821$&$0.5005$ \\
\hline
\end{tabular}
\linebreak
\linebreak
\begin{tabular}{|c|c|c|c|c|c|c|c|}
\hline
\multicolumn{8}{|c|}{The $\epsilon_i$ from different corrections}\\
\hline
$i$ &$\frac{1}{L(t)}$ range 
& $F(s)=1$ 
& $F(s)=\ln \frac{1}{s}$ 
& $F(s)= \left( \ln \frac{1}{s}\right)^{0.5}$ 
& $F(s)= \left( \ln \frac{1}{s}\right)^{1/4}$ 
& $F(s)=\ln \ln \frac{1}{s}$ 
& $F_{mal}$ \\
\hline
$0$&$2e7 \sim 4e8 $&$0.0050$&$0.0125$&$0.0039$&$6.55e-4$&$5.00e-4$&$4.79e-4$ \\
\hline
$1$&$4e8 \sim 7e9 $&$0.0046$&$0.0118$&$0.0038$&$5.86e-4$&$5.09e-4$&$3.38e-4$ \\
\hline
$2$&$7e9 \sim 1e11 $&$0.0042$&$0.0112$&$0.0037$&$5.53e-4$&$5.09e-4$&$7.87e-4$ \\
\hline
$3$&$1e11 \sim 2e12 $&$0.0039$&$0.0107$&$0.0035$&$5.36e-4$&$5.04e-4$&$6.99e-4$ \\
\hline
$4$&$2e12 \sim 2e13 $&$0.0037$&$0.0102$&$0.0034$&$5.28e-4$&$4.97e-4$&$7.16e-4$ \\
\hline
$5$&$2e13 \sim 3e14 $&$0.0035$&$0.0098$&$0.0033$&$5.23e-4$&$4.88e-4$&$6.72e-4$ \\
\hline
$6$&$3e14 \sim 5e15 $&$0.0033$&$0.0095$&$0.0032$&$5.20e-4$&$4.79e-4$&$6.22e-4$ \\
\hline
$7$&$5e15 \sim 7e16 $&$0.0032$&$0.0092$&$0.0031$&$5.17e-4$&$4.70e-4$&$6.04e-4$ \\
\hline
\end{tabular}
\linebreak
\linebreak
\caption{4d case. Top table: comparison of curve fitting for different choices of correction terms $F(s)$.  
Bottom table: The standard deviation $\epsilon_i$ for different corrections $F(s)$ from the above table. Similar to the 3d case, the {\it log-log} correction produces the minimal error in $\epsilon_i$.}
\label{diviation 4d} 
\end{table}
}

{\small
\begin{table}[ht]
\begin{tabular}{|c|c|c|c|c|c|c|c|c|}
\hline
\multicolumn{9}{|c|}{The fitting power $\rho_i$ from different corrections}\\
\hline
$i$ &$\frac{1}{L(t)}$ range 
& $F(s)=1$ 
& $F(s)=\ln \frac{1}{s}$ 
& $F(s)=\left( \ln \frac{1}{s} \right)^{0.5}$ 
& $F(s)=\left( \ln \frac{1}{s} \right)^{1/4}$ 
& $F(s)=\ln \ln \frac{1}{s}$  
& $\gamma^{\star}$ 
& $F_{mal}$ \\
\hline
$0$&$1e10 \sim 4e10 $&$0.5037$&$0.4931$&$0.4983$&$0.5010$&$0.5009$&$0.3453$&$0.5018$ \\
\hline
$1$&$4e10 \sim 2e11 $&$0.5034$&$0.4934$&$0.4983$&$0.5009$&$0.5008$&$0.3359$&$0.5018$ \\
\hline
$2$&$2e11 \sim 7e11 $&$0.5032$&$0.4936$&$0.4983$&$0.5007$&$0.5007$&$0.3275$&$0.5013$ \\
\hline
$3$&$7e11 \sim 3e12 $&$0.5029$&$0.4939$&$0.4984$&$0.5006$&$0.5006$&$0.3198$&$0.5014$ \\
\hline
$4$&$3e12 \sim 1e13 $&$0.5027$&$0.4941$&$0.4984$&$0.5005$&$0.5006$&$0.3066$&$0.5013$ \\
\hline
$5$&$1e13 \sim 5e13 $&$0.5025$&$0.4943$&$0.4984$&$0.5005$&$0.5005$&$0.3009$&$0.5012$ \\
\hline
$6$&$5e13 \sim 2e14 $&$0.5024$&$0.4945$&$0.4984$&$0.5004$&$0.5005$&$0.2956$&$0.5015$ \\
\hline
$7$&$2e14 \sim 7e14 $&$0.5023$&$0.4947$&$0.4985$&$0.5003$&$0.5004$&$0.2906$&$0.5018$ \\
\hline
\end{tabular}
\linebreak
\linebreak
\begin{tabular}{|c|c|c|c|c|c|c|c|}
\hline
\multicolumn{8}{|c|}{The $\epsilon_i$ from different corrections}\\
\hline
$i$ &$\frac{1}{L(t)}$ range 
& $F(s)=1$ 
& $F(s)=\ln \frac{1}{s}$ 
& $F(s)= \left( \ln \frac{1}{s}\right)^{0.5}$ 
& $F(s)= \left( \ln \frac{1}{s}\right)^{1/4}$ 
& $F(s)=\ln \ln \frac{1}{s}$ 
& $F_{mal}$ \\
\hline
$0$&$1e10 \sim 4e10 $&$0.0037$&$0.0069$&$0.0017$&$0.0010$&$9.09e-4$&$0.0018$ \\
\hline
$1$&$4e10 \sim 2e11 $&$0.0036$&$0.0068$&$0.0017$&$9.54e-4$&$8.57e-4$&$0.0017$ \\
\hline
$2$&$2e11 \sim 7e11 $&$0.0035$&$0.0065$&$0.0017$&$8.89e-4$&$8.11e-4$&$0.0016$ \\
\hline
$3$&$7e11 \sim 3e12 $&$0.0033$&$0.0064$&$0.0017$&$8.33e-4$&$7.71e-4$&$0.0016$ \\
\hline
$4$&$3e12 \sim 1e13 $&$0.0032$&$0.0063$&$0.0016$&$7.84e-4$&$7.36e-4$&$0.0015$ \\
\hline
$5$&$1e13 \sim 5e13 $&$0.0031$&$0.0062$&$0.0016$&$7.41e-4$&$7.04e-4$&$0.0015$ \\
\hline
$6$&$5e13 \sim 2e14 $&$0.0030$&$0.0061$&$0.0016$&$7.03e-4$&$6.76e-4$&$0.0015$ \\
\hline
$7$&$2e14 \sim 7e14 $&$0.0029$&$0.0060$&$0.0016$&$56.69e-4$&$6.50e-4$&$0.0015$ \\
\hline
\end{tabular}
\linebreak
\linebreak
\caption{5d case. Top table: comparison of curve fitting for various choices of $F(s)$. 
Bottom table: standard deviation $\epsilon_i$ for different $F(s)$ from the above table. 
As in 3d, 4d, the {\it log-log} correction produces the minimal error in 
$\epsilon_i$.}
\label{diviation 5d} 
\end{table}
}

{\small
\begin{table}[ht]
\begin{tabular}{|c|c|c|c|c|c|c|c|c|}
\hline
\multicolumn{9}{|c|}{The fitting power $\rho_i$ from different corrections}\\
\hline
$i$ &$\frac{1}{L(t)}$ range 
& $F(s)=1$ 
& $F(s)=\ln \frac{1}{s}$ 
& $F(s)=\left( \ln \frac{1}{s} \right)^{0.5}$ 
& $F(s)=\left( \ln \frac{1}{s} \right)^{0.25}$ 
& $F(s)=\ln \ln \frac{1}{s}$  
& $\gamma^{\star}$ 
& $F_{mal}$ \\
\hline
$0$&$1e10 \sim 4e10 $&$0.5036$&$0.4929$&$0.4982$&$0.5009$&$0.5008$&$0.3322$&$0.5028$ \\
\hline
$1$&$4e10 \sim 2e11 $&$0.5033$&$0.4932$&$0.4982$&$0.5008$&$0.5007$&$0.3250$&$0.5044$ \\
\hline
$2$&$2e11 \sim 7e11 $&$0.5031$&$0.4935$&$0.4983$&$0.5007$&$0.5006$&$0.3184$&$0.4951$ \\
\hline
$3$&$7e11 \sim 3e12 $&$0.5029$&$0.4938$&$0.4983$&$0.5006$&$0.5006$&$0.3124$&$0.5084$ \\
\hline
$4$&$3e12 \sim 1e13 $&$0.5027$&$0.4940$&$0.4983$&$0.5005$&$0.5005$&$0.3069$&$0.4972$ \\
\hline
$5$&$1e13 \sim 4e13 $&$0.5025$&$0.4942$&$0.4983$&$0.5004$&$0.5005$&$0.3018$&$0.5013$ \\
\hline
$6$&$4e13 \sim 1e14 $&$0.5024$&$0.4946$&$0.4984$&$0.5004$&$0.5004$&$0.2970$&$0.5046$ \\
\hline
$7$&$1e14 \sim 5e14 $&$0.5023$&$0.4948$&$0.4984$&$0.5003$&$0.5004$&$0.2926$&$0.4981$ \\
\hline
\end{tabular}
\linebreak
\linebreak
\begin{tabular}{|c|c|c|c|c|c|c|c|}
\hline
\multicolumn{8}{|c|}{The $\epsilon_i$ from different corrections}\\
\hline
$i$ &$\frac{1}{L(t)}$ range 
& $F(s)=1$ 
& $F(s)=\ln \frac{1}{s}$ 
& $F(s)= \left( \ln \frac{1}{s}\right)^{0.5}$ 
& $F(s)= \left( \ln \frac{1}{s}\right)^{0.25}$ 
& $F(s)=\ln \ln \frac{1}{s}$ 
& $F_{mal}$ \\
\hline
$0$&$1e10 \sim 4e10 $&$0.0036$&$0.0071$&$0.0018$&$9.16e-4$&$7.68e-4$&$0.0028$ \\
\hline
$1$&$4e10 \sim 2e11 $&$0.0035$&$0.0069$&$0.0018$&$8.89e-4$&$7.30e-4$&$0.0037$ \\
\hline
$2$&$2e11 \sim 7e11 $&$0.0034$&$0.0068$&$0.0018$&$8.29e-4$&$6.97e-4$&$0.0041$ \\
\hline
$3$&$7e11 \sim 3e12 $&$0.0032$&$0.0067$&$0.0018$&$7.78e-4$&$6.67e-4$&$0.0055$ \\
\hline
$4$&$3e12 \sim 1e13 $&$0.0031$&$0.0065$&$0.0017$&$7.32e-4$&$6.41e-4$&$0.0051$ \\
\hline
$5$&$1e13 \sim 4e13 $&$0.0030$&$0.0064$&$0.0017$&$6.92e-4$&$6.16e-4$&$0.0047$ \\
\hline
$6$&$4e13 \sim 1e14 $&$0.0030$&$0.0063$&$0.0017$&$6.56e-4$&$5.94e-4$&$0.0046$ \\
\hline
$7$&$1e14 \sim 7e14 $&$0.0029$&$0.0062$&$0.0017$&$6.24e-4$&$5.74e-4$&$0.0045$ \\
\hline
\end{tabular}
\linebreak
\linebreak
\caption{6d case. Top table: comparison of the curve fitting for different choices of $F(s)$.
Bottom table: standard deviation $\epsilon_i$ for different $F(s)$ from the above table. 
The {\it log-log} correction produces the minimal error in $\epsilon_i$.}
\label{diviation 6d} 
\end{table}
}

{\small
\begin{table}[ht]
\begin{tabular}{|c|c|c|c|c|c|c|c|c|}
\hline
\multicolumn{9}{|c|}{The fitting power $\rho_i$ from different corrections}\\
\hline
$i$ &$\frac{1}{L(t)}$ range 
& $F(s)=1$ 
& $F(s)=\ln \frac{1}{s}$ 
& $F(s)=\left( \ln \frac{1}{s} \right)^{0.5}$ 
& $F(s)=\left( \ln \frac{1}{s} \right)^{1/4}$ 
& $F(s)=\ln \ln \frac{1}{s}$  
& $\gamma^{\star}$ 
& $F_{mal}$ \\
\hline
$0$&$6e9 \sim 4e10 $&$0.5033$&$0.4924$&$0.4978$&$0.5005$&$0.5004$&$0.2983$&$0.5127$ \\
\hline
$1$&$4e10 \sim 2e11 $&$0.5030$&$0.4929$&$0.4979$&$0.5004$&$0.5004$&$0.2930$&$0.4924$ \\
\hline
$2$&$2e11 \sim 1e12 $&$0.5028$&$0.4932$&$0.4980$&$0.5004$&$0.5003$&$0.2880$&$0.4432$ \\
\hline
$3$&$1e12 \sim 8e12 $&$0.5025$&$0.4937$&$0.4981$&$0.5003$&$0.5003$&$0.2834$&$0.3996$ \\
\hline
$4$&$8e12 \sim 4e13 $&$0.5024$&$0.4940$&$0.4981$&$0.5002$&$0.5003$&$0.2792$&$0.6070$ \\
\hline
$5$&$4e13 \sim 2e14 $&$0.5022$&$0.4943$&$0.4982$&$0.5002$&$0.5003$&$0.2752$&$0.5287$ \\
\hline
$6$&$2e14 \sim 1e15 $&$0.5021$&$0.4946$&$0.4983$&$0.5002$&$0.5002$&$0.2715$&$0.5070$ \\
\hline
$7$&$1e15 \sim 7e15 $&$0.5019$&$0.4948$&$0.4983$&$0.5001$&$0.5002$&$0.2680$&$0.4918$ \\
\hline
\end{tabular}
\linebreak
\linebreak
\begin{tabular}{|c|c|c|c|c|c|c|c|}
\hline
\multicolumn{8}{|c|}{The $\epsilon_i$ from different corrections}\\
\hline
$i$ &$\frac{1}{L(t)}$ range 
& $F(s)=1$ 
& $F(s)=\ln \frac{1}{s}$ 
& $F(s)= \left( \ln \frac{1}{s}\right)^{0.5}$ 
& $F(s)= \left( \ln \frac{1}{s}\right)^{1/4}$ 
& $F(s)=\ln \ln \frac{1}{s}$ 
& $F_{mal}$ \\
\hline
$0$&$6e9 \sim 4e10 $&$0.0033$&$0.0076$&$0.0022$&$5.32e-4$&$3.94e-4$&$0.0127$ \\
\hline
$1$&$4e10 \sim 2e11 $&$0.0032$&$0.0074$&$0.0022$&$4.88e-4$&$3.79e-4$&$0.0105$ \\
\hline
$2$&$2e11 \sim 1e12 $&$0.0030$&$0.0071$&$0.0021$&$4.50e-4$&$3.65e-4$&$0.0339$ \\
\hline
$3$&$1e12 \sim 8e12 $&$0.0029$&$0.0070$&$0.0021$&$4.17e-4$&$3.52e-4$&$0.1035$ \\
\hline
$4$&$8e12 \sim 4e13 $&$0.0028$&$0.0068$&$0.0020$&$3.89e-4$&$3.40e-4$&$0.1042$ \\
\hline
$5$&$4e13 \sim 2e14 $&$0.0027$&$0.0066$&$0.0020$&$3.65e-4$&$3.29e-4$&$0.0959$ \\
\hline
$6$&$2e14 \sim 1e15 $&$0.0026$&$0.0065$&$0.0020$&$3.43e-4$&$3.18e-4$&$0.0888$ \\
\hline
$7$&$1e15 \sim 7e15 $&$0.0026$&$0.0063$&$0.0019$&$3.24e-4$&$3.09e-4$&$0.0831$ \\
\hline
\end{tabular}
\linebreak
\linebreak
\caption{7d case. Top table: comparison of curve fitting for different choices of $F(s)$. 
Bottom table: standard deviation $\epsilon_i$ for different $F(s)$ from the above table 
The {\it log-log} correction produces the minimal error in $\epsilon_i$. 
}
\label{diviation 7d} 
\end{table}
}

\subsection{Second attempt: asymptotic analysis for the correction term}\label{AA}
We follow the argument for the $L^2$-critical NLS equation in \cite[Chapter 8]{SS1999}.
From asymptotic considerations we will confirm the hypothesis that $a(\tau) \sim 1/(\ln(\tau)+3 \ln \ln \tau)$, which leads to the {\it log-log} correction term on the blow-up rate for the $L^2$-critical gHartree equation, i.e., we show that
\begin{equation}\label{blowup rate}
L(t) \approx \left( \frac{2 \pi (T-t)}{\ln \ln(\frac{1}{T-t})}\right)^{\frac{1}{2}}.
\end{equation}
We note that while numerically we saw no difference in the blow-up regime for different dimensions, results in this section are conditional for dimensions $d \geq 5$, since 
the local well-posedness is not yet available in gHartree when $p<2$ (or $\sigma<\frac{1}{2}$).

\subsubsection{Slow decay of $a(\tau)$}
Recalling the proof of Proposition \ref{$a=0$}, we note that the blow-up solutions have the quadratic phase 
\begin{align}
\theta(\xi)=-a\xi^2/4.
\end{align}
Writing $Q=e^{-ia\xi^2/4}P$, we have $P$ satisfy
\begin{align}\label{P eqn}
\begin{cases}
\Delta P -P +\dfrac{a^2\xi^2}{4}P-ia\dfrac{d\sigma-2}{2\sigma}P+ \left((-\Delta)^{-1}|P|^{2 \sigma +1} \right)|P|^{2 \sigma -1} P=0, \\
 P_{\xi}(0)=0, \quad P(\xi)=0 \:\: \mathrm{as}\:\: \xi \rightarrow \infty, \quad P(0) \:\:\textrm{is  real}.
 \end{cases}
\end{align}
Note that when $a=0$, \eqref{P eqn} reduces to the ($L^2$-critical) ground state (to distinguish it here, 
we write $R$ instead of $Q$):
\begin{align}\label{R}
\Delta R-R+ \left((-\Delta)^{-1}|R|^{2 \sigma +1} \right)|R|^{2 \sigma -1} R=0.
\end{align}
This suggests that the blow-up profiles converge to the ground state $R$ as $a\rightarrow 0$, which matches our numerical observations shown in Figures \ref{p3d}-\ref{p7d}. 
The following proposition shows that $a(\tau)$ decays to zero slower than any polynomial rate. We show later (similar to the NLS case in \cite[Section 8.1.4]{SS1999}), the diminishing criticality $s_c \to 0$, or dimension $d \to \frac{2}{\sigma}$, which involves the complex term in \eqref{P eqn} with
\begin{equation}\label{E:nu}
\nu(a) \defeq i a\,\frac{\sigma d -2}{2\sigma} \longrightarrow 0,
\end{equation}
is responsible for the {\it log-log} correction in the blow-up rate.  

\begin{proposition}\label{P: a decay}
If $Q$ is a solution of \eqref{Q eqn} with finite Hamiltonian, the function $d(a)$ is differentiable to all orders at $a=0$ and
\begin{align}\label{a slow}
\dfrac{d^p}{da^p}\bigg( d(a)-\frac{2}{\sigma }\bigg)\vert_{a=0}=0 \quad \mbox{for}~\mbox{all}~~p=0,1,2,\cdots.
\end{align}
\end{proposition}
The proof for Proposition \ref{P: a decay} is similar to the one of Proposition 8.1 in \cite[Chapter 8]{SS1999} with the appropriate modifications of the identities involving the potential term, and the fact that the non-zero constant Hamiltonian of $Q$ gives the same conclusion as the zero Hamiltonian of $Q$ in the NLS case. We only show the differences.
\begin{lemma}\label{L: lemma sec4} For R in \eqref{R}, we have
\begin{equation}\label{E: R 1} 
\int \left( |R_{\xi}|^2 -\frac{1}{2\sigma+1} V(R) \right)\xi^{d-1} d\xi=0.
\end{equation}
For $\rho$ satisfying
\begin{equation}\label{E: rho}
\Delta \rho -\rho +2\sigma \left((-\Delta)^{-1} R^{2\sigma+1} \right)R^{2\sigma-1}\rho+(2\sigma+1)  \left((-\Delta)^{-1} R^{2\sigma}\rho \right)R^{2\sigma}=-\frac{1}{4}\xi^2R,
\end{equation}
we have 
\begin{equation}\label{E: R 2}
\int \left( R\rho -\frac{1}{8}\xi^2 R^2 \right)\xi^{d-1} d\xi=0.
\end{equation}
For $g$ solving 
\begin{equation}\label{E:g}
\Delta g -g +2\sigma \left((-\Delta)^{-1} R^{2\sigma+1} \right)R^{2\sigma-1}g+(2\sigma+1)  \left((-\Delta)^{-1} R^{2\sigma}g \right)R^{2\sigma}=-\frac{R_{\xi}}{\xi},
\end{equation}
we have
\begin{equation}\label{E: R 3}
-2\int Rg\xi^{d-1}d\xi+\int \left( R_{\xi}^2-\frac{1}{2\sigma+1} V(R) \right)\xi^{d-1}\ln \xi d\xi=\frac{1}{2}\int R^2\xi^{d-1}d\xi.
\end{equation}
\end{lemma}

\begin{proof}
The identity \eqref{E: R 1} is simply the Pohozaev identity. 
To get \eqref{E: R 2}, we consider  
\begin{align}\label{E: P}
\Delta P-P +\frac{a^2\xi^2}{4}P+V(P)=0,
\end{align}
which is the perturbation of the ground state equation \eqref{R}. When $s_c=0$ (critical case), let
\begin{align}\label{E: P expansion}
P^{(n)}=R+\frac{a^2}{2}P_2+\cdots+\frac{a^{2n}}{(2n)!}P_n
\end{align}
be a sequence of approximations of \eqref{E: P} with monotonic profiles $P^{(n)}$ that obey 
\begin{align}\label{E: Pn}
\Delta P^{(n)} -P^{(n)}+\frac{a^2\xi^2}{4}P^{(n)}+V(P^{(n)})=O(a^{2n+2}).
\end{align}
For a fixed $a$, $P(0;a) \equiv P(0)$ is real, so is $P^{(n)}(0;a)$, and consequently, $P^{(n)}(\xi;a)$ is real for any $n$. The estimate for the Hamiltonian of $P^{(n)}$ is 
\begin{align}\label{E: HPn}
H(P^{(n)})=\int \left( |P^{(n)}_{\xi}|^2-\frac{1}{2\sigma+1}V(P^{(n)})+\frac{a^2\xi^2}{4}|P^{(n)}|^2 \right)\xi^{d-1}d\xi=O(a^{2n+2}).
\end{align}
Putting \eqref{E: P expansion} in \eqref{E: HPn}, the coefficients of $P^{(2)}$ produce the identity \eqref{E: R 2}. 

To prove \eqref{E: R 3}, we first obtain the following list of identities:
\begin{align}\label{E:RR}
-2\sigma\int g \left((-\Delta)^{-1}R^{2\sigma+1} \right)R^{2\sigma} \xi^{d-1}d\xi=\int R\frac{R_{\xi}}{\xi} \xi^{d-1}d\xi\,,\\
-2\int g \Delta R \xi^{d-1}d\xi=\int R^2_{\xi} \xi^{d-1}d\xi\,,\\
-\int \left( R_{\xi}^2+R^2-V(R)  \right) \ln \xi \xi^{d-1}d\xi=\int R\frac{R_{\xi}}{\xi} \xi^{d-1}d\xi\,,\\
\int \left( (1-\sigma)R_{\xi}^2+R^2-\frac{1}{2\sigma+1}V(R) \right) \ln \xi \xi^{d-1}d\xi=(1-\sigma/2) \int R^2 \xi^{d-1}d\xi.
\end{align}
The proof of the above four identities comes from Pohozaev identities and is similar to the proof in \cite[Chapter 8, Lemma 8.3]{SS1999}, replacing $R^{2\sigma+1}$ with $\left((-\Delta)^{-1}R^{2\sigma+1} \right) R^{2\sigma}$ and such. 
\end{proof}
\begin{proof}[Proof of Proposition \ref{P: a decay}]
We again follow \cite[Chapter 8, Proposition 8.1]{SS1999} using Lemma \ref{L: lemma sec4} and
recalling 
\eqref{E: P Hamiltonian}.
We only note the modifications needed in this case, the rest follows the NLS case. 
Considering $d=d(a)$ as a function of $a$, differentiating \eqref{E: P Hamiltonian} with respect to $a$, and evaluating at $a=0$ ($d(0)=2/\sigma$) gives
\begin{align}\label{E: P proposition}
\int \left( -\Delta R- \left((-\Delta)^{-1}R^{2\sigma+1} \right)R^{2\sigma}\right) \Re(p_1)\xi^{d-1}d\xi+d'(0)\int \left( R_{\xi}^2 -\frac{1}{2\sigma+1}V(R) \right)\ln \xi \xi^{d-1}d\xi=0,
\end{align}
where $\ds p_1=\frac{dP}{da}\Big|_{a=0}$ solves
$$
\Delta p_1 -p_1 +2\sigma \left((-\Delta)^{-1} R^{2\sigma+1} \right)R^{2\sigma-1}p_1+(2\sigma+1)  \left((-\Delta)^{-1} R^{2\sigma}p_1 \right)R^{2\sigma}=-d'(0)\frac{R_{\xi}}{\xi}.
$$
For the rest of the proof, replace the term $R^{2\sigma+1}$ with $\left((-\Delta)^{-1}R^{2\sigma+1} \right)R^{2\sigma}$, then the term $R^{2\sigma+2}$ with $\left((-\Delta)^{-1}R^{2\sigma+1} \right)R^{2\sigma+1}$ and the linearization term $(2\sigma+1)R^{2\sigma}g$ with $2\sigma \left((-\Delta)^{-1} R^{2\sigma+1} \right)R^{2\sigma-1}g+(2\sigma+1)  \left((-\Delta)^{-1} R^{2\sigma}g \right)R^{2\sigma}$.
\end{proof}

\subsubsection{Convergence of profiles as $a \to 0$, or a non-uniform limit} We are now ready to conclude the convergence of profiles $Q$ in the slightly $L^2$-supercritical case to the profiles $R$ of the $L^2$-critical case as $a \to 0$, recalling that the 
ground state $H^1$ solutions $R$ have zero Hamiltonian and could be obtained via minimization as described in \cite{KR2019}, in particular, the value $\|R\|_{L^2}$ is uniquely defined. 
The proof of the following statement is verbatum (with modifications in the potential term as above) of \cite[Prop.8.4]{SS1999}

\begin{proposition}
When $a \to 0^+$, and hence, $d(a) \searrow d(0)=\frac2{\sigma}$, admissible solutions 
$Q$ of \eqref{Q eqn} satisfy
\begin{itemize}
\item[(i)] For $a \xi \ll 1$, $Q\approx R(\xi)e^{-ia\xi^2/4}$; equivalently, the solution $P$ of \eqref{P eqn} approaches the ground state $R$.

\item[(ii)] For $a \xi \gg 1$, $Q \approx \mu \, \xi^{-\frac1{\sigma}-\frac{i}{a}}$ with $\mu^2=d-\frac{2}{\sigma} \|R\|^2_{L^2}$.
\item[(iii)] The asymptotic behavior of $d(a)$, or correspondingly, $\nu(a) = a \left(\frac{d}2-\frac1{\sigma}\right)$, is given by
\begin{align}\label{as d}
\qquad d(a) \approx \frac2{\sigma} + \dfrac{2\, \nu_0^2}{a \|R\|^2_{L^2}} \, e^{-\pi /a},
\quad \mbox{or} ~~ \nu(a) \approx \frac{\nu_0^2}{\|R\|^2_{L^2}} e^{-\pi/a}, 
\quad \mbox{with} \quad \nu_0=\lim_{\xi \rightarrow \infty} \xi^{\frac{d-1}{2}}e^{\xi}R(\xi).
\end{align}
\end{itemize}
\end{proposition}

Finally, we describe the construction of the {\it log-log} blow-up solutions.

\subsubsection{Construction of asymptotic solutions}
We return to the equation \eqref{RgHartree} and follow the argument in \cite[Chapter 8.2]{SS1999}. The change of variables with the quadratic phase $v=e^{i\tau-ia\xi^2/4}w$ gives
\begin{align}\label{R b}
iw_{\tau}+\Delta w -w +\dfrac{1}{4}b(\tau) \xi^2 w + \left((-\Delta)^{-1}|w|^{2 \sigma +1} \right)|w|^{2 \sigma -1} w=0
\end{align}
where
\begin{align}\label{b}
b(\tau)=a^2+a_{\tau}=-L^3L_{tt},
\end{align}
with the parameter $L$ defined in Section \ref{S:profiles} ($a = -L  L_t$).
We would like to obtain explicit dependence of $b$ on $\tau$ as $\tau \to \infty$. 
Writing
$$
w(\xi,\tau)=P(\xi,b(\tau))+W(\xi,\tau) \quad \mbox{with} \quad  W \ll P,
$$ 
we have $P$ solve
\begin{align}\label{Pb}
\begin{cases}
\Delta P -P +\dfrac{b\xi^2}{4}P-i\nu(\sqrt{b} \,)P+ \left((-\Delta)^{-1}|P|^{2 \sigma +1} \right)|P|^{2 \sigma -1} P=0, \\
 P_{\xi}(0)=0, \quad P(0) \:\:\textrm{is  real}, \quad P(\xi)=0 ~~ \mbox{as} ~~ \xi \rightarrow \infty,
 \end{cases}
\end{align}
and also satisfy the {\it finite} Hamiltonian condition
\begin{align}\label{zero P}
\int_{\mathbb{R}^d}\left( |P_{\xi}|^2-\frac{1}{2(2 \sigma +1)} \left((-\Delta)^{-1}|P|^{2 \sigma +1} \right)|P|^{2 \sigma +1} +\sqrt{b}\Im\left( \xi P \bar{P}_{\xi} \right) +\frac{b}{4}\xi^2|P|^2 \right) d\xi=const.
\end{align}
Considering $a_\tau$ to be of the lower order than $a^2$ in \eqref{b} (similar to the NLS), we have $a \approx \sqrt b$ (which is also confirmed later), and therefore, the parameter $\nu$ in \eqref{Pb} is approximated for large times $\tau$ as 
\begin{align}\label{nu}
\nu(\sqrt{b}) \approx \frac{\nu^2_0}{\|R\|^2_{L^2}}e^{-\pi/\sqrt{b}}. 
\end{align}

The following proposition determines $b$ as a function of $\tau$ from the condition that the decomposition $w(\xi,\tau)=P(\xi,b(\tau))+W(\xi,\tau)$ is an asymptotic solution of \eqref{R b}.
\begin{proposition}
Collapsing solutions of the generalized Hartree equation in the critical case 
$p= \frac{4}{d}+1$, or equivalently $\sigma=2/d$, near a singularity have the asymptotic form 
\begin{align}\label{asymp sol}
u(x,t)\approx \frac{1}{L(t)}e^{i(\tau(t)-a(t)\frac{|x|^2}{4L^2(t)})}P\left(\frac{|x|}{L(t)},b(t)\right),
\end{align}
where
\begin{align}\label{reduced eqn1}
\tau_t=L^{-2},\qquad -LL_t=a, \qquad L^3L_{tt}=-b,
\end{align}
and $b=a^2+a_{\tau} \approx a^2$ obeys
\begin{align}\label{btau}
b_{\tau}=-\frac{2 \|R \|^2_{L^2}}{M} \, \nu(\sqrt{b})\approx -\frac{2\nu_0^2}{M}e^{-\pi/\sqrt{b}},
\end{align}
where $M=\frac{1}{4}\int_{\mathbb{R}^d}R^2\xi^2 d\xi$ is the momentum.
\end{proposition}
The proof of this proposition is verbatim adapted from \cite[Proposition 8.5]{SS1999} as the only difference is in the nonlinear term, which plays no role in the analysis here.

We state the next two propositions about the {\it log-log} law and its range, omitting the proofs as the nonlinearity no longer affects them. 
\begin{proposition}
The leading order in the expansion for $a(\tau)$ as $\tau \rightarrow \infty$ is 
\begin{align}
a(\tau)\approx b^{1/2}\approx \frac{\pi}{\ln \tau}.
\end{align}
The corresponding scaling factor $L(t)$ has the asymptotic form
\begin{align}\label{L loglog}
L(t) \approx \left( \dfrac{2\pi(T-t)}{\ln\ln \frac{1}{T-t}}\right)^{1/2}.
\end{align}
In addition,
\begin{align}\label{a t tau}
\tau(t) \approx \frac{1}{2\pi}\ln \left( \frac{1}{T-t}\right)\ln\ln\left( \frac{1}{T-t}\right).
\end{align}
\end{proposition}

\begin{proposition}
The asymptotic form of the solution given in Proposition 3.5 extends in the range $0<r<r_{out}$, where
\begin{align}
r_{out} \approx 1/\sqrt{b} \sim \ln\ln \left(\frac{1}{T-t}\right).
\end{align}
\end{proposition}

\subsection{Adiabatic regime}\label{S:Adiabatic}
From the asymptotic analysis in Section \ref{AA} and fitting analysis in Section \ref{Fitting}, the blow-up rate follows the {\it log-log} regime at the {\it very high} focusing, which is currently impossible to observe numerically. The fittings in Section \ref{Fitting} indicate that there may be other laws for the blow-up rate before reaching the {\it log-log} level. In the NLS equation, the solution reaches the {\it adiabatic} regime, which can be numerically observed, before finally settling into the {\it log-log} regime. In this section, we show that the gHartree equation also has the {\it adiabatic} regime. 

Recalling (\ref{reduced eqn1}) and (\ref{btau}), we have
\begin{align}\label{reduced eqn2}
L_{tt}=-L^{-3}{b}, \qquad b_{\tau}=-\nu(\sqrt{b})\quad\mbox{with} \quad \nu(\sqrt{b})=-c_{\nu}e^{-\pi/\sqrt{b}},  
\end{align}
where $c_{\nu}$ is a positive constant. The equations in (\ref{reduced eqn2}) are called the reduced equations in the $L^2$-critical NLS equation, see \cite[Chapter 17-18]{F2015}. The NLS analysis gives, for example, two adiabatic laws:
\begin{align}\label{Malkin law}
L(t) = \sqrt{2\sqrt{b}(T-t)} \qquad (\mathrm{Malkin \, Law}),
\end{align}
and 
\begin{align}\label{Fibich law}
\qquad L(t) \approx \sqrt{2\sqrt{b}(T-t)+C(t)(T-t)^2}, \quad C(t)=\dfrac{a^2-b}{L^2}=-\dfrac{a_{\tau}}{L^2}\quad (\mathrm{Fibich \, Law}).
\end{align}
We next show how well the numerical solution matches these two adiabatic laws. 
The parameter $b=a^2+a_{\tau}$ is obtained from calculating the value $a_{\tau}$ by the fourth order backward difference (higher or lower order of finite difference method can also be applied but we found that they do not make much difference). Suppose that the rate $L(t)$ is the blow-up rate from the computational simulation, and $F(t)$ is the predicted rate (Malkin law or Fibich law). We show how the relative error 
\begin{align}\label{relative error}
\mathcal{E}_r=\left| \frac{L(t)}{F(t)}-1 \right|
\end{align}
changes as the time $t\rightarrow T$. For comparison, we also show the relative error for the {\it log-log} law and the $\gamma$-law with $\gamma=1$, see Figure \ref{adiabatic error}. As in the NLS, we take the constant equals to $2\pi$ in the $\gamma$-law, i.e., 
$$
L(t)\approx \sqrt{\dfrac{T-t}{2\pi \ln \frac{1}{T-t}}}.
$$
To test the consistency, we report the relative error for the 2d NLS equation case in Figure \ref{adiabatic error NLS}, which is similar to the plot in \cite[Fig. 18.4]{F2015}, indicating that our numerical method is trustful. On the right subplot in Figure \ref{adiabatic error NLS} we show the numerical error (also for the NLS) in the 4d case.  
\begin{figure}[ht]
\begin{center}
\includegraphics[width=0.45\textwidth]{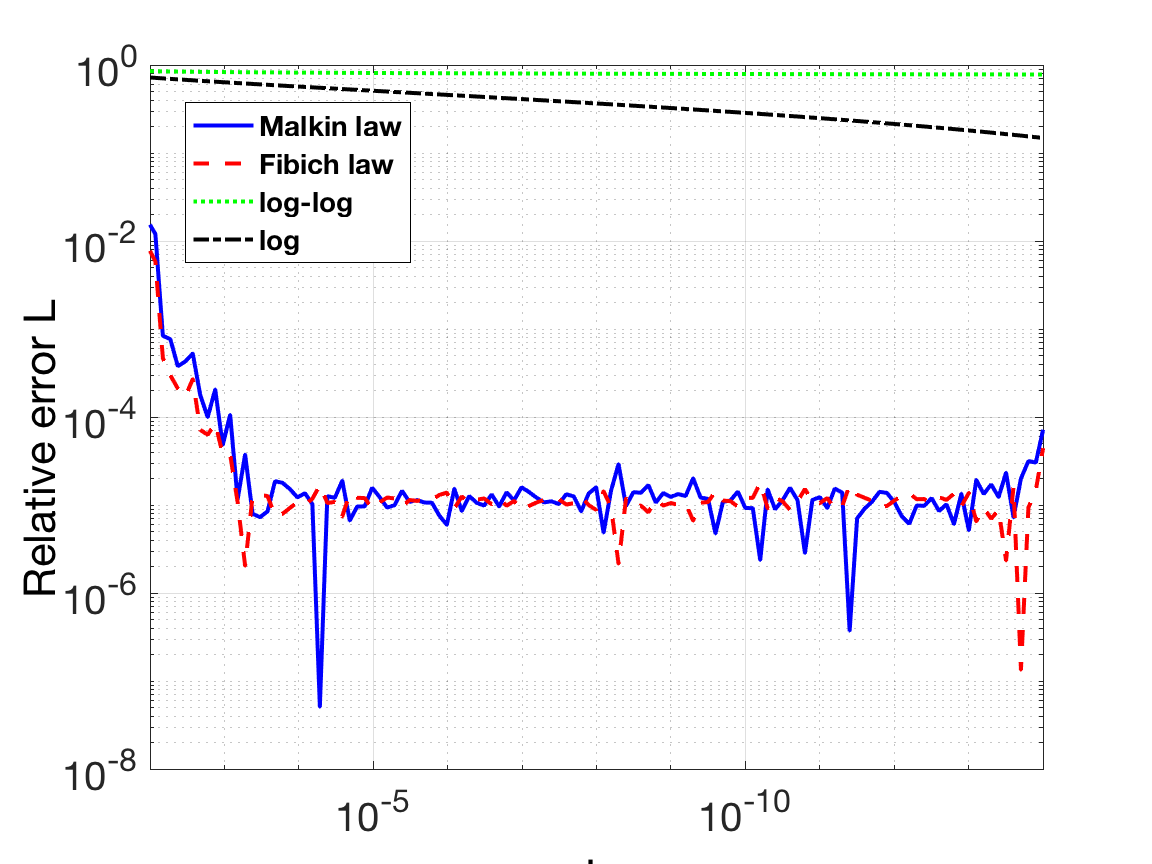}
\includegraphics[width=0.45\textwidth]{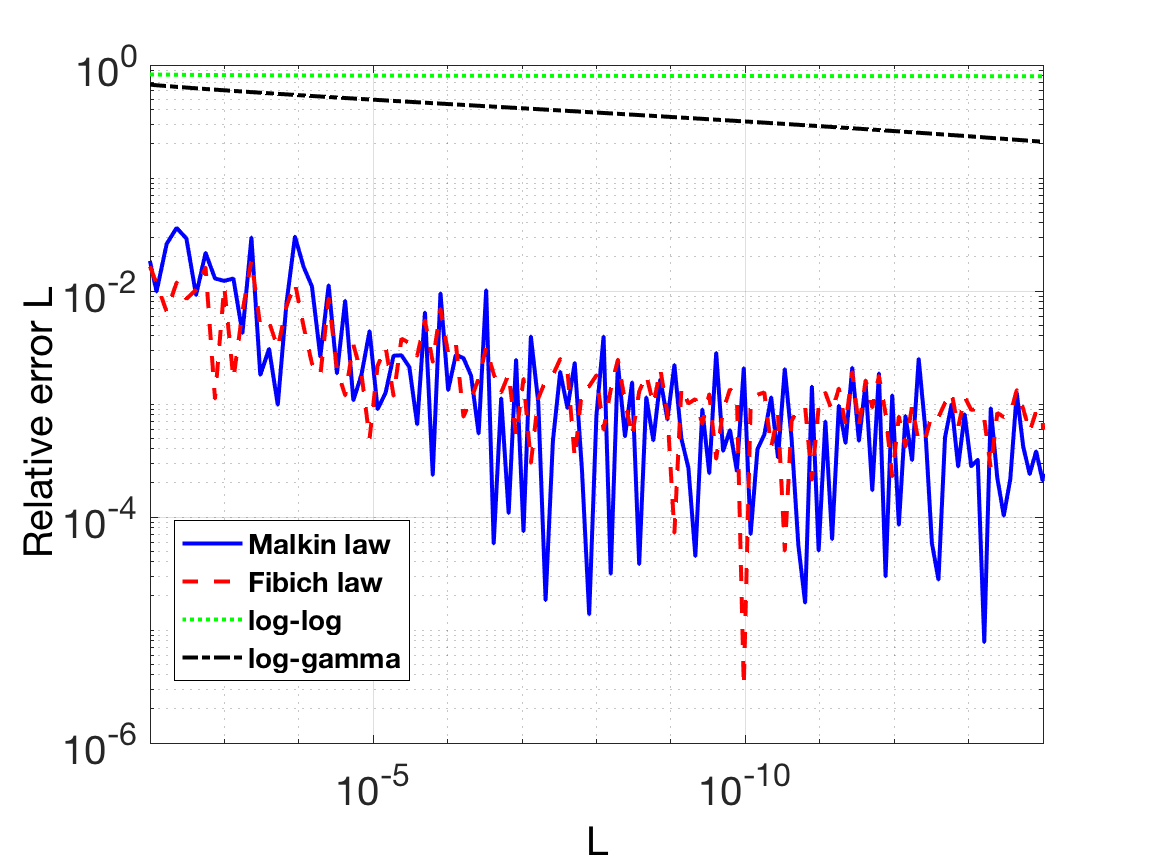}
\caption{ Consistency check: the relative error for two adiabatic laws compared to the {\it log-log} and $\gamma$-laws for the $L^2$-critical NLS equation in 2d and 4d. Left: 2d case with $u_0=2.77e^{-r^2}$, this plot is similar to the right plot in \cite[Fig. 18.4]{F2015}). Right: 4d case with $u_0=4e^{-r^2}$.} 
\label{adiabatic error NLS}
\end{center}
\end{figure}

\begin{figure}[ht]
\begin{center}
\includegraphics[width=0.49\textwidth]{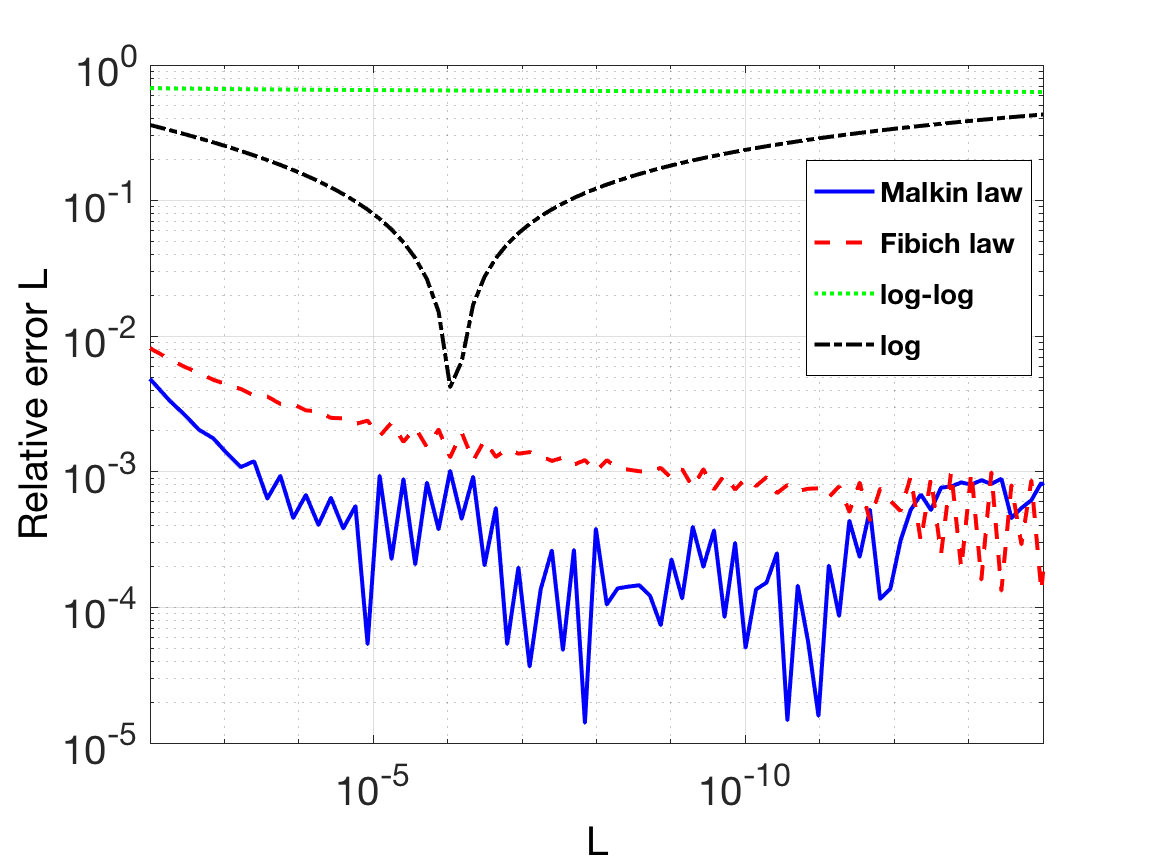}
\includegraphics[width=0.49\textwidth]{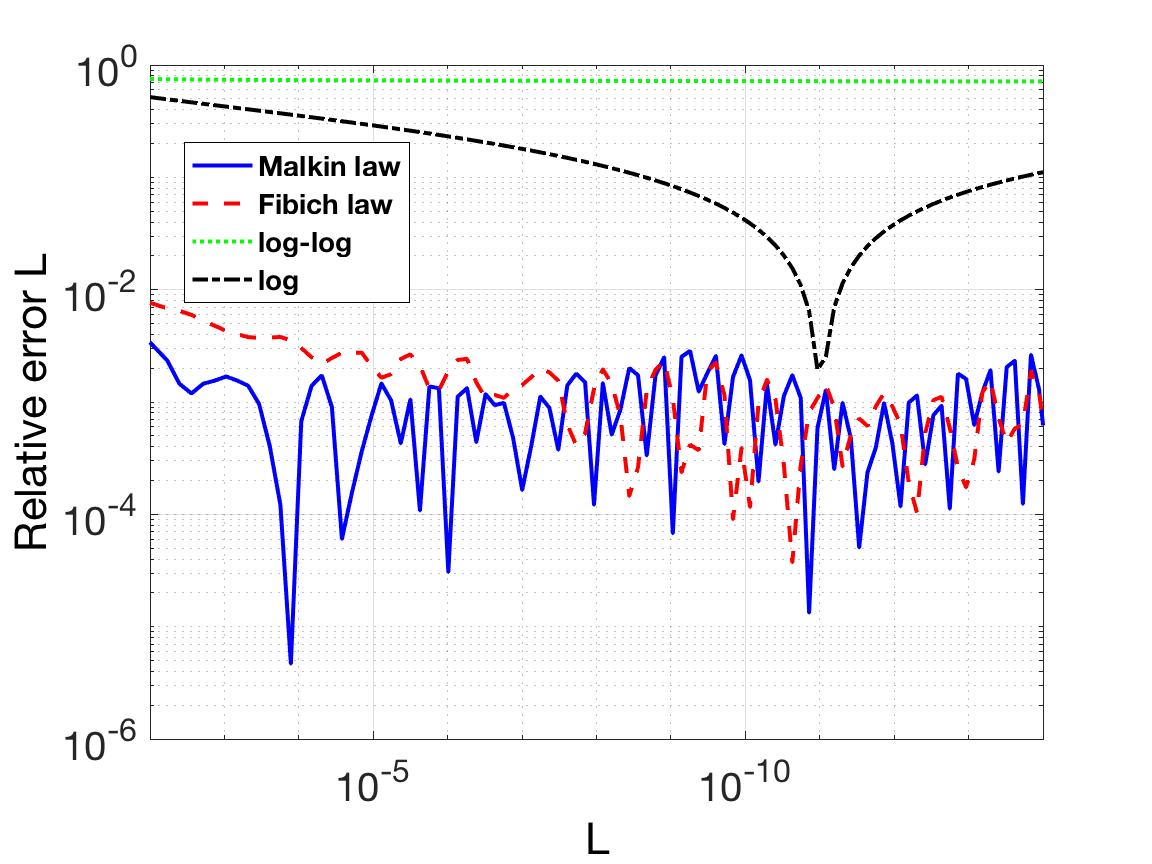}
\includegraphics[width=0.32\textwidth]{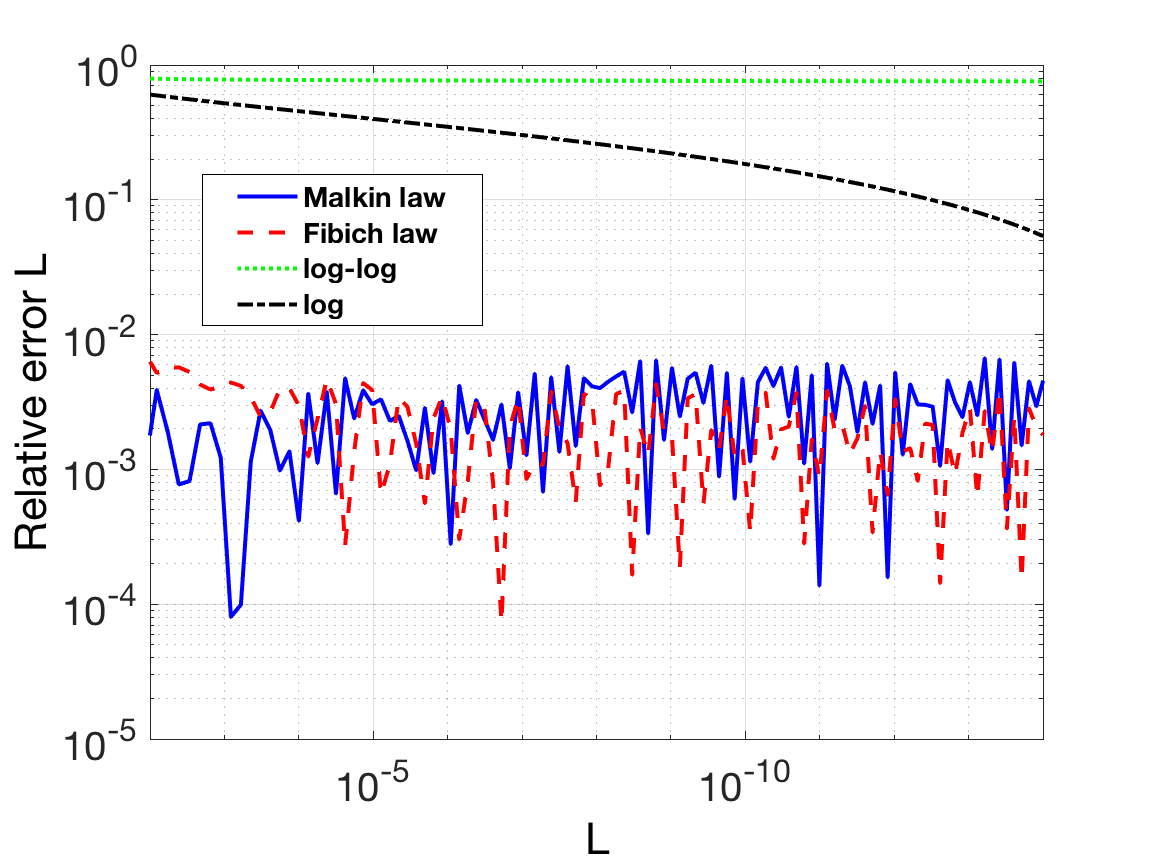}
\includegraphics[width=0.32\textwidth]{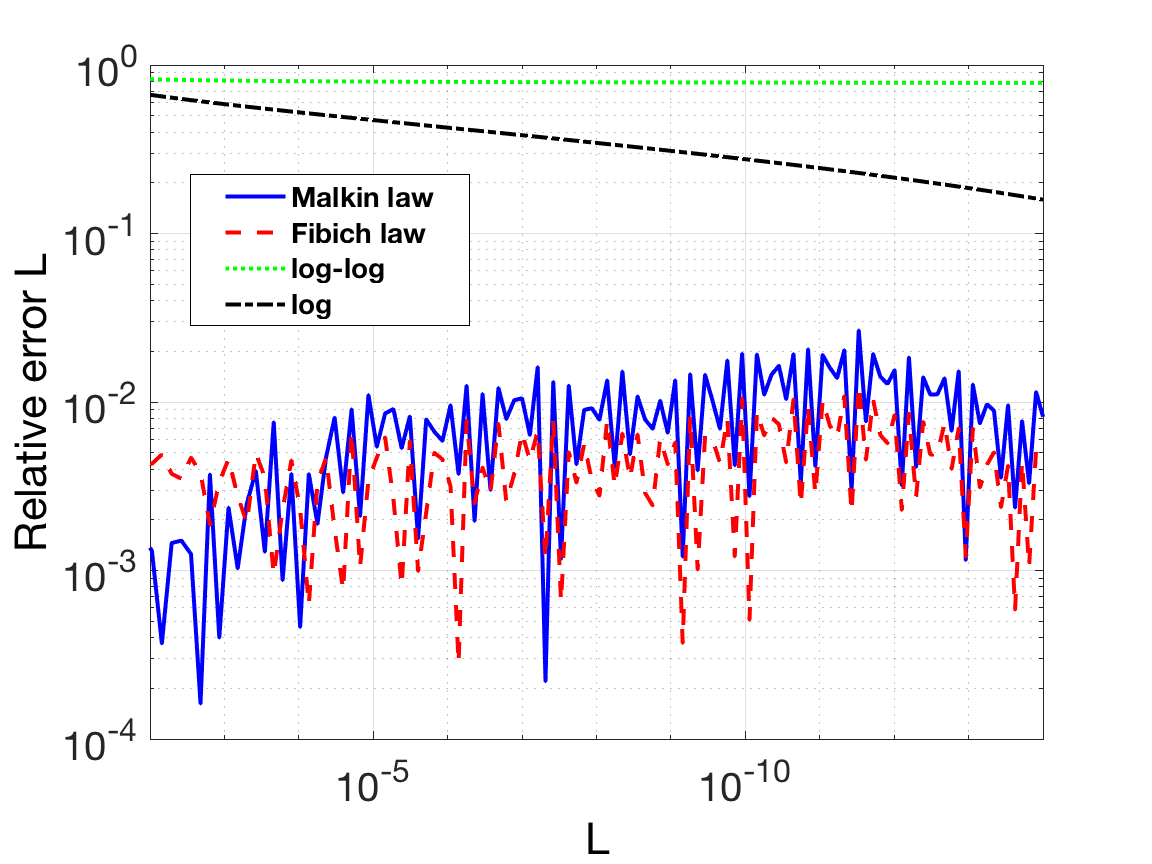}
\includegraphics[width=0.32\textwidth]{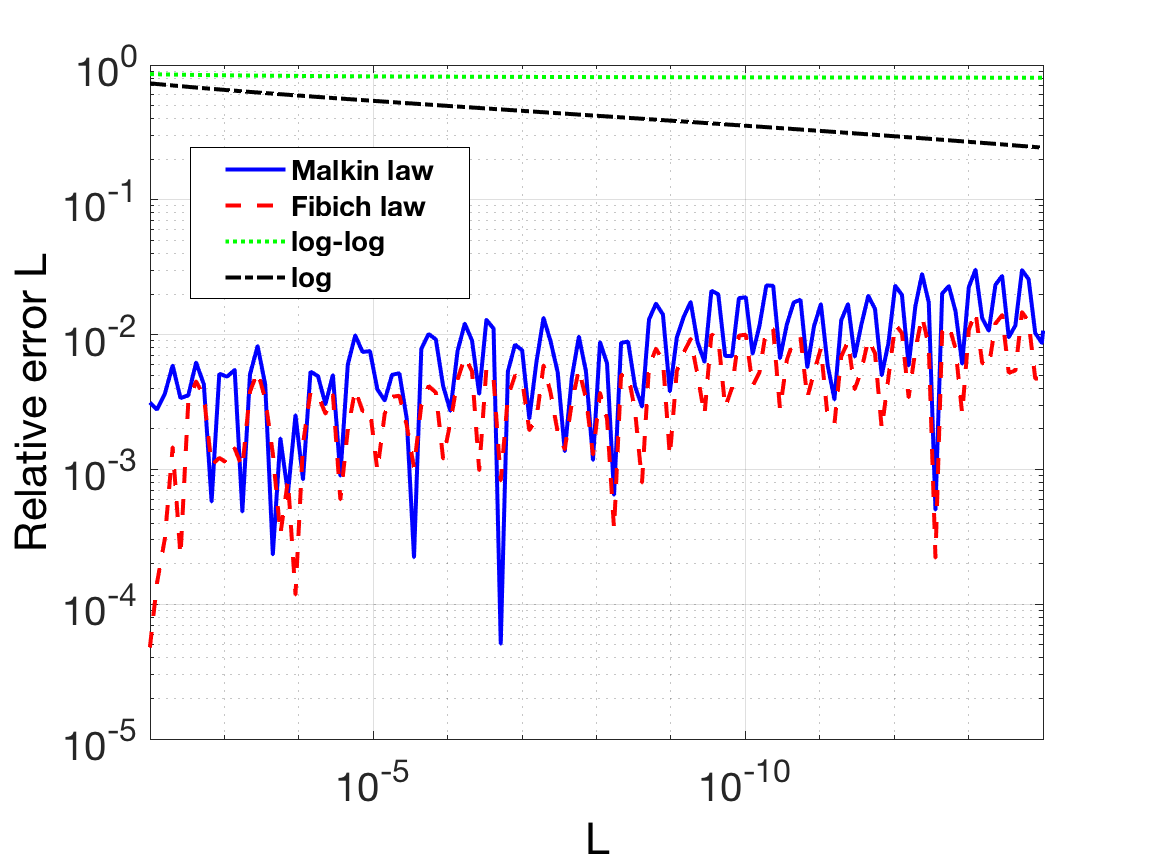}
\caption{ The relative error in gHartree case for different laws including adiabatic regimes, in dimensions $d=3,4,5,6,7$.} 
\label{adiabatic error}
\end{center}
\end{figure}

From Figure \ref{adiabatic error}, we can see that adiabatic Malkin and Fibich laws are both equally good (except for the case $d=3$, where Malkin law seems to be producing a slightly smaller error). Both Malkin law and Fibich law are better than the $\gamma$-law or the {\it log-log} law, this is due to the intermediate range of focusing (the {\it log-log} regime is yet to be reached at much higher focusing level). We only show the $\gamma$-law with $\gamma = 1$, 
since this option of $\gamma$ is the best among other values in the adiabatic regime.

In dimensions $d=3$ and $d=4$, we notice that there exist ranges of focusing regime that almost coincide in terms of the relative error, where the {\it log}-law (with $\gamma=1$) is as good as the two adiabatic laws (for 3d the range $L \sim 10^{-6}$ and for 4d the range $L \sim 10^{-11}$). This supports our calculations that the adiabatic Malkin law (also possibly Fibich law) have the rates with the leading order $\ln\frac{1}{s}$ (i.e., $\gamma=1$). 

\section{The $L^2$-supercritical case}\label{Section-Supercritical}

In this section, we consider the blow-up dynamics in the $L^2$-supercritical gHartree equation. Since the existence and {\it local uniqueness} theory of self-similar profile solutions was discussed in Section \ref{S:profiles}, we now introduce our numerical method for finding such blow-up profiles. Afterwards, we simulate the blow-up solutions for several $L^2$-supercritical gHartree equations and show the results of the convergence of the stable blow-up to the specific profiles and the rate.
 
\subsection{Numerical approach to compute profiles $Q$}
We start with recalling that admissible solutions to the profile equation \eqref{Q eqn} are the ones without the fast oscillating decay in $Q=\alpha Q_1+\beta Q_2$, where $Q_1 \approx |\xi|^{-\frac{i}{a}-\frac{1}{\sigma}}, ~ Q_2 \approx e^{-\frac{ia\xi^2}{2}}|\xi|^{-\frac{i}{a}-d+\frac{1}{\sigma}}$ as $|\xi| \to \infty$,
and thus, we are looking for the solutions with $\beta=0$. Excluding $Q_2$, we note that the solution $Q$ must be linearly dependent to $Q_1$ as $\xi \to \infty$, thus, computing the Wronskian  for $Q$ and $Q_1$ gives  $\left(\frac{1}{\sigma}+\frac{i}{a} \right) Q(\xi) + \xi \,Q_{\xi}(\xi)=0$ as $\xi \to \infty$. This gives the artificial boundary condition 
\begin{equation}\label{Q eqn bc2}
\left(\frac{1}{\sigma}+\frac{i}{a} \right) Q(K) + KQ_{\xi}(K)=0
\end{equation} 
by taking sufficiently large $K$. 

We next split $Q$ into the real and imaginary parts $Q=P+iW$, rewriting \eqref{Q eqn} and \eqref{Q eqn bc2} as
\begin{align}\label{Q compute}
\begin{cases}
\Delta P-P-a(\frac{W}{\sigma}+\xi W_{\xi})+((-\Delta)^{-1}(P^2+W^2)^{\sigma+\frac12})(P^2+W^2)^{\sigma-\frac12}P=0,\\
\Delta W-W+a(\frac{P}{\sigma}+\xi P_{\xi})+((-\Delta)^{-1}(P^2+W^2)^{\sigma+\frac12})(P^2+W^2)^{\sigma-\frac12}W=0,\\
P_{\xi}(0)=0,\\
W(0)=0,\\
W_{\xi}(0)=0,\\
\frac{1}{\sigma}P-\frac{1}{a}W+KP_{\xi}=0,\\
\frac{1}{a}P+\frac{1}{\sigma}W+KW_{\xi}=0.
\end{cases}
\end{align}
We solve the equation system \eqref{Q compute} in two ways. We first use the matlab solver \textit{bvp4c}. We set $-\Delta \varphi=(P^2+W^2)^{\sigma+0.5}$ to deal with the nonlocal term. Thus, our solver will deal with a system of six equations. An alternative way to work with this system is to rewrite it into a system of nonlinear algebraic equations. Then the matlab solver \textit{fsolve} can be applied (with the algorithm option \textit{``levenberg-marquardt"} to make sure it converges).

During the computation, these two methods generate almost the same profiles. The residue shows that \textit{fsolve} is more accurate if we use $N=257$ Chebyshev-collocation points in our computations. Furthermore, both methods need a suitable initial guess. As we have previously handled NLS (see \cite{YRZ2018}), we take the solution of the following NLS boundary value problem as our initial guess:
\begin{eqnarray}\label{Q compute NLS}
(NLS)_d \qquad \qquad 
\begin{cases}
\Delta P-P-a(\frac{W}{\sigma}+\xi W_{\xi})+(P^2+W^2)^{\sigma}P=0,\\
\Delta W-W+a(\frac{P}{\sigma}+\xi P_{\xi})+(P^2+W^2)^{\sigma}W=0,\\
P_{\xi}(0)=0,\\
W(0)=0,\\
W_{\xi}(0)=0,\\
\frac{1}{\sigma}P-\frac{1}{a}W+KP_{\xi}=0,\\
\frac{1}{a}P+\frac{1}{\sigma}W+KW_{\xi}=0.
\end{cases}
\end{eqnarray}
In \cite{YRZ2019} we obtained solutions of this NLS system \eqref{Q compute NLS}, hence, we can use them as our initial guess. 

Similar to the NLS case, there are multiple solutions to the system \eqref{Q compute}, and we are able to find some of them, 
though not all these solutions are profiles for {\it stable} blow-up. Most likely they serve as profiles for the unstable blow-up solutions, but we have not verified that. 
In order to find the appropriate admissible profiles, constraints can be put either on the parameter $a$ or on the value $Q(0)$. Here, we choose to put constraints on the parameter $a$, i.e., we find the value of $a$ such that $\alpha \leq a \leq \beta$ for prescribed constants $\alpha$ and $\beta$. 
In order to put these constraints into \eqref{Q compute}, we consider a mapping $f: \mathbb{R} \rightarrow [\alpha,\beta]$ and set $$a(s)=\alpha f(s)+\beta (1-f(s)),$$ 
where $s \in \mathbb{R}$ and $f(s) \in [0,1]$. Then, we solve the equation \eqref{Q compute} by substituting  $a(s)=\alpha f(s)+\beta (1-f(s))$ with an explicitly given function $f(s)$. For example, one can take $a=\alpha \sin^2s+\beta \cos^2s$. 
The constant $a$ is reconstructed after obtaining the value $s$.

We emphasize that while we are able to put the constraints into the equation \eqref{Q compute}, we still need a relatively suitable initial guess. This issue is similar to the NLS $L^2$-supercritical case: selecting initial guess to find the profiles with no oscialltions as $\xi \to \infty$ is extremely sensitive. For example, as discussed in \cite{BCR1999} the initial guess is sensitive to 4\% difference of the actual values of $a$ and $Q(0)$ (to give convergence to the corresponding multi-bump profile). We choose the corresponding multi-bump solutions from NLS equation from \cite{YRZ2019} as the initial guess, which is suitable in the gHartree setting. 

\subsection{Admissible profiles} 
Among all admissible solutions to \eqref{Q eqn} there is no uniqueness as it was shown in \cite{BCR1999}, \cite{KL1995}, \cite{YRZ2019}. These solutions generate branches of multi-bump profiles. We label the solution $Q_{1,0}$ the first solution in the branch $Q_{1,K}$ (this is the branch, which converges to the $L^2$-critical ground state solution $R$ as $s_c \to 0$), and we consider $Q_{1,0}$ as the potential profile for stable gHartree blow-up, see the blue curve in Figure \ref{Q3d3p}.
By using another initial guess for parameters $a$ and $Q(0)$ as described above, 
we obtain the solution $Q_{1,1}$, which is the first bifurcation from $Q_{1,0}$ (see the red dashed curve in Figure \ref{Q3d3p}). 

To better understand the dependence of solutions on parameters $a$ and $Q(0)$, we study the pseudo-phase plane, which was introduced in the NLS case by Kopell and Landman in \cite{KL1995} and adopted in Budd, Chen and Russell \cite{BCR1999}. We write
\begin{align}\label{Q phase}
Q\equiv C(\xi)\exp \left( i\int_0^{\xi} \psi \right), \qquad D(\xi)={C_{\xi}}/{C}\equiv\Re({Q_{\xi}}/{Q}).
\end{align}
In other words, $C$ is the amplitude of $Q$, $C(\xi) = |Q(\xi)|$, $D$ is its logarithmic derivative, and $\psi$ is the gradient of the phase. 
{In the coordinates $(C,D)$ we track the behavior of the graph as it decreases down to the origin when both $C$ and $D$ approach zero as $\xi \to \infty$. To see that recall from \eqref{E:Q1-Q2} that asymptotically
$$
Q(\xi) \sim \alpha \, \xi^{-\frac{1}{\sigma}} \exp \left(-\frac{i}{a} \, \log(\xi) \right) + \beta \, \xi^{-(d-\frac1{\sigma})} \exp\left( -\frac{i\, a \, \xi^2}2  + \frac{i}{a} \log(\xi)\right),  
$$
where the first term is slowly decaying and the second term decays faster with rapid oscillations. The solution $Q$ that varies slowly at infinity, would have no oscillations at the end of the curve (as $C \to 0$), since 
$$
C \sim \frac{\alpha}{\xi^{1/\sigma}} \quad \mbox{and} \quad D \sim -\frac1{\sigma \,\xi} \quad \mbox{as} \quad \xi \to \infty.
$$
Thus, such solutions will approach the origin in coordinates $(C,D)$ along the curve 
$D \sim -\frac1{\sigma \, \alpha^\sigma} \, C^{\sigma}$. In the case of $\sigma=1$, this will be a straight line with slope $-1/\alpha$, which we demonstrate in the paths shown in Figures \ref{Q3d3p} and \ref{Q4d3p} (right plot). In the case of $\sigma =2$, this will be a parabola $D \sim - \frac1{2\sigma^2}c^2$, which can be seen in Figures \ref{Q3d5p} and \ref{Q4d5p} (also subplots on the right).

If the solution $Q$ oscillates fast at infinity, then its graph in the coordinates $(C,D)$ will approach the origin in the oscillating manner, since 
$$
C \sim \frac{\alpha}{\xi^{1/\sigma}} \quad \mbox{and} \quad D \sim
- \frac{\beta \, a}{\alpha} \, \frac1{\xi^{d-2/\sigma-1}} \, \sin\left(\frac{a \, \xi^2}2 - \frac2{a} \log (\xi)\right).
$$
\begin{figure}[ht]
\begin{center}
\includegraphics[width=0.45\textwidth]{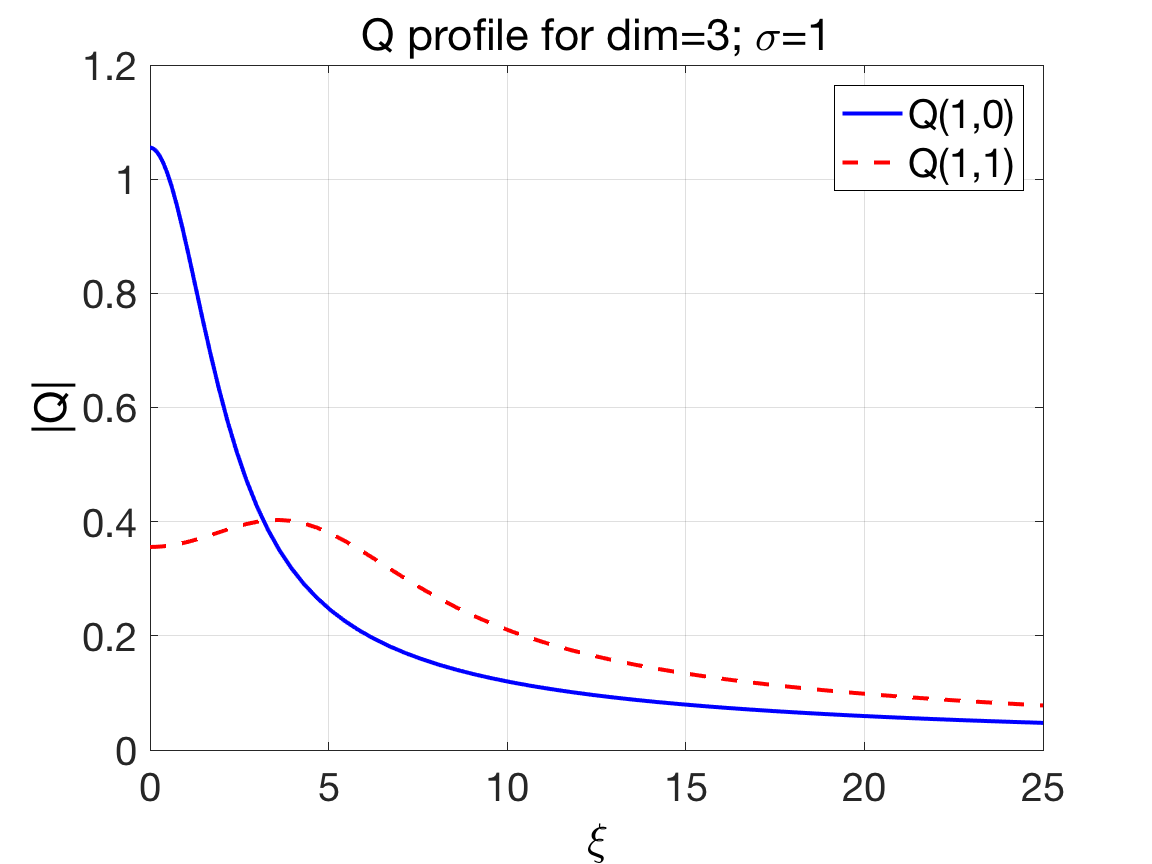}
\includegraphics[width=0.45\textwidth]{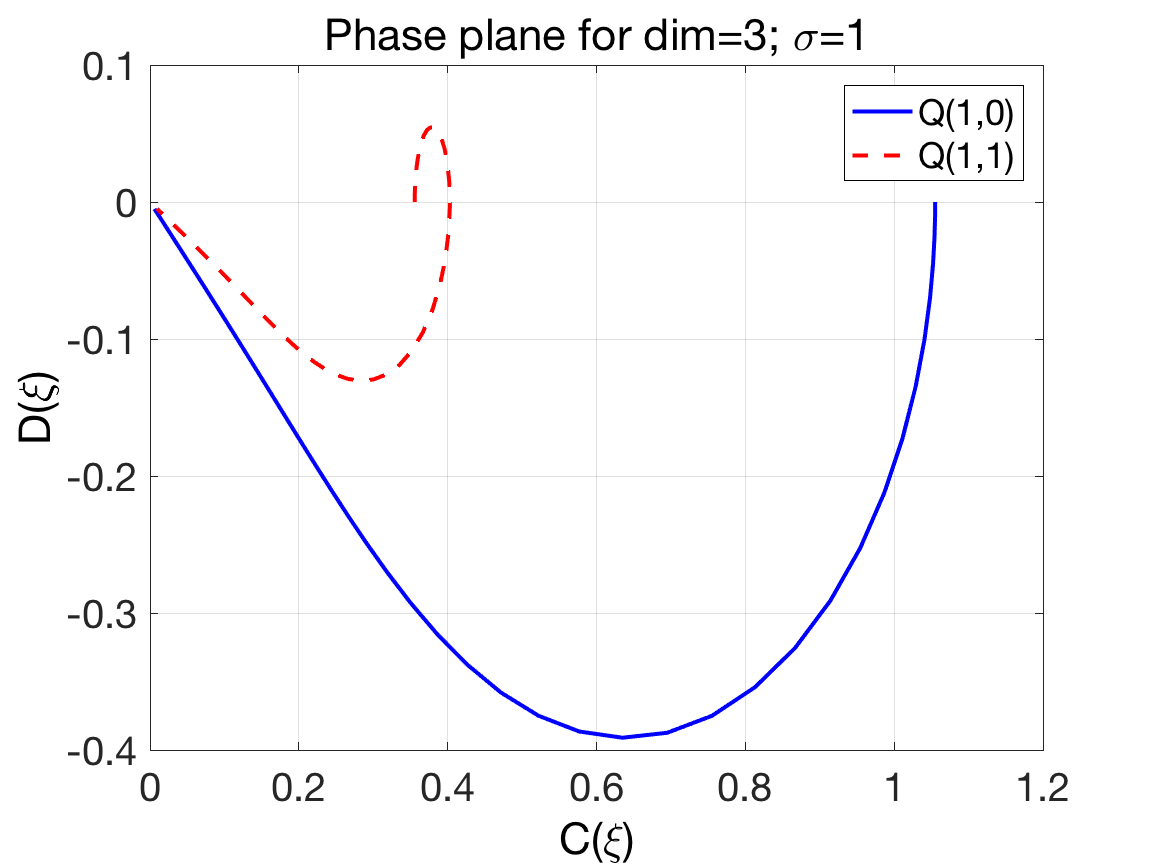}
\caption{$Q$ profiles for $d=3, \sigma=1$. Left: the monotone solution $Q_{1,0}$ (blue) and the first bifurcation solution $Q_{1,1}$ (red). Right: the phase plane $(C,D)$, here $D \sim - C$.}
\label{Q3d3p}
\end{center}
\end{figure}
\begin{figure}[ht]
\begin{center}
\includegraphics[width=0.45\textwidth]{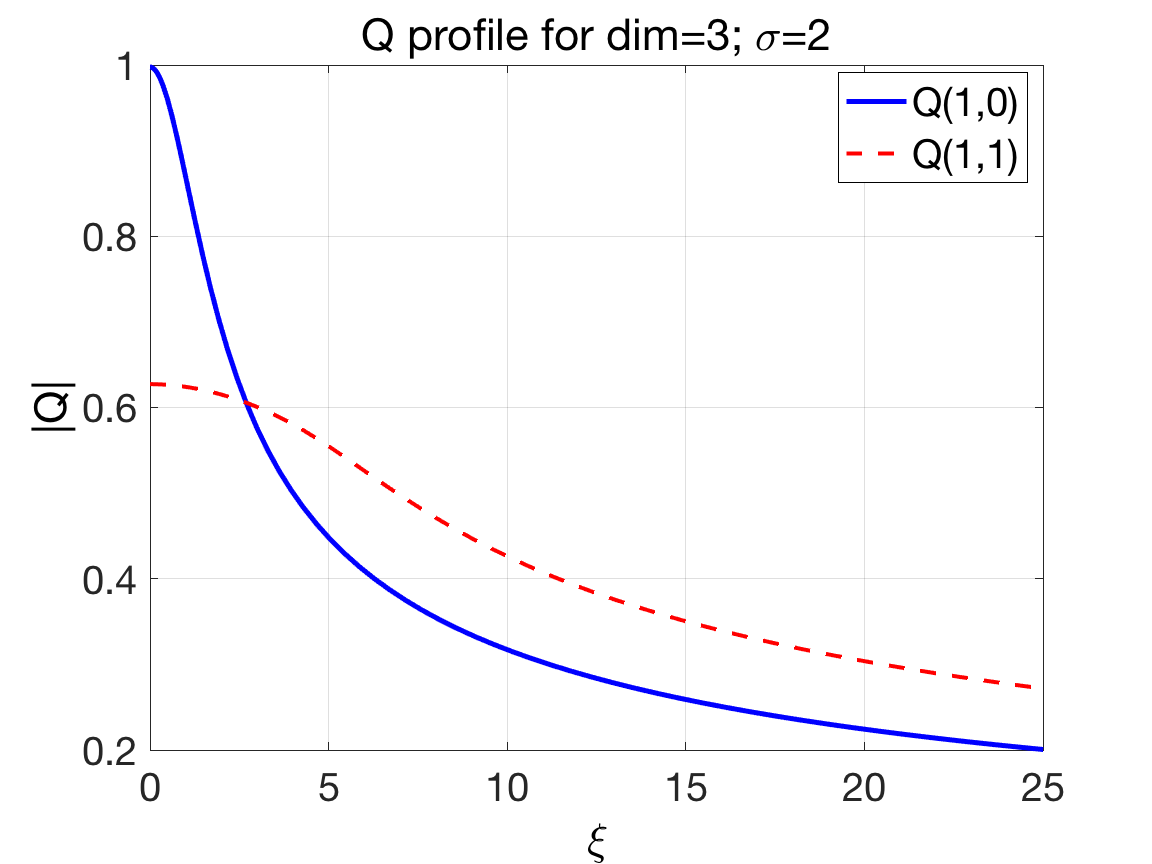}
\includegraphics[width=0.45\textwidth]{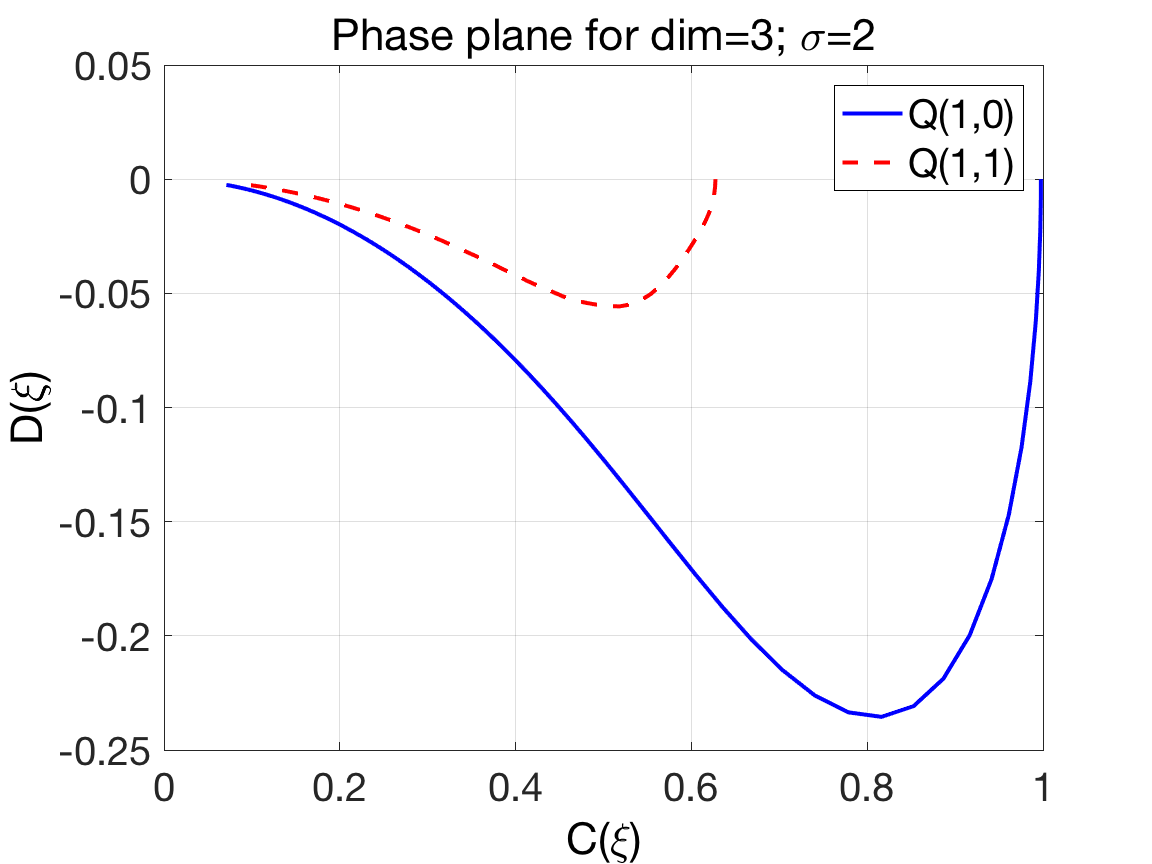}
\caption{ $Q$ profile for $d=3, \sigma=2$. Left: the monotone solution $Q_{1,0}$ (blue) and the first bifurcation solution $Q_{1,1}$ (red). 
Right: the phase plane $(C,D)$, here $D \sim - C^2$.}
\label{Q3d5p}
\end{center}
\end{figure}

Figures \ref{Q3d3p}, \ref{Q3d5p}, \ref{Q4d3p} and \ref{Q4d5p} (left subplots) are profiles of $|Q|$. 
We also show how the values of $a$ and $Q(0)$ continuously change with respect to the dimension (taking $d$ as a continuous parameter) in Figure \ref{Q_cubic} ($\sigma=1$) and Figure \ref{Q_quintic} ($\sigma=2$). Table \ref{Q-a} contains the values for $a$ and $Q(0)$ that we obtain in our simulations.

\begin{figure}
\begin{center}
\includegraphics[width=0.45\textwidth]{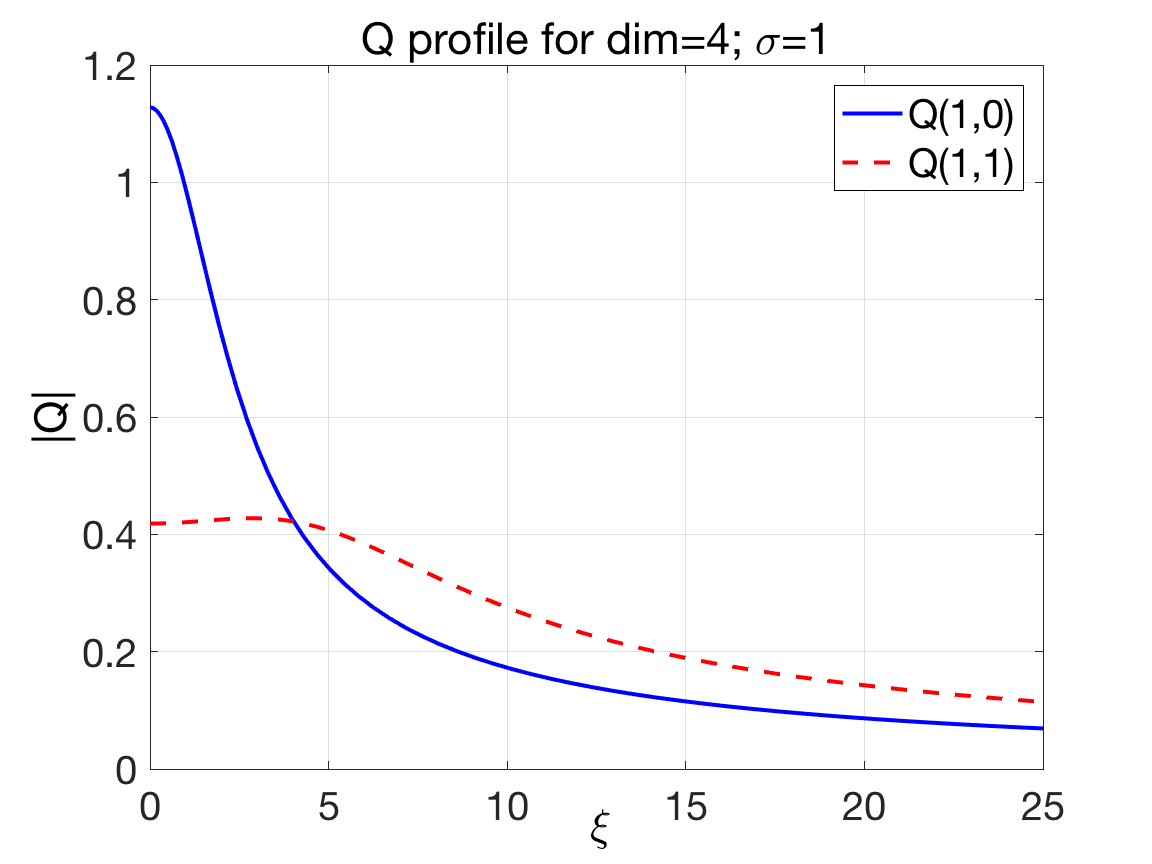}
\includegraphics[width=0.45\textwidth]{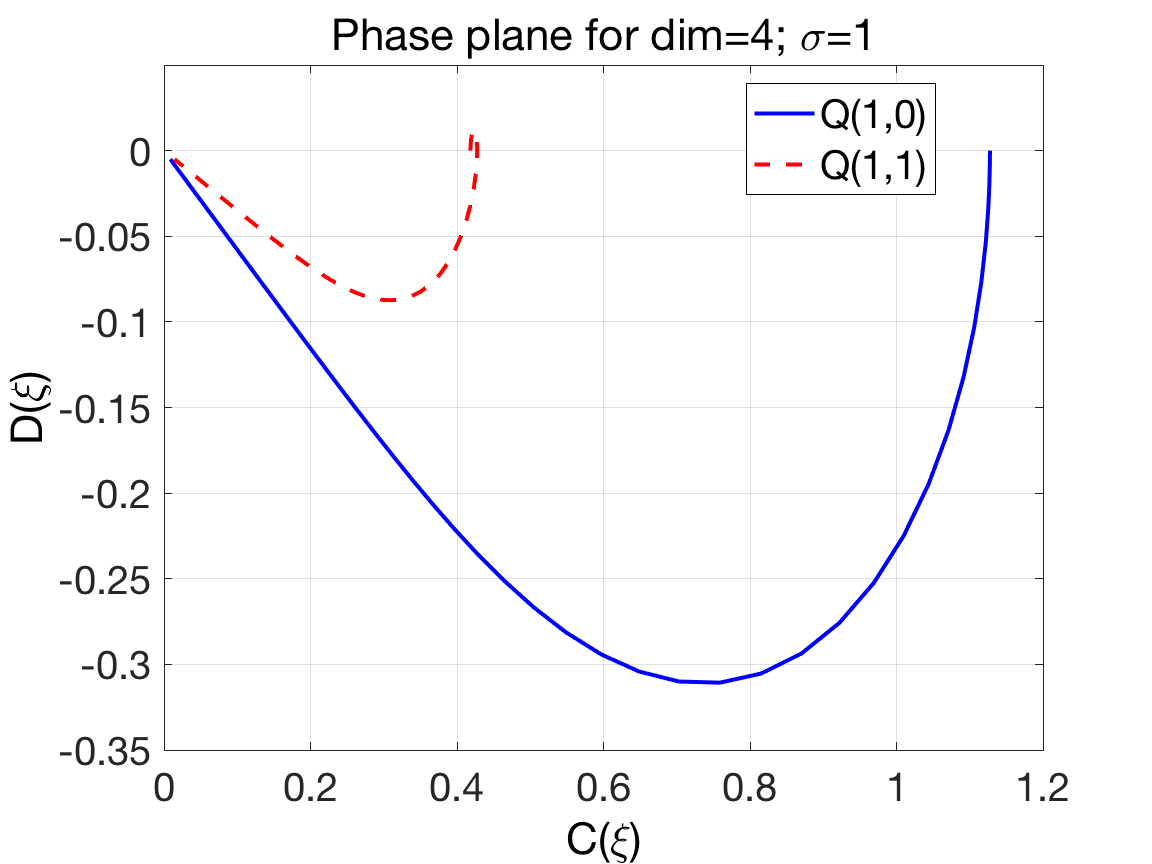}
\caption{ $Q$ profile for $d=4, \sigma=1$. Left: the monotone solution $Q_{1,0}$ (blue) and the first bifurcation solution $Q_{1,1}$ (red).
Right: the phase plane $(C,D)$, here $D \sim -C$.}
\label{Q4d3p}
\end{center}
\end{figure}
\begin{figure}
\begin{center}
\includegraphics[width=0.45\textwidth]{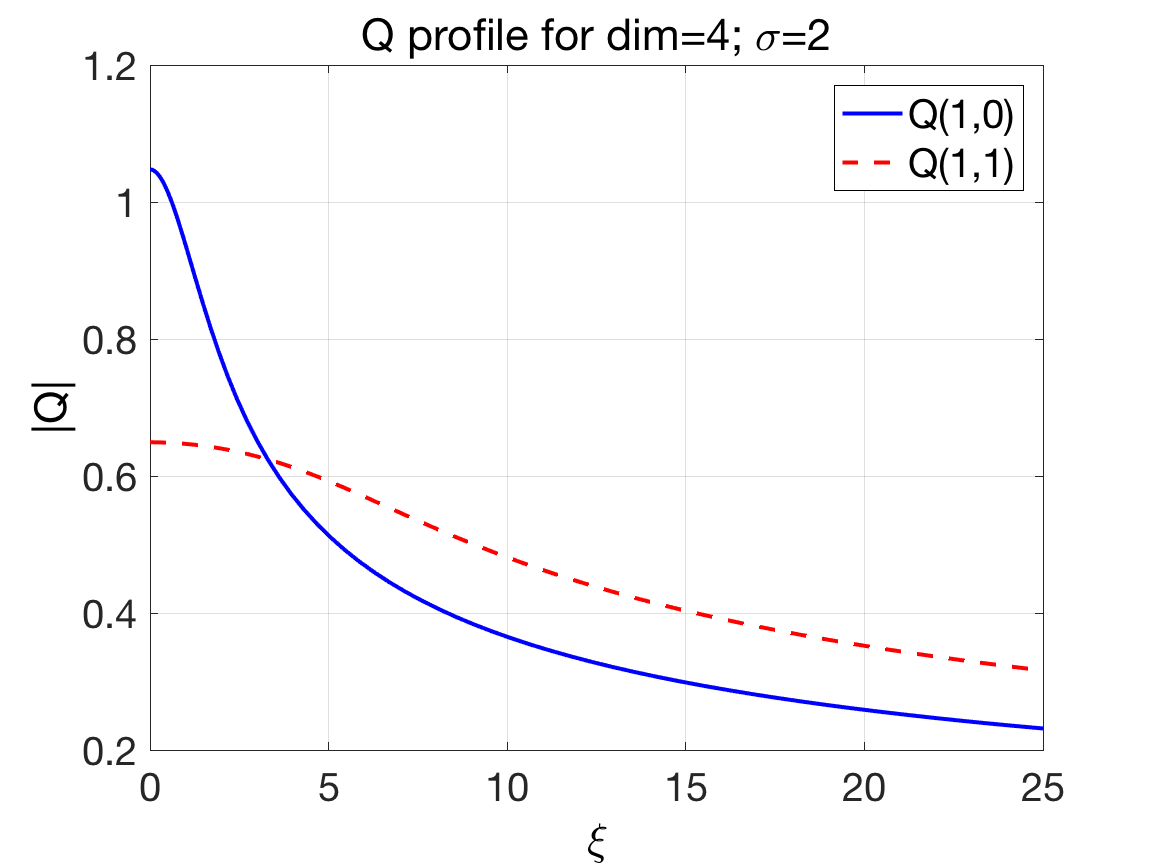}
\includegraphics[width=0.45\textwidth]{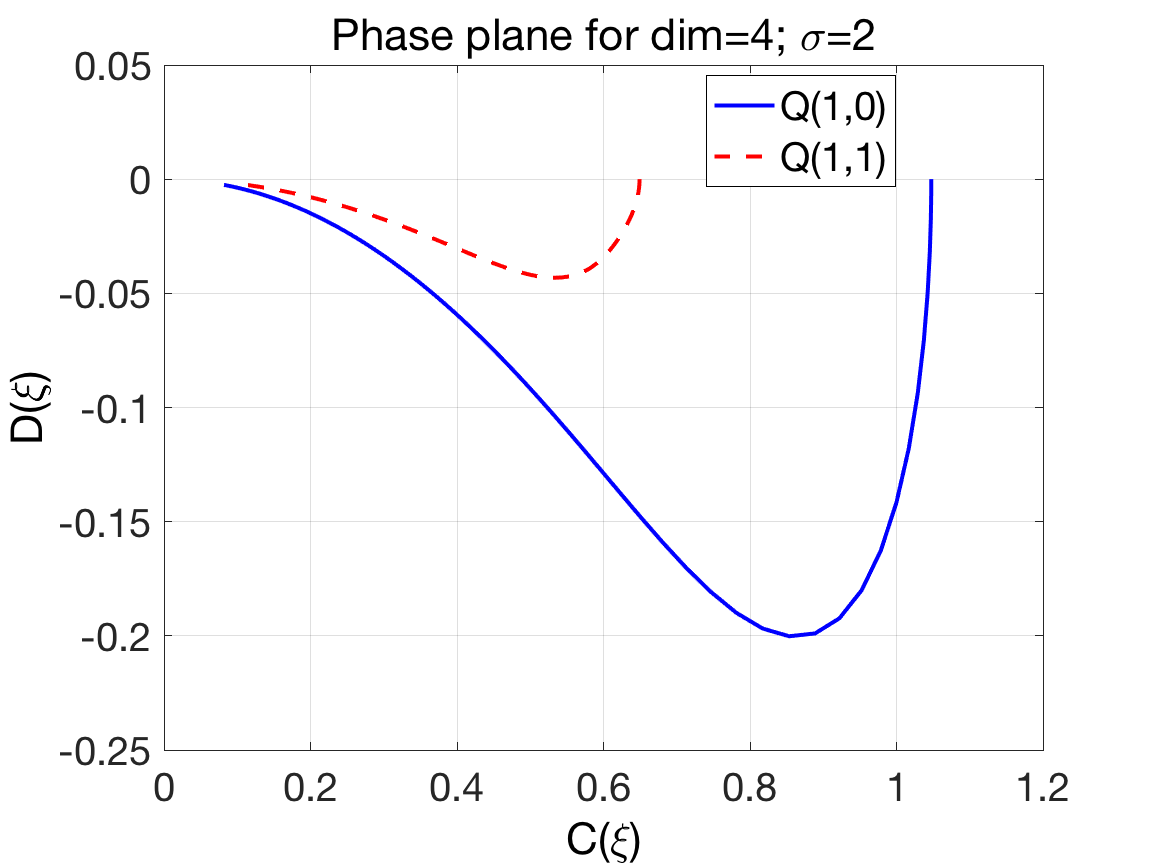}
\caption{ $Q$ profile for $d=4, \sigma=2$. Left: the monotone solution $Q_{1,0}$ (blue) and the first bifurcation solution $Q_{1,1}$ (red). 
Right: the phase plane $(C,D)$, here $D \sim -C^2$.}
\label{Q4d5p}
\end{center}
\end{figure}

{\small
\begin{table}
\begin{tabular}{|c|c|c|c|c|c|}
\hline
\multirow{2}*{$d$}& \multirow{2}*{$\sigma$}  & \multicolumn{2}{c|} {$Q(1,0)$}&\multicolumn{2}{c|} {$Q(1,1)$} \\
\cline{3-6} 
     &     & $a$ & $Q(0)$ &  $a$ & $Q(0)$ \\
\hline
 $3$ & $1$ & $0.60868$ & $1.05512$ &  $0.21180$ & $0.35569$ \\
\hline
 $3$ & $2$ & $1.40186$ & $0.99765$ &  $0.31725$ & $0.60782$ \\
\hline
 $4$ & $1$ & $0.71853$ & $1.12757$ &  $0.22045$ & $0.41841$ \\
\hline
 $4$ & $2$ & $1.57426$ & $1.04749$ &  $0.32613$ & $0.62799$ \\
\hline
\end{tabular}
\linebreak
\linebreak
\caption{The values for $a$ and $Q(0)$ for the monotone solution $Q(1,0)$ and first bifurcation solution $Q(1,1)$. }
\label{Q-a}
\end{table}
}
\begin{figure}
\begin{center}
\includegraphics[width=0.43\textwidth]{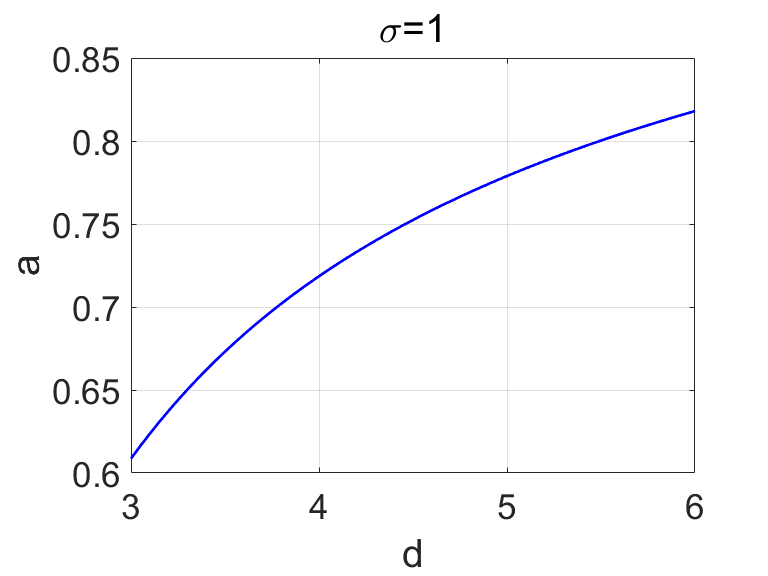}
\includegraphics[width=0.43\textwidth]{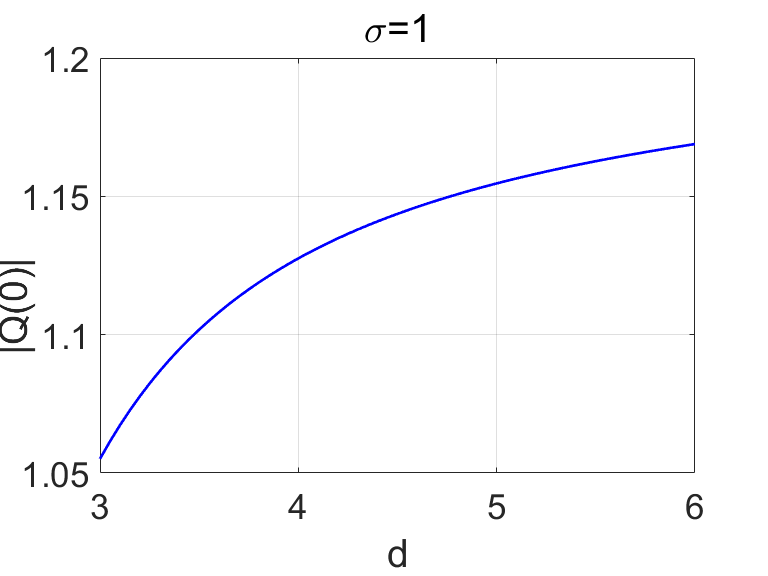}
\caption{ The change of $a$ and $Q(0)$ with respect to the dimension $d$ for $\sigma=1$.}
\label{Q_cubic}
\end{center}
\end{figure}
\begin{figure}
\begin{center}
\includegraphics[width=0.43\textwidth]{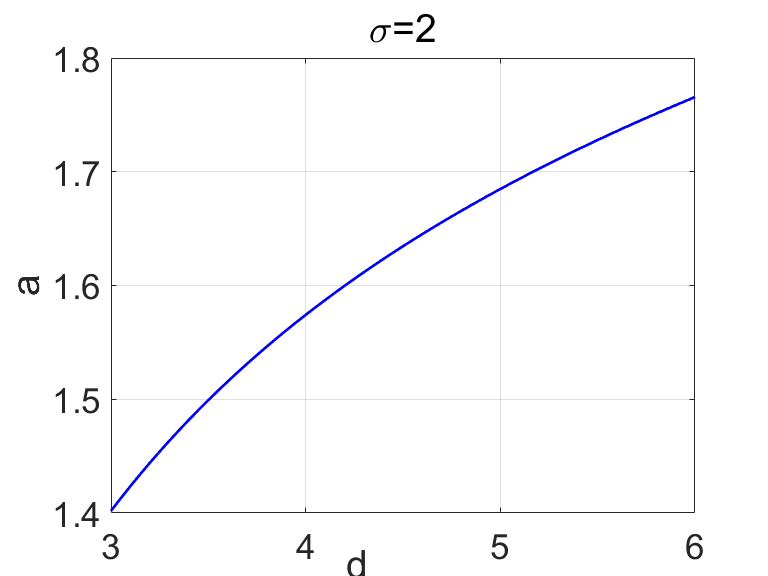}
\includegraphics[width=0.43\textwidth]{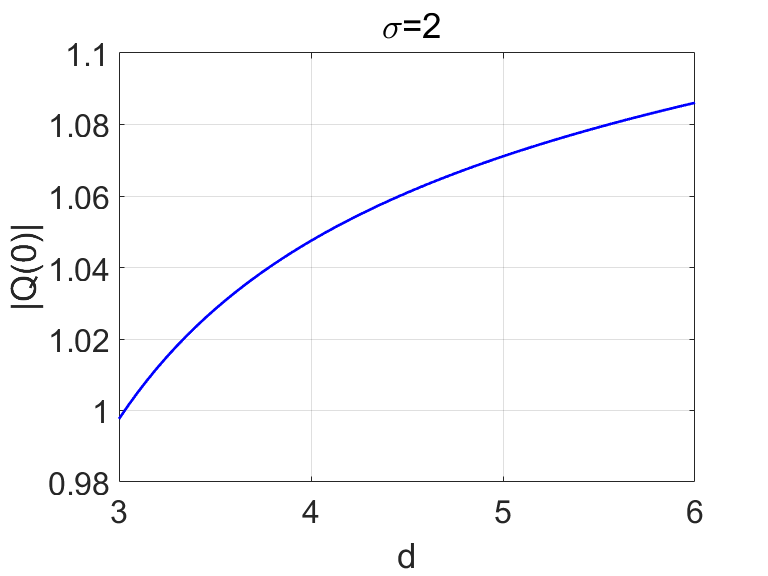}
\caption{ The change of $a$ and $Q(0)$ with respect to the dimension $d$ for $\sigma=2$.}
\label{Q_quintic}
\end{center}
\end{figure}

\subsection{Direct simulation of the blow-up dynamics}
We simulate the blow-up dynamics of the $L^2$-supercritical gHartree equation in the following cases:
\begin{itemize}
\item $3d$ $\sigma=1$ ($s_c=\frac{1}{2}$ - energy-subcritical);
\item $3d$ $\sigma=2$ ($s_c=1$ - energy-critical);
\item $4d$ $\sigma=1$ ($s_c=1$ - energy critical) and
\item $4d$ $\sigma=2$ ($s_c=\frac{3}{2}$ - energy-supercritical). 
\end{itemize}

We take the initial data $5e^{-r^2}$ for the nonlinearity $\sigma=1$, and $2.5e^{-r^2}$ for the nonlinearity $\sigma=2$. Such initial data lead to the negative energy for the case $d=3$ and positive energy for $d=4$.

The numerical results shown for the super-critical case are computed by the finite difference discretization described in Section \ref{Section-DR}. 
We terminate our simulation when $L(t)<10^{-24}$, though for clarity most of the results are presented only up to $L(t) \sim 10^{-20}$.

For the $L^2$-supercritical case ($s_c>0$), it is easy to follow the analysis for $a(\tau)$ in \cite{LePSS1988} (see also \cite{F2015}) to obtain the blow-up rate, since again, the nonlinear term plays no role in the asymptotic analysis. Thus, the blow-up rate is predicted to be
\begin{align}\label{gHartree blowup rate}
L_{\mathrm{pred}}(t) \approx (2a(T-t))^{\frac{1}{2}}.
\end{align}

Notice that if $Q(\xi)$ is the profile obtained when solving \eqref{Q eqn}, and $\tilde{Q}(\eta)$ is another profile with $\|\tilde{Q}\|_{L^{\infty}}= |v_0(0)|$ (e.g., we know that $Q(0)\approx 1.05$ for $\d=3$ and $\sigma=1$ from Table \ref{Q-a} but we set $\|\tilde{Q}\|_{L^{\infty}}=|v_0(0)|=1$ in our numerical simulation of the blow-up dynamics), then from \eqref{self-similar sol}, we have a family of $Q$ profiles
\begin{align}\label{Q rescale}
Q(\xi)= \left(\dfrac{Q(0)}{\tilde{Q}(0)} \right) \tilde{Q} \left( \xi \left( \dfrac{Q(0)}{\tilde{Q}(0)} \right)^{{\sigma}} \right).
\end{align} 
Consequently, the corresponding rescaled $\tilde{a}$ satisfies
\begin{align}\label{a rescale}
\tilde{a}=a\left[\dfrac{|v_0(0)|}{Q(0)} \right]^{2{\sigma}}.
\end{align} 
For simplicity, we still use $Q$ to represent the family of $Q$ profiles, adding ``up to scaling". 

In Figures \ref{Profi 3d5p} - \ref{4d5p data} we provide the following results from our simulations: blow-up profiles, blow-up rate $\ln(L)$ vs. $\ln(T-t)$, the value of $a(\tau)$ depending on time $\tau$, the distance between $Q$ and $v$ in time $\tau$, i.e., $$\| v(\tau) -Q \|_{L^{\infty}_{\xi}},$$ and the relative error between the numerical results and the predicted rate, i.e., $$\mathcal{E}_{rel} = \left| \left(\dfrac{L(t)}{\sqrt{2\tilde{a} (T-t)}}\right)^{\frac{1}{\sigma}}-1 \right|,$$
where $\tilde{a}=a(\tau_{\mathrm{end}})$ is the value of $a$ when we terminate our numerical simulation. In the numerical computation, the relative error $\mathcal{E}_{rel}$ is actually calculated by
$$\mathcal{E}_{rel} = \left| \exp\left(\dfrac{1}{2\sigma}(2\ln(L)-\ln(T-t)-\ln2 -\ln \tilde{a}) \right) -1 \right|$$
to make every term moderate (not to large or small), and thus, increase the accuracy.
\begin{figure}[ht]
\begin{center}
\includegraphics[width=0.32\textwidth]{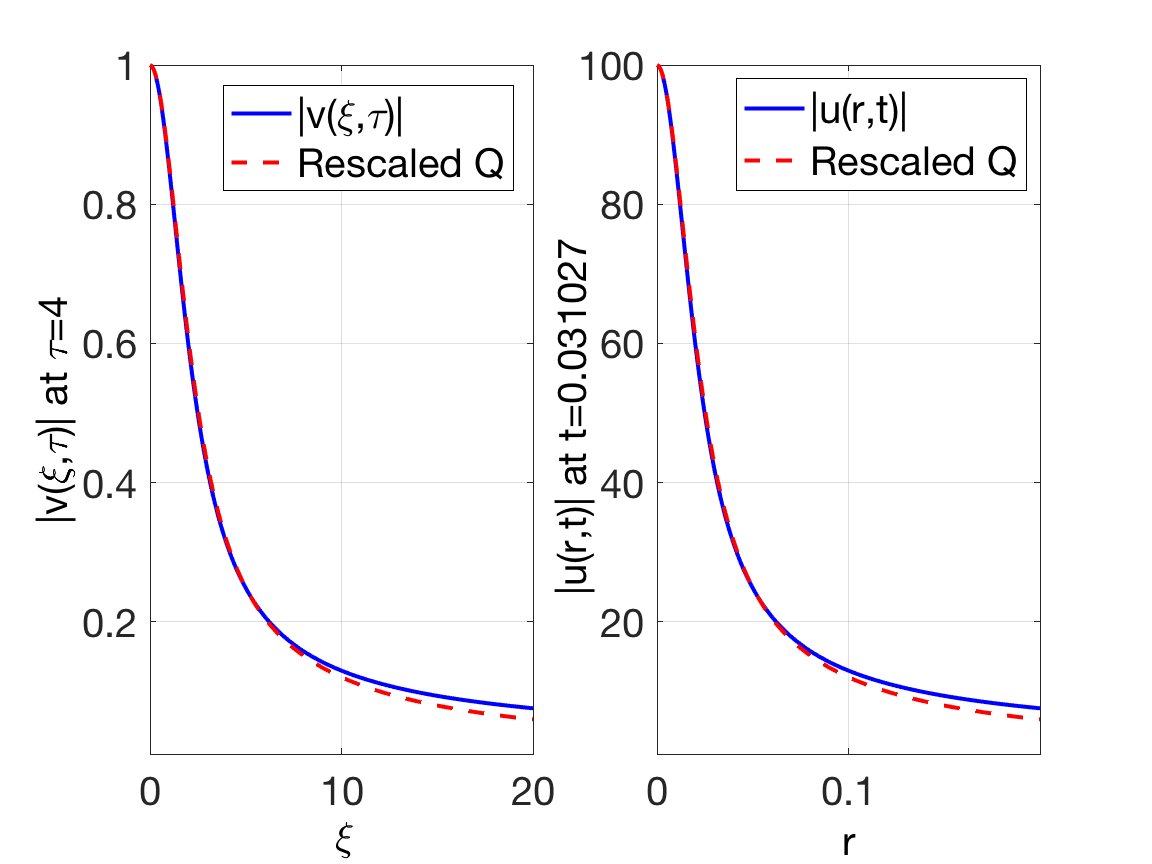}
\includegraphics[width=0.32\textwidth]{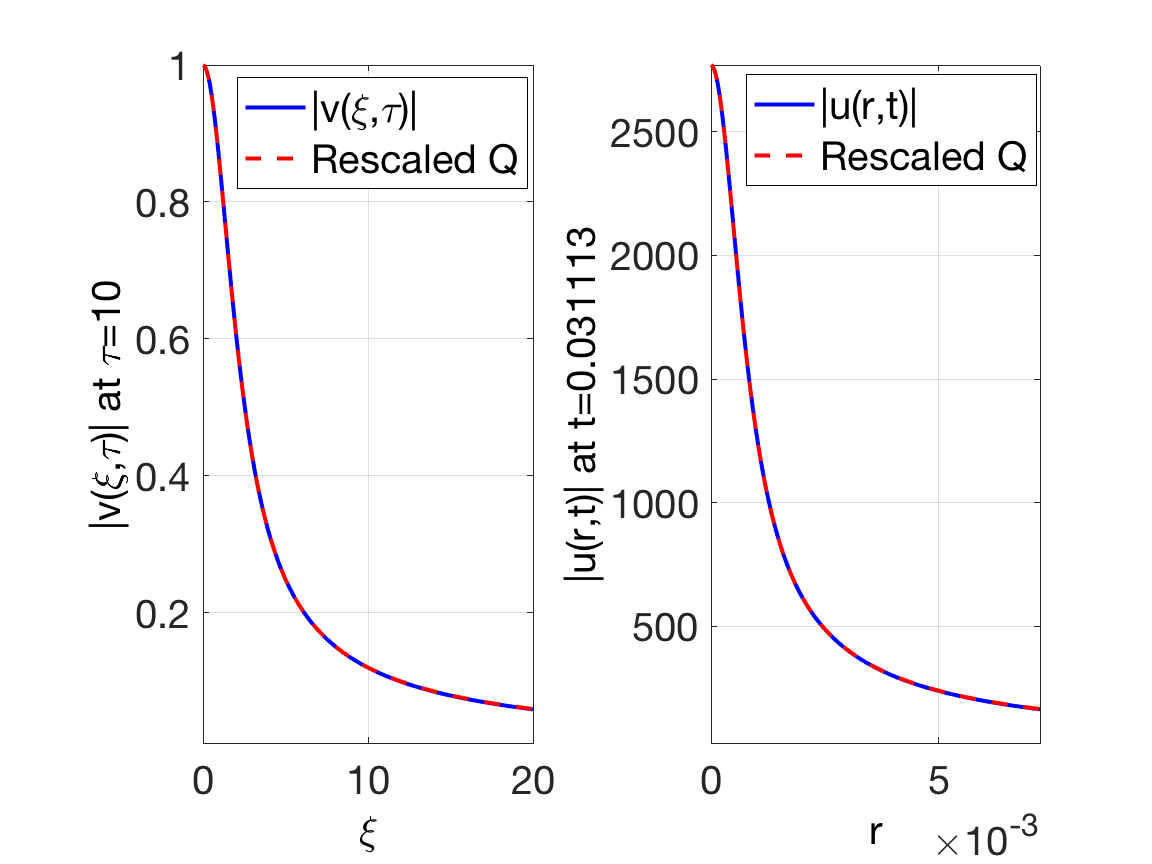}
\includegraphics[width=0.32\textwidth]{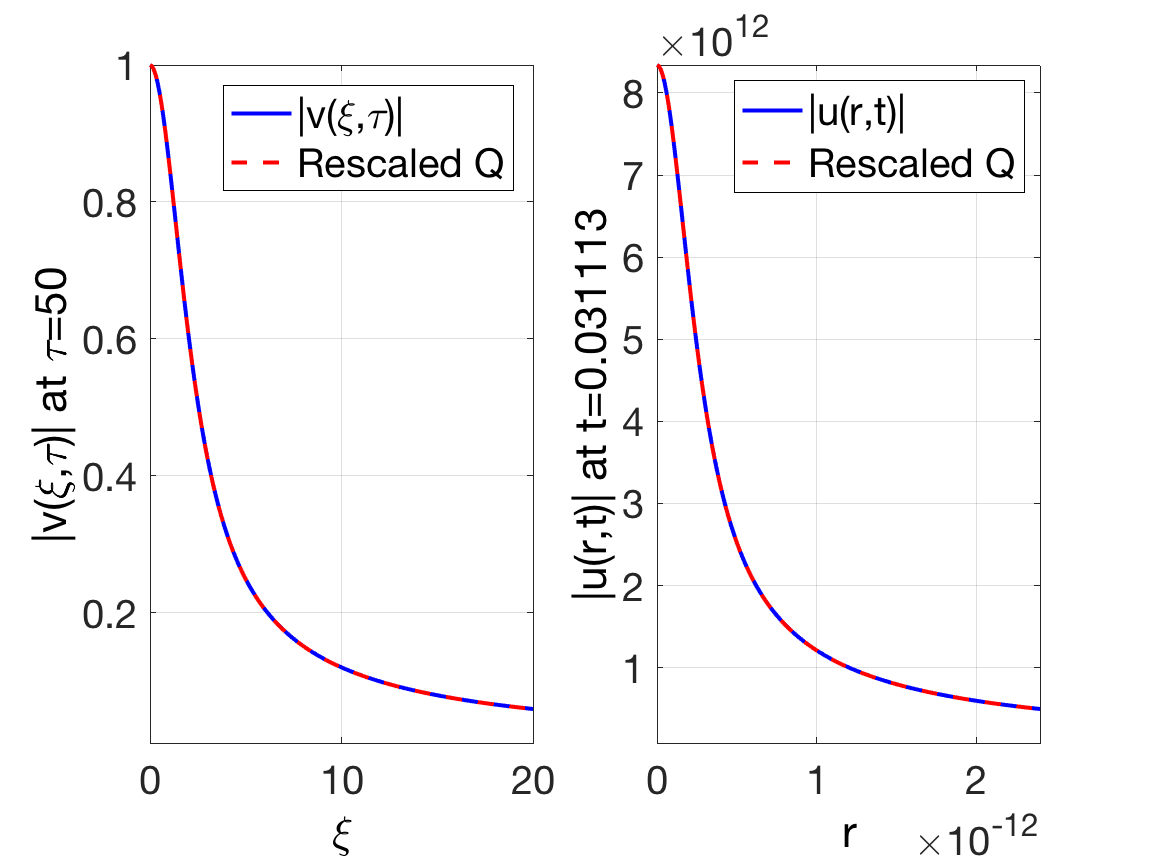}
\includegraphics[width=0.32\textwidth]{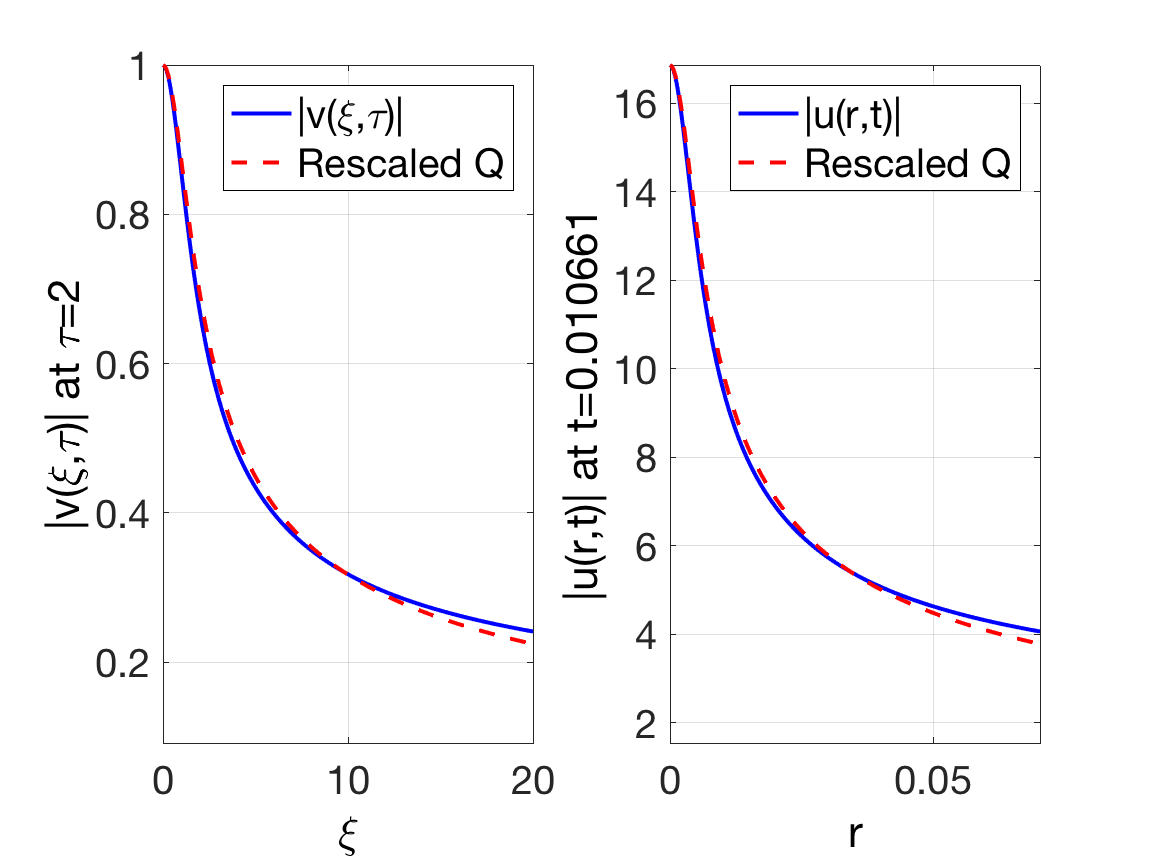}
\includegraphics[width=0.32\textwidth]{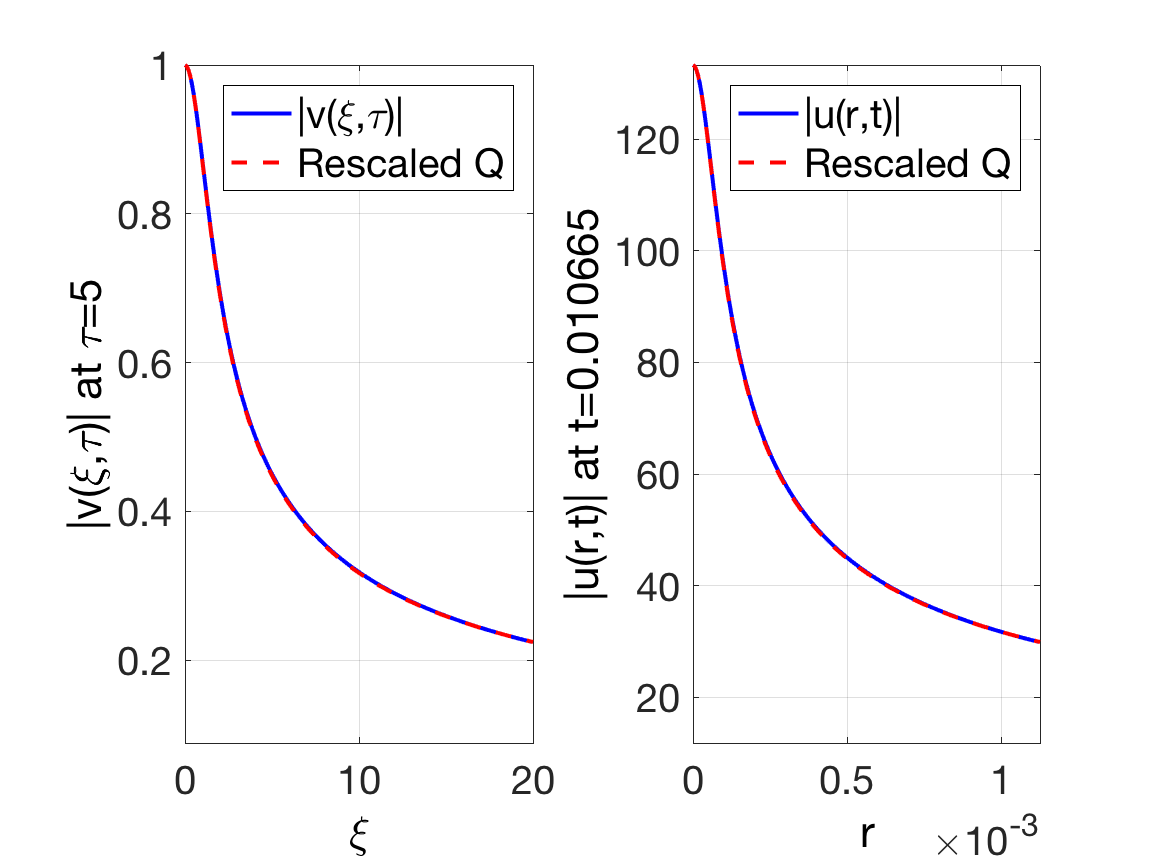}
\includegraphics[width=0.32\textwidth]{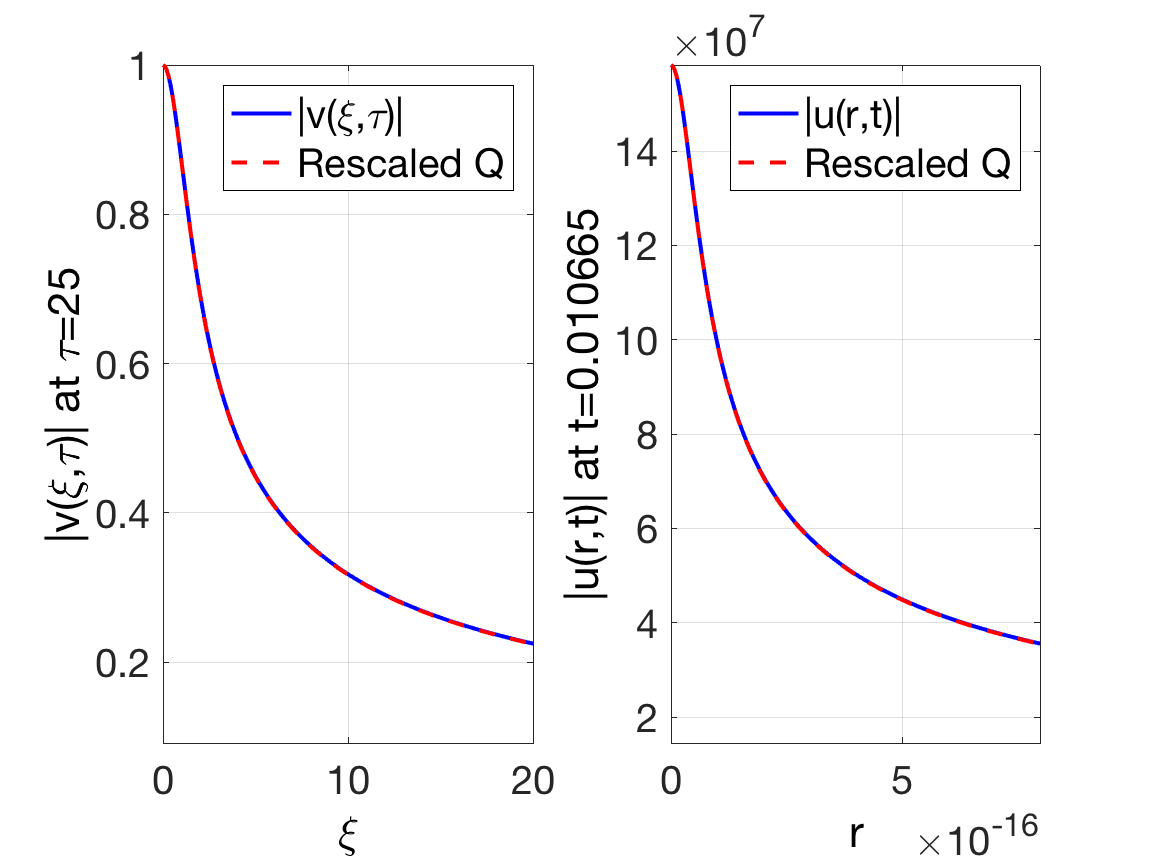}
\caption{ Blow-up profiles for the case $d=3$: $\sigma=1$(top) and $\sigma=2$(bottom) at different times.}
\label{Profi 3d5p}
\end{center}
\end{figure}

\begin{figure}[ht]
\begin{center}
\includegraphics[width=0.32\textwidth]{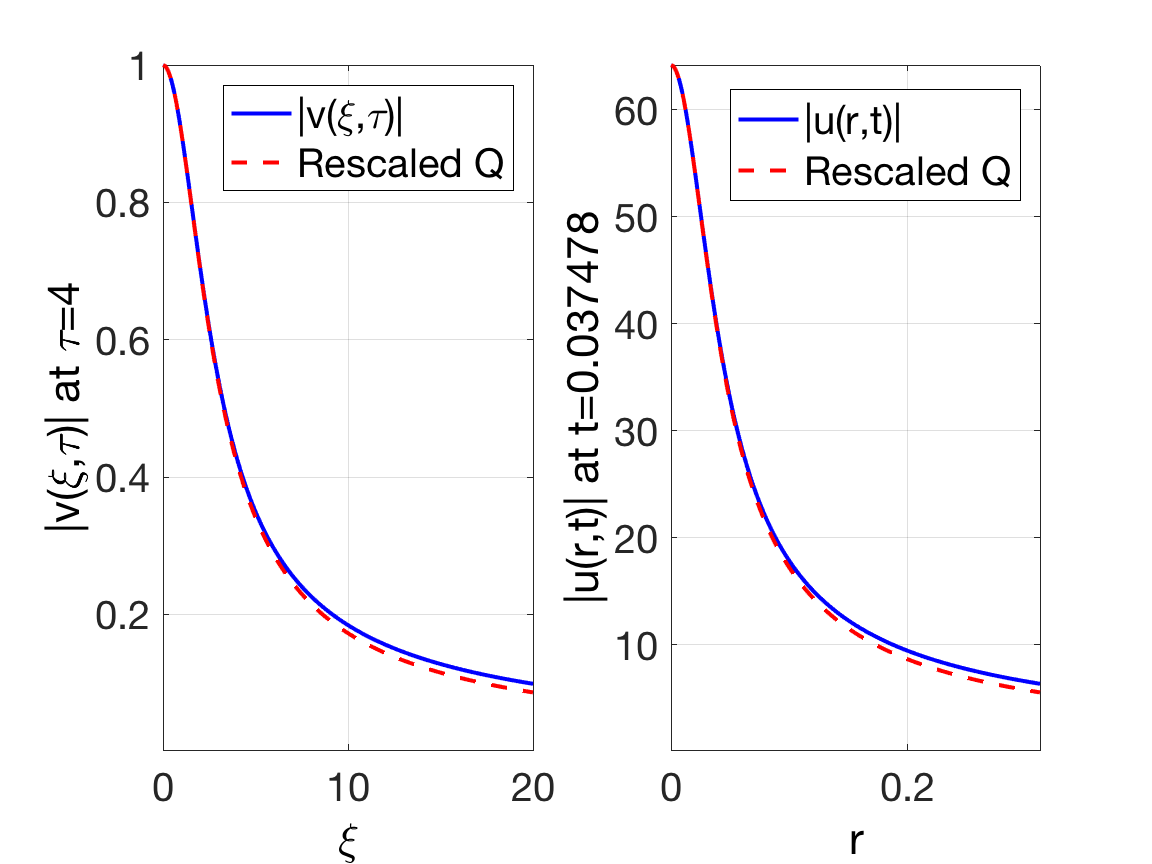}
\includegraphics[width=0.32\textwidth]{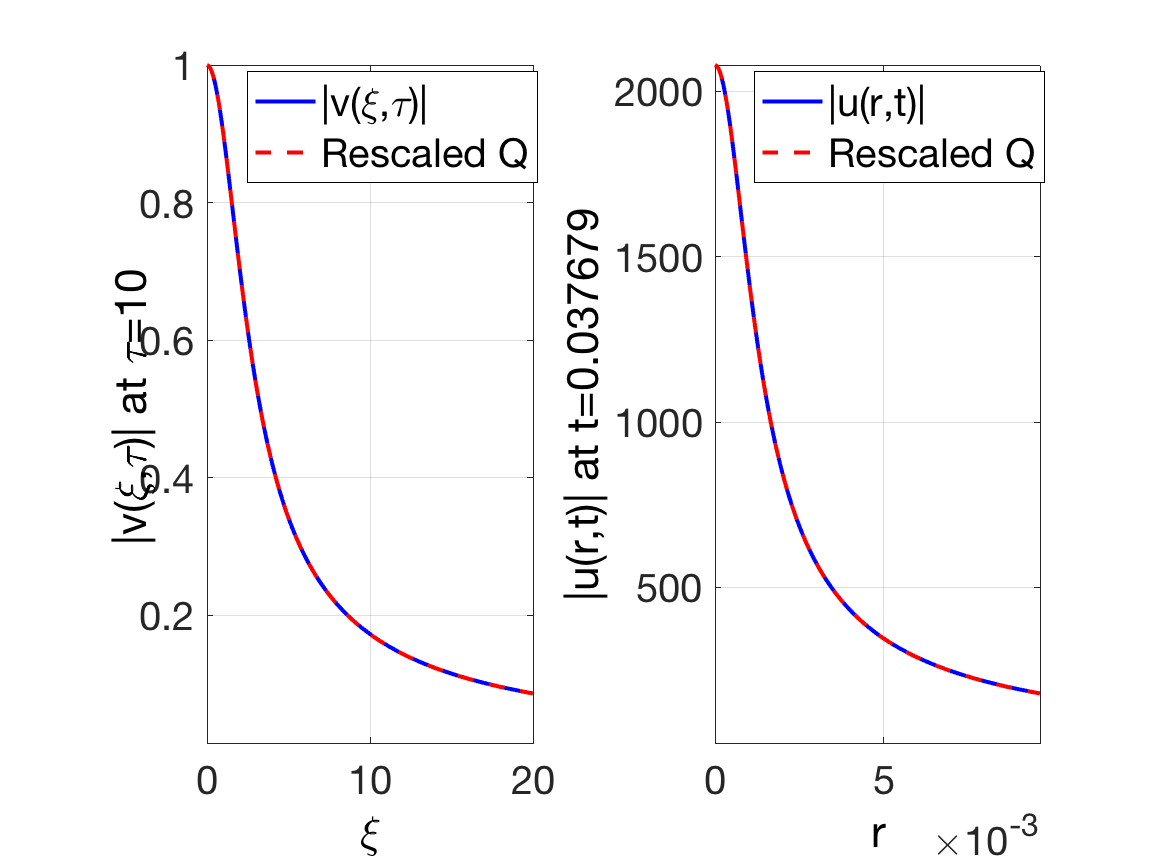}
\includegraphics[width=0.32\textwidth]{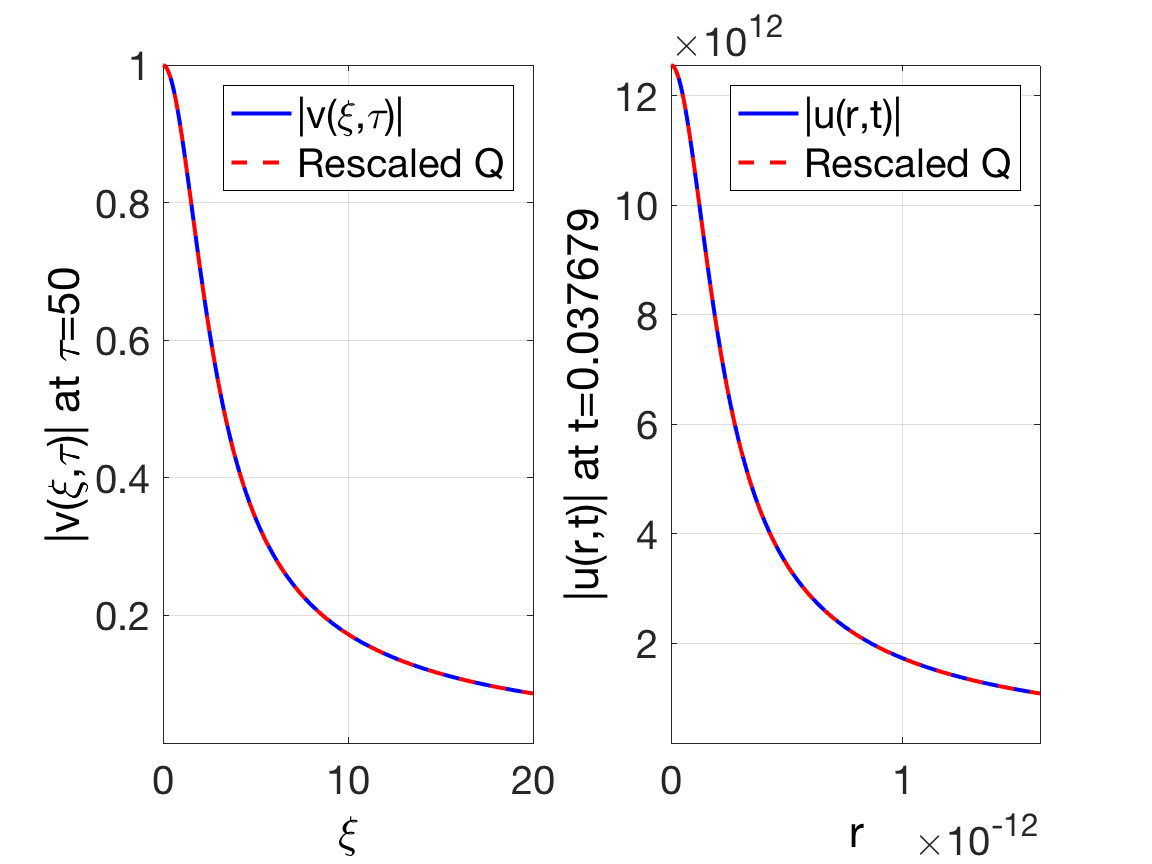}
\includegraphics[width=0.32\textwidth]{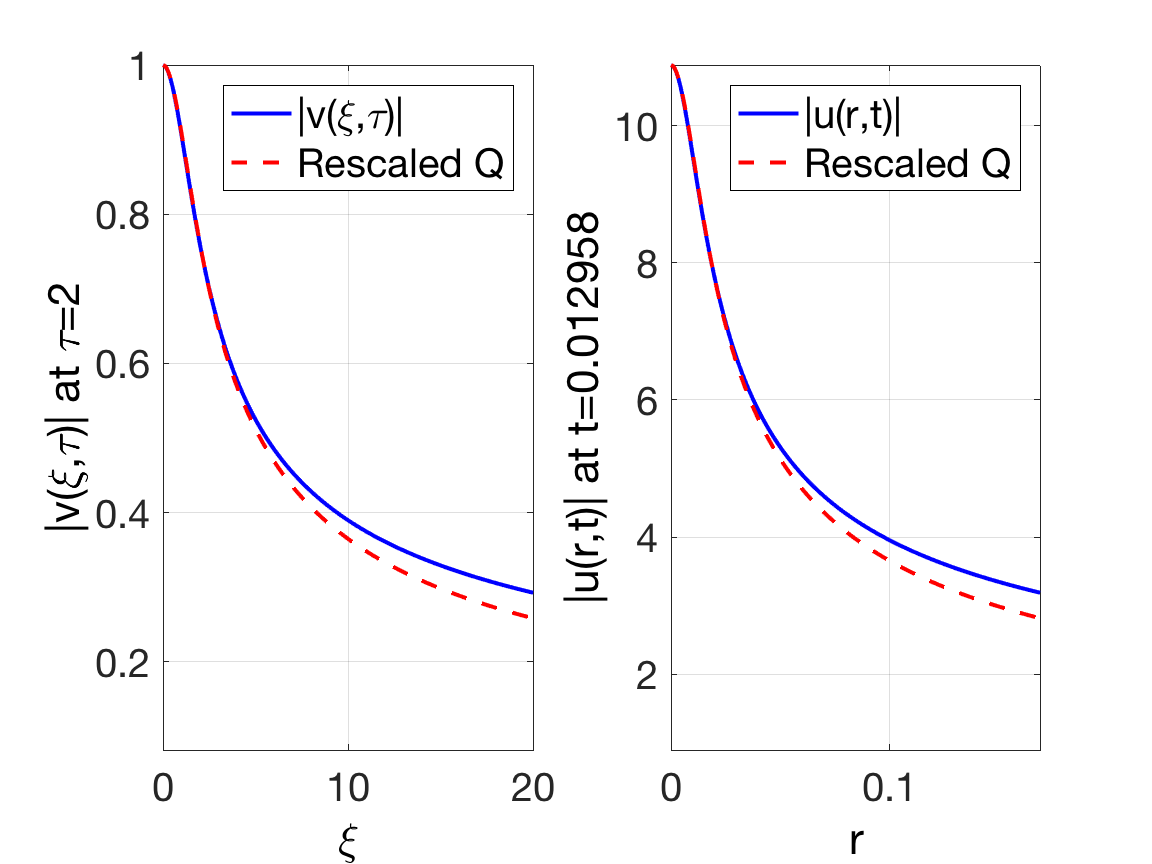}
\includegraphics[width=0.32\textwidth]{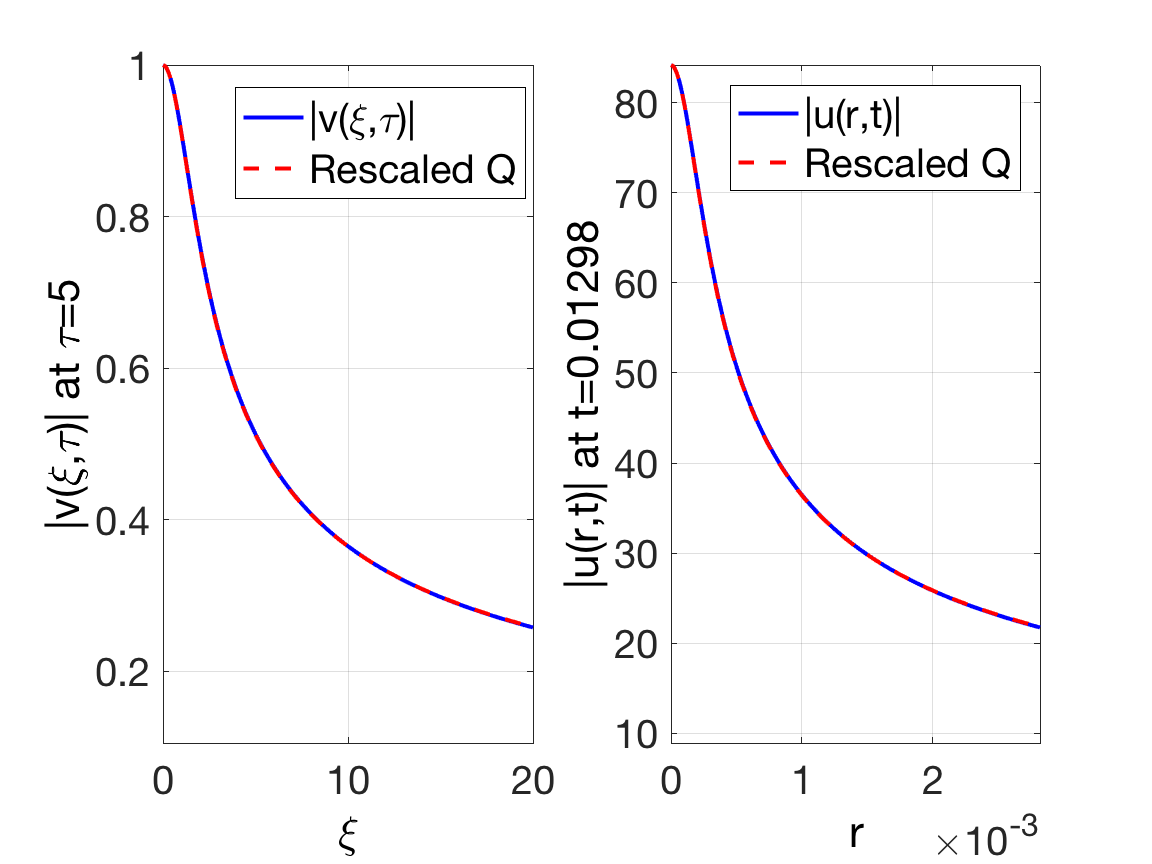}
\includegraphics[width=0.32\textwidth]{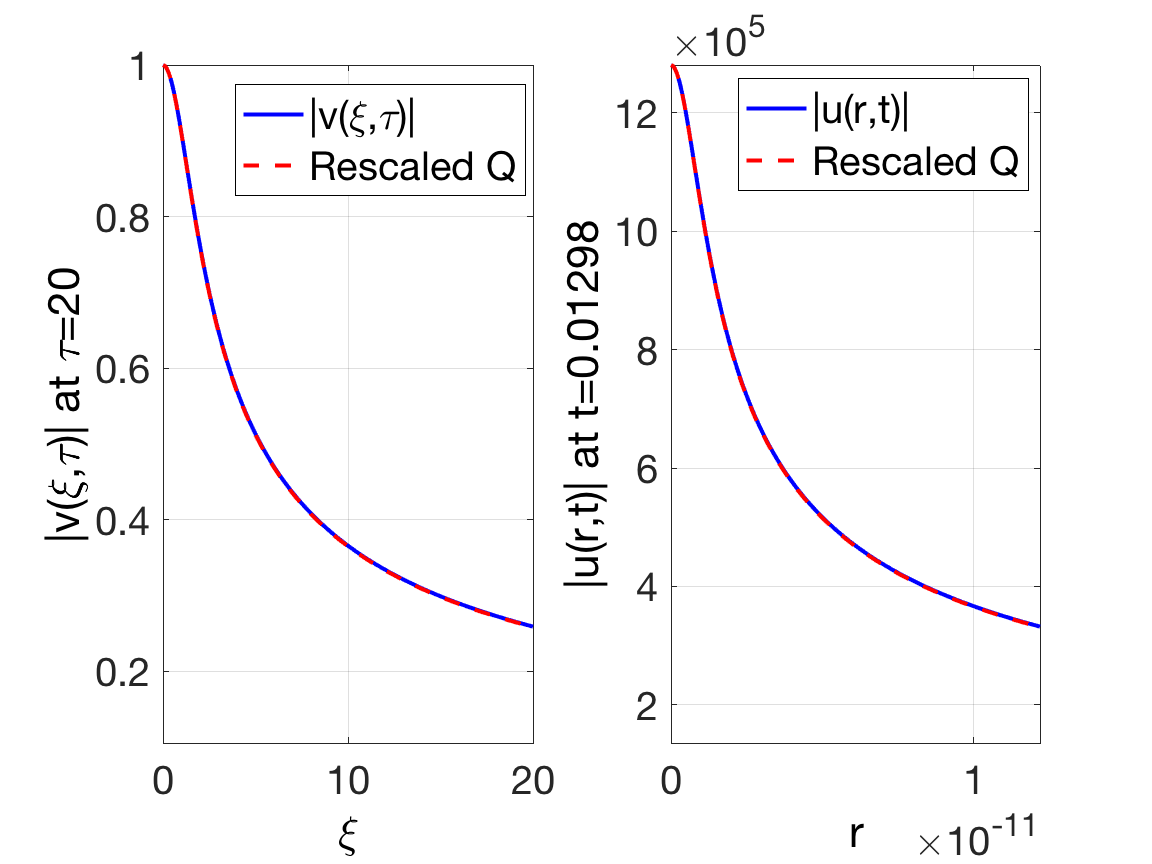}
\caption{ Blow-up profiles for the case $d=4$: $\sigma =1$(top) and $\sigma=2$(bottom) at different times.}
\label{Profi 4d5p}
\end{center}
\end{figure}

Figures \ref{Profi 3d5p} to \ref{Profi 4d5p} show the convergence of the blow-up solution to the profile $Q_{1,0}$, which we obtained in the previous subsection. 

Figures \ref{3d5p data} to 
Figure \ref{4d5p data} show: 
\begin{itemize}
\item[(1)] the slope of $\ln(L)$ vs. $\ln(T-t)$ is approximately $\frac{1}{2}$ (in all supercritical cases of gHartree that we considered);
\item[(2)] the parameter $a(\tau)$ goes to a constant as $\tau \rightarrow \infty$;
\item[(3)] the distance between the rescaled solution $v(\xi,\tau)$ and $Q(\xi)$ with respect to the time $\tau$ by the $L^{\infty}$ norm is small;
\item[(4)] the relative error between the value $|u(0,t)|$ and the predicted blow-up rate $L_{\mathrm{pred}}(t)$. 
\end{itemize}


One can observe that our numerical simulations match the predicted $Q_{1,0}$ blow-up profile really well and the square root rate for $L(t)$ also has a nearly perfect fitting; computationally-wise the matching is on the order of $10^{-3}$ and $10^{-5}$, respectively. This confirms the Conjecture \ref{C:2}.
\begin{figure}[ht]
\begin{center}
\includegraphics[width=0.41\textwidth]{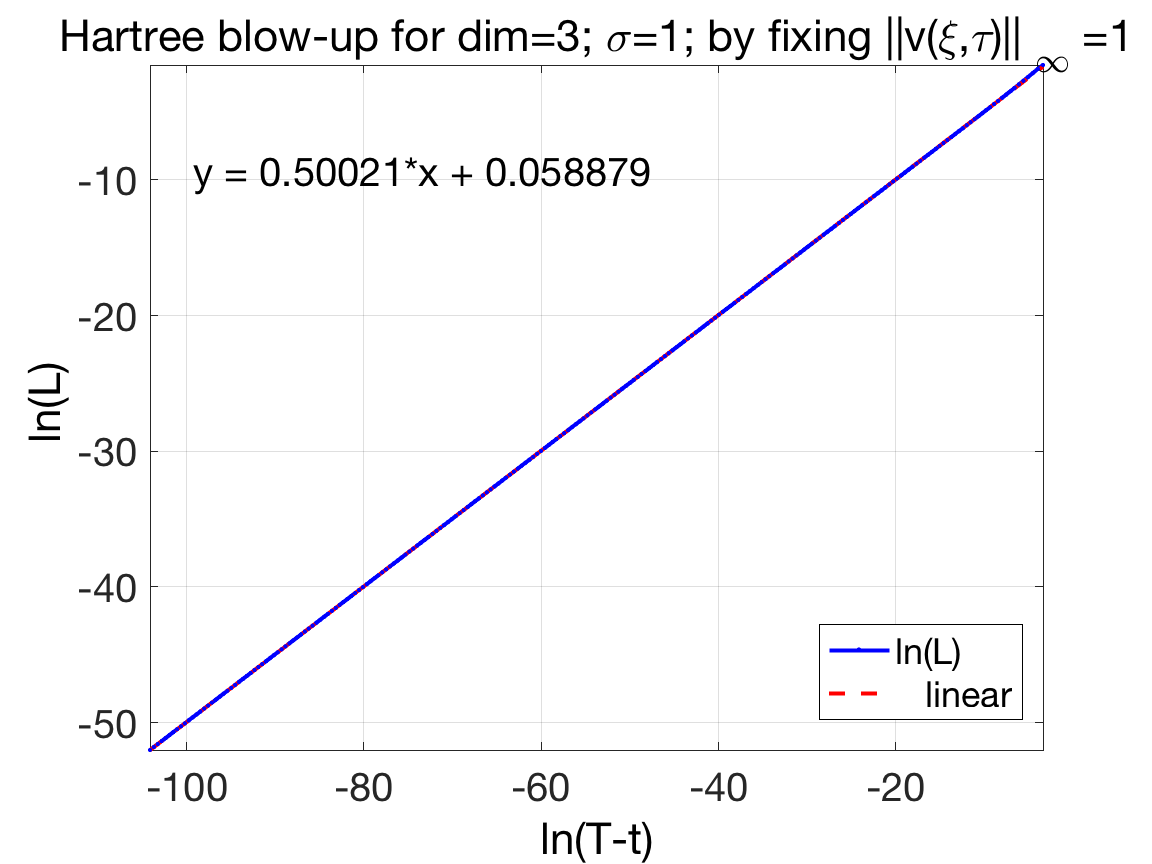}
\includegraphics[width=0.41\textwidth]{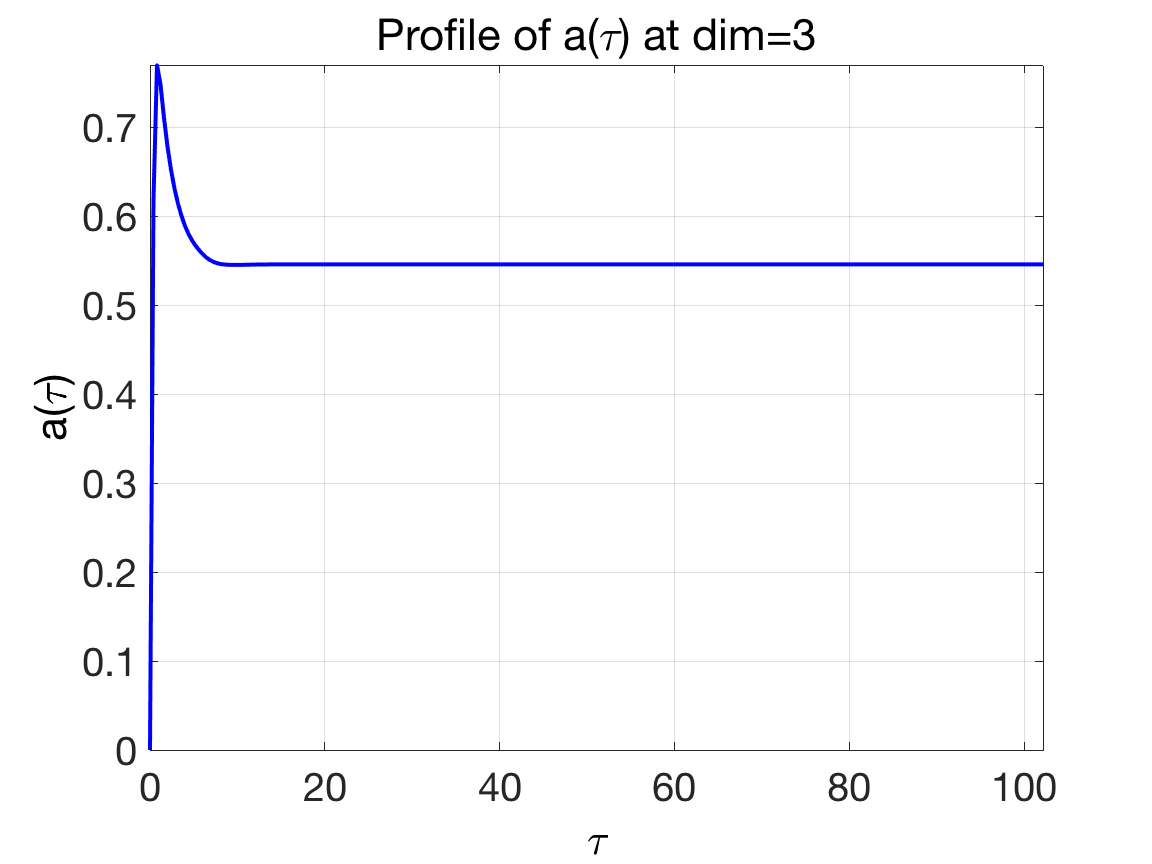}
\includegraphics[width=0.41\textwidth]{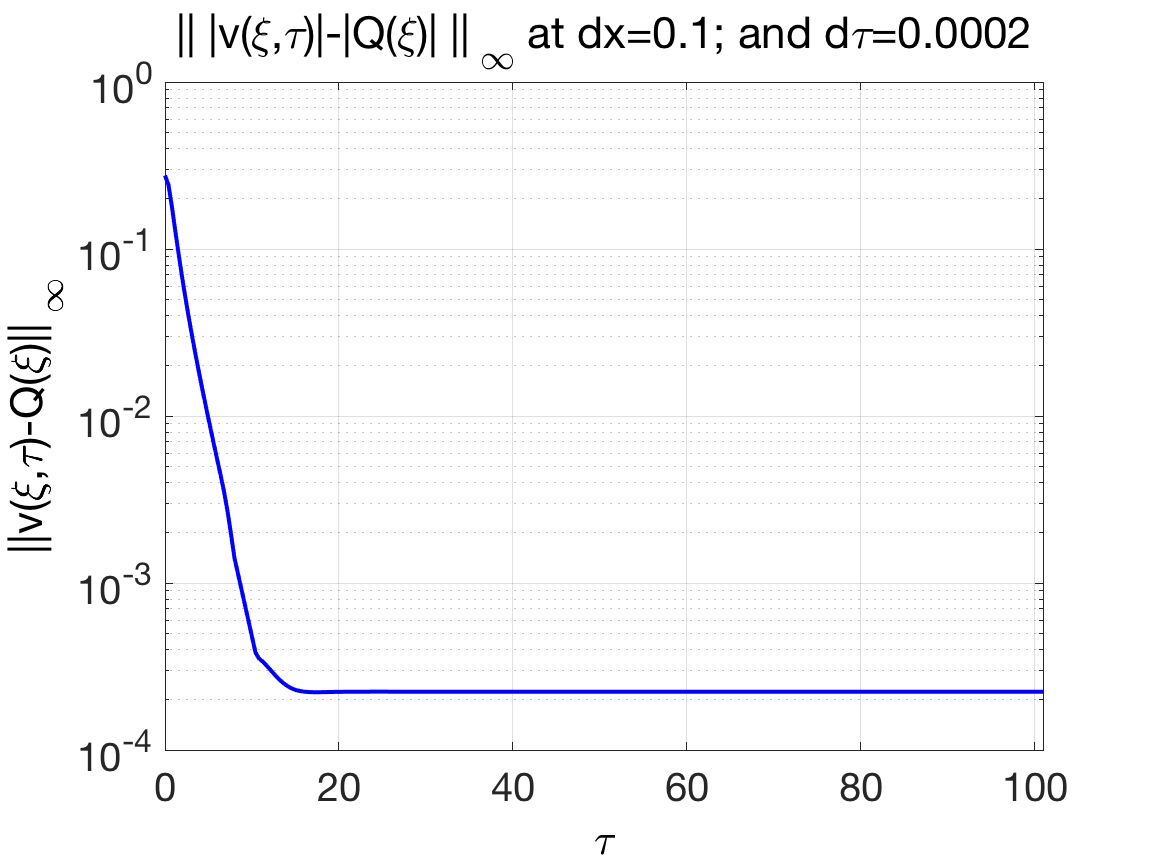}
\includegraphics[width=0.41\textwidth]{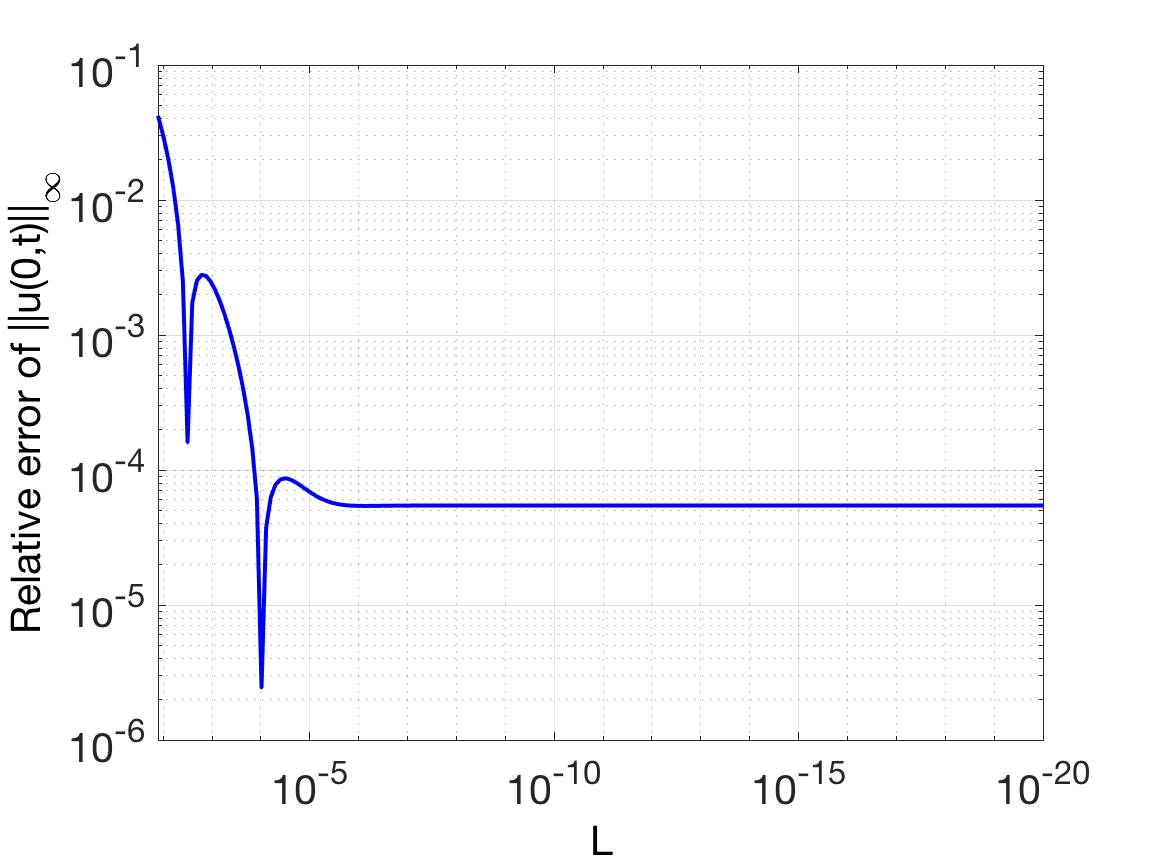}
\includegraphics[width=0.41\textwidth]{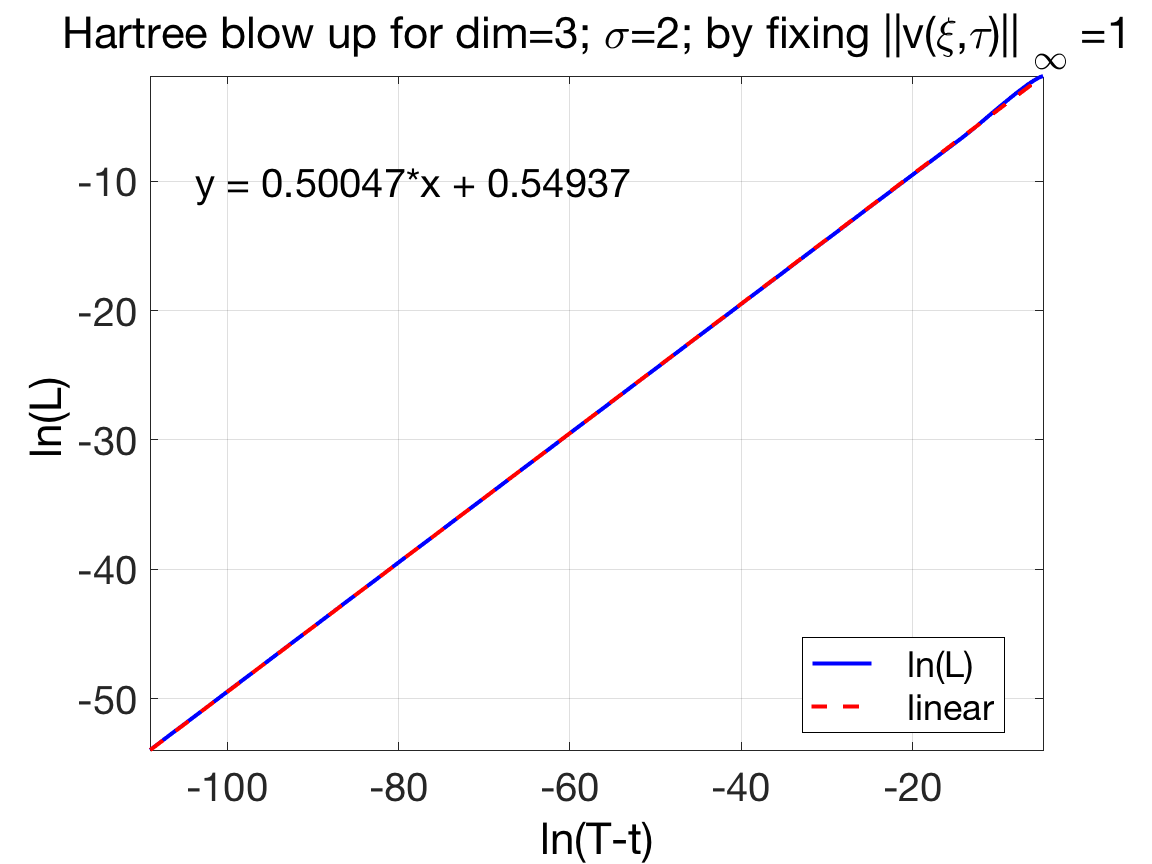}
\includegraphics[width=0.41\textwidth]{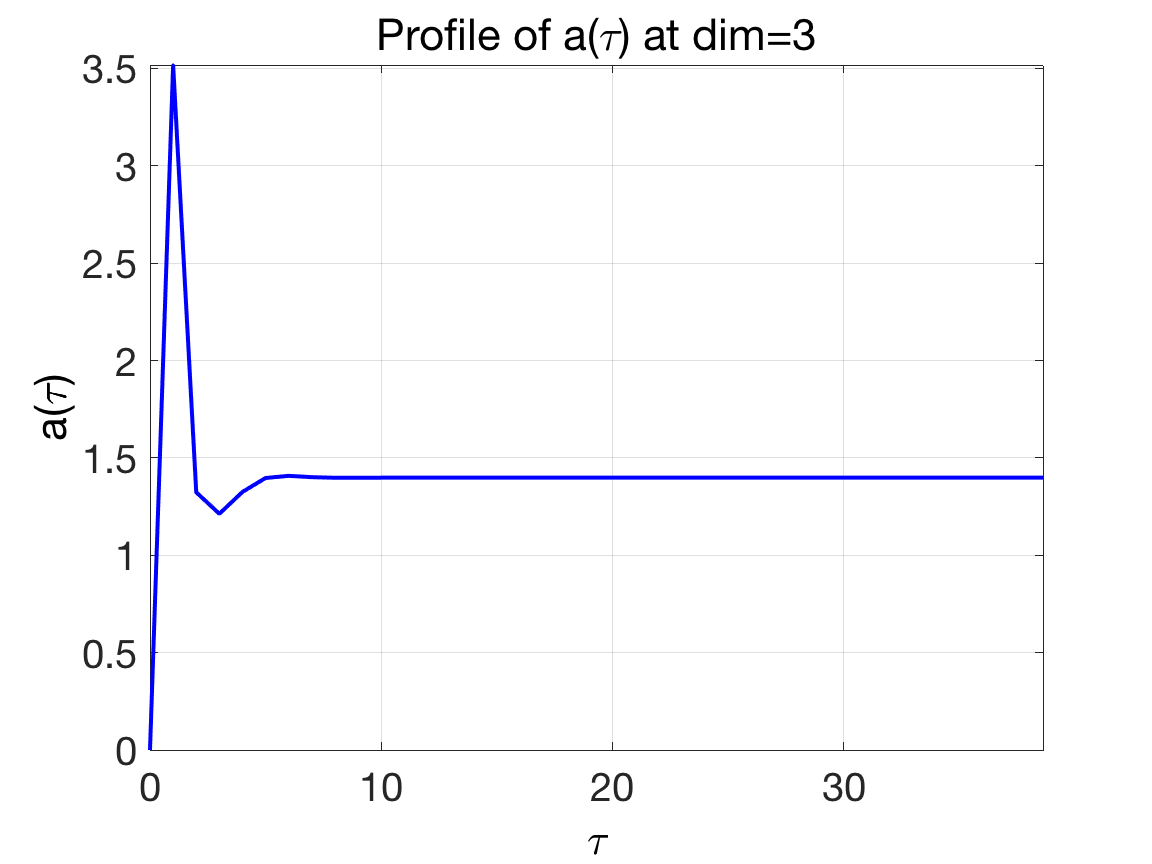}
\includegraphics[width=0.41\textwidth]{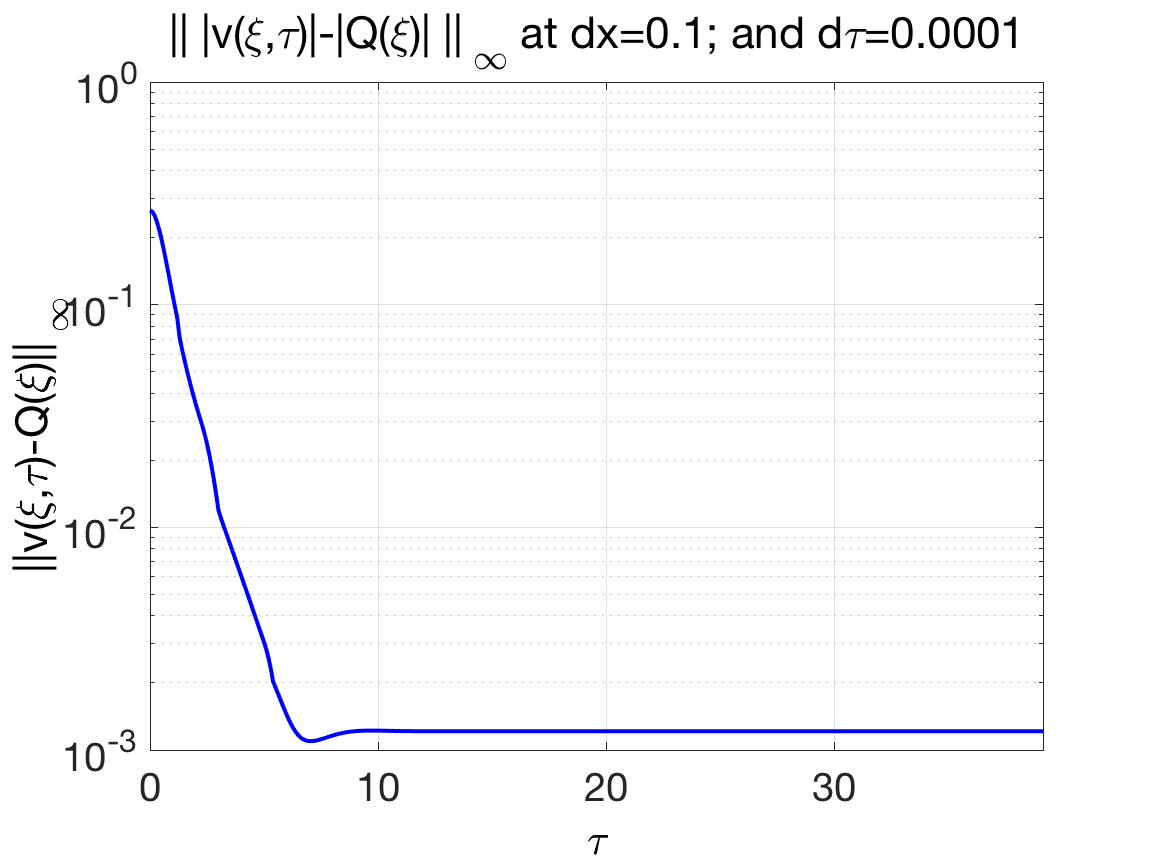}
\includegraphics[width=0.41\textwidth]{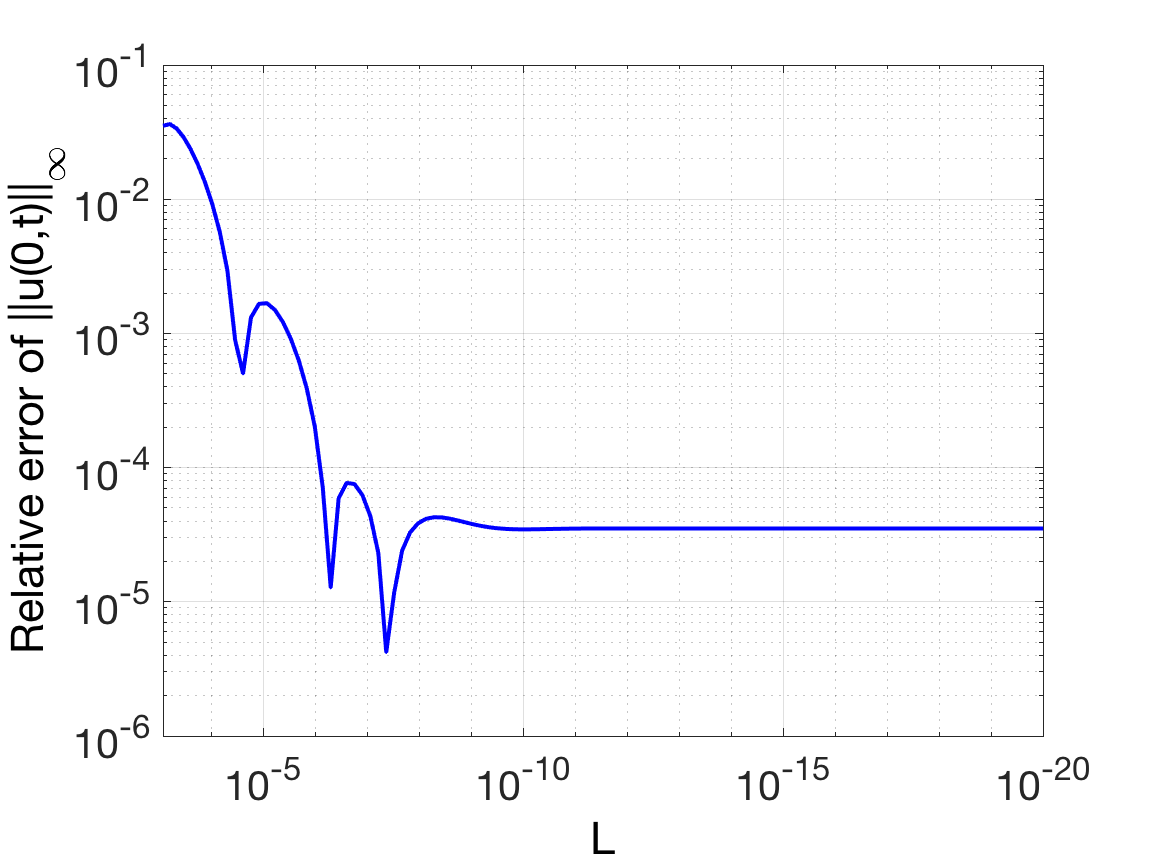}
\caption{ Blow-up data for the 3d cubic (top half) and 3d quintic(bottom half) cases: $\ln(T-t)$ vs. $\ln(L)$ (upper left), the quantity $a(\tau)$ (upper right), the distance between $Q$ and $v$ on time $\tau$ ($\| v(\tau) -Q \|_{L^{\infty}_{\xi}}$) (lower left), the relative error with respect to the predicted blow-up rate (lower right). }
\label{3d5p data}
\end{center}
\end{figure}

\begin{figure}[ht]
\begin{center}
\includegraphics[width=0.41\textwidth]{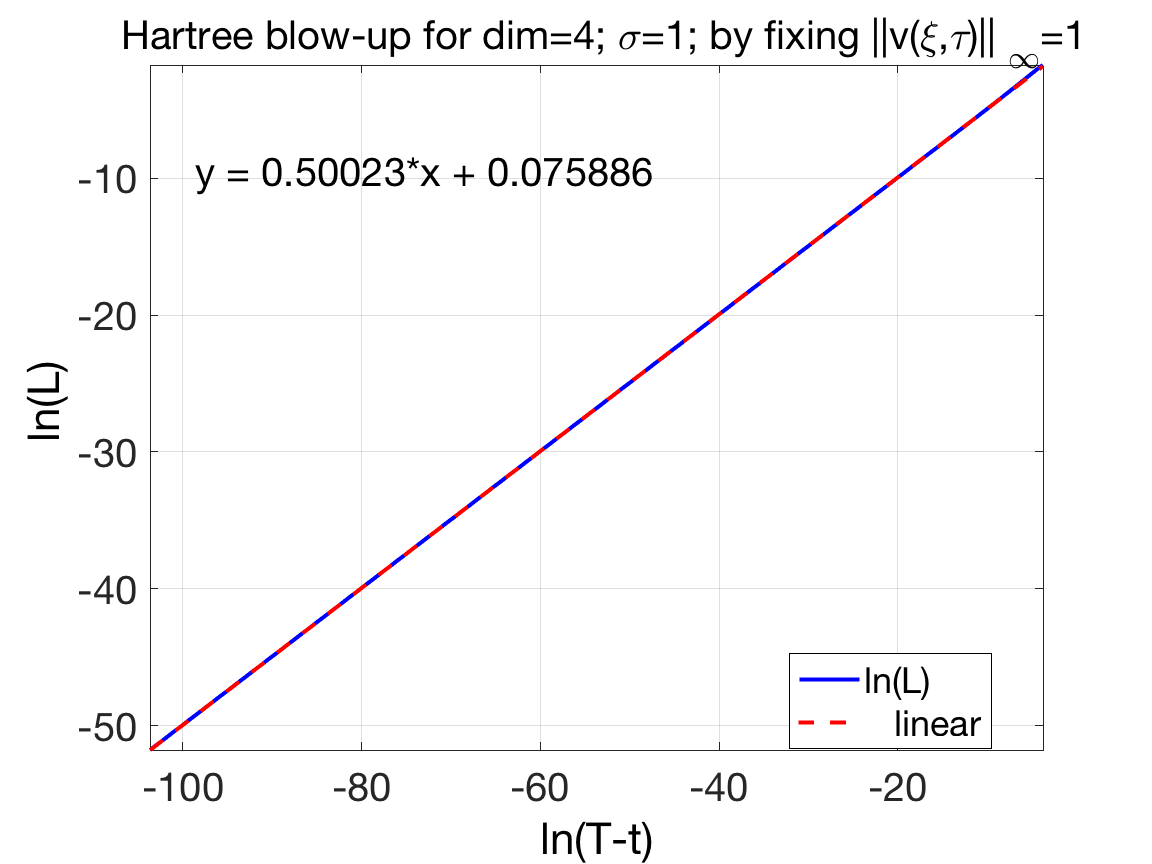}
\includegraphics[width=0.41\textwidth]{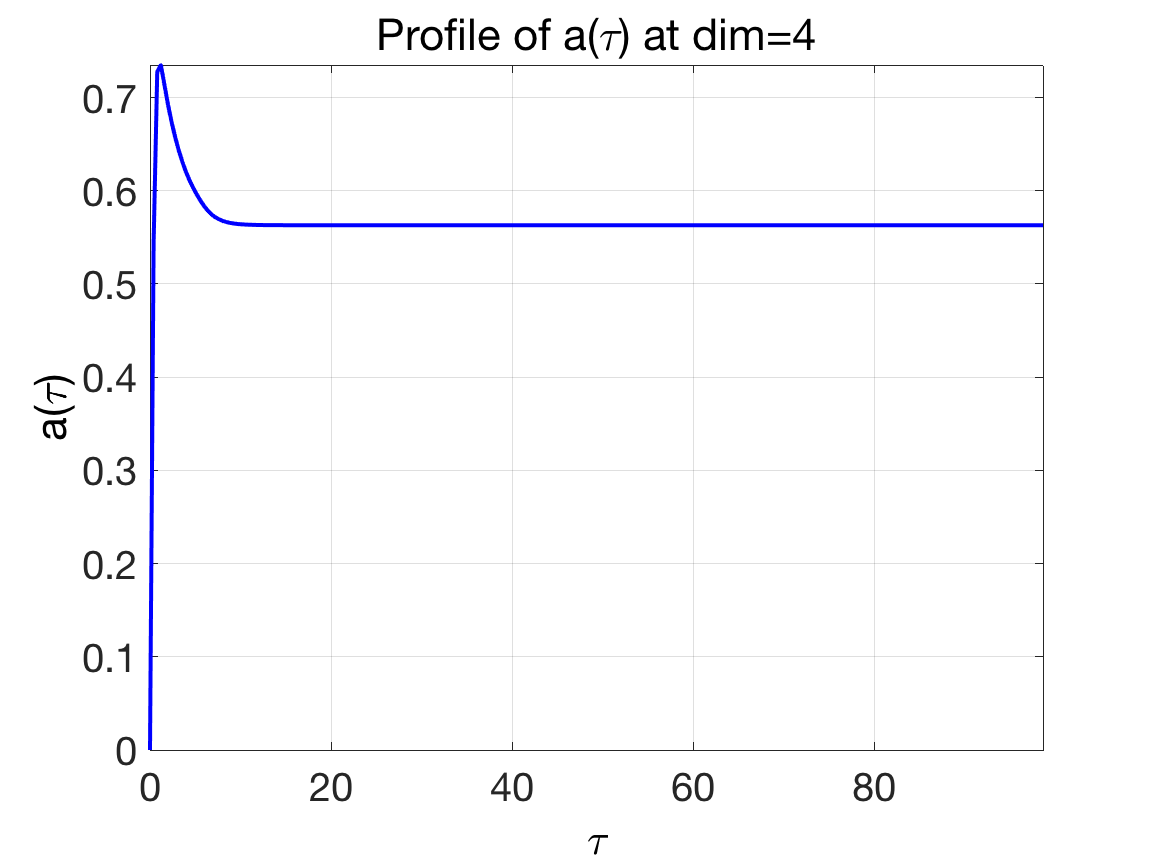}
\includegraphics[width=0.41\textwidth]{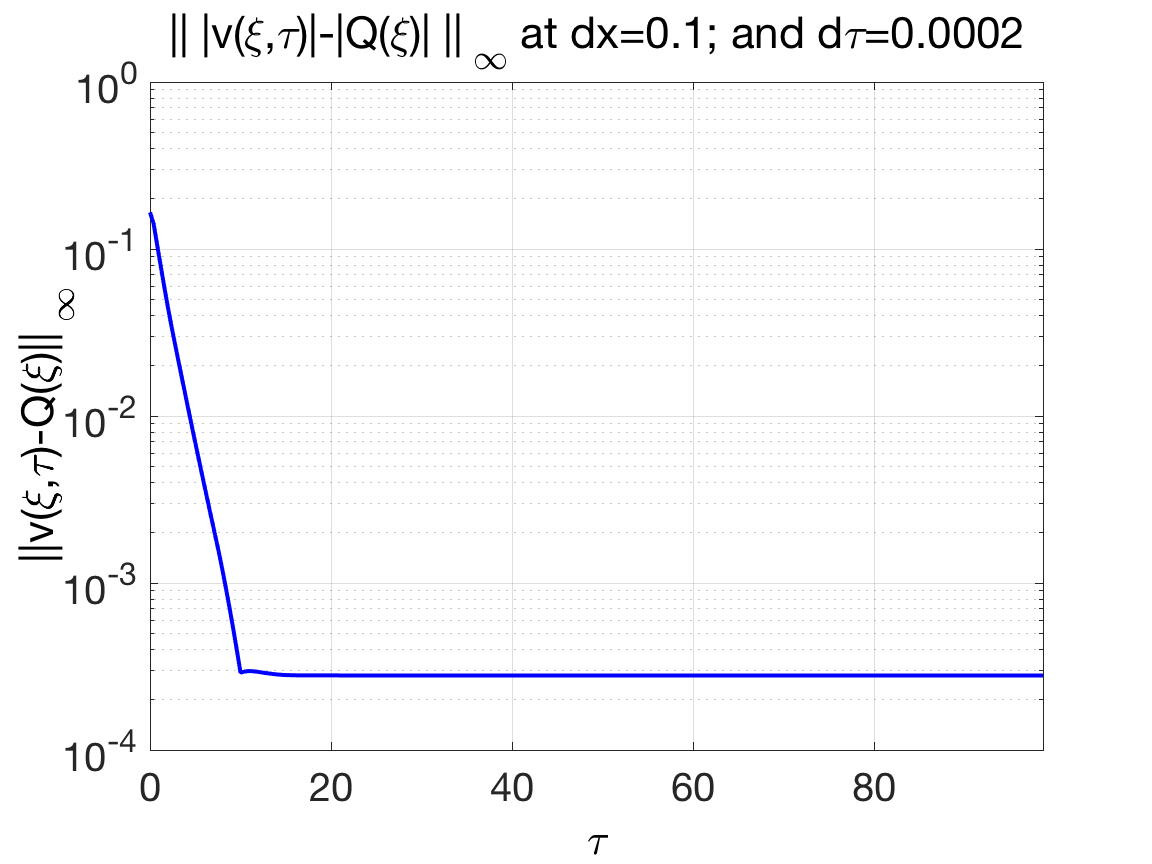}
\includegraphics[width=0.41\textwidth]{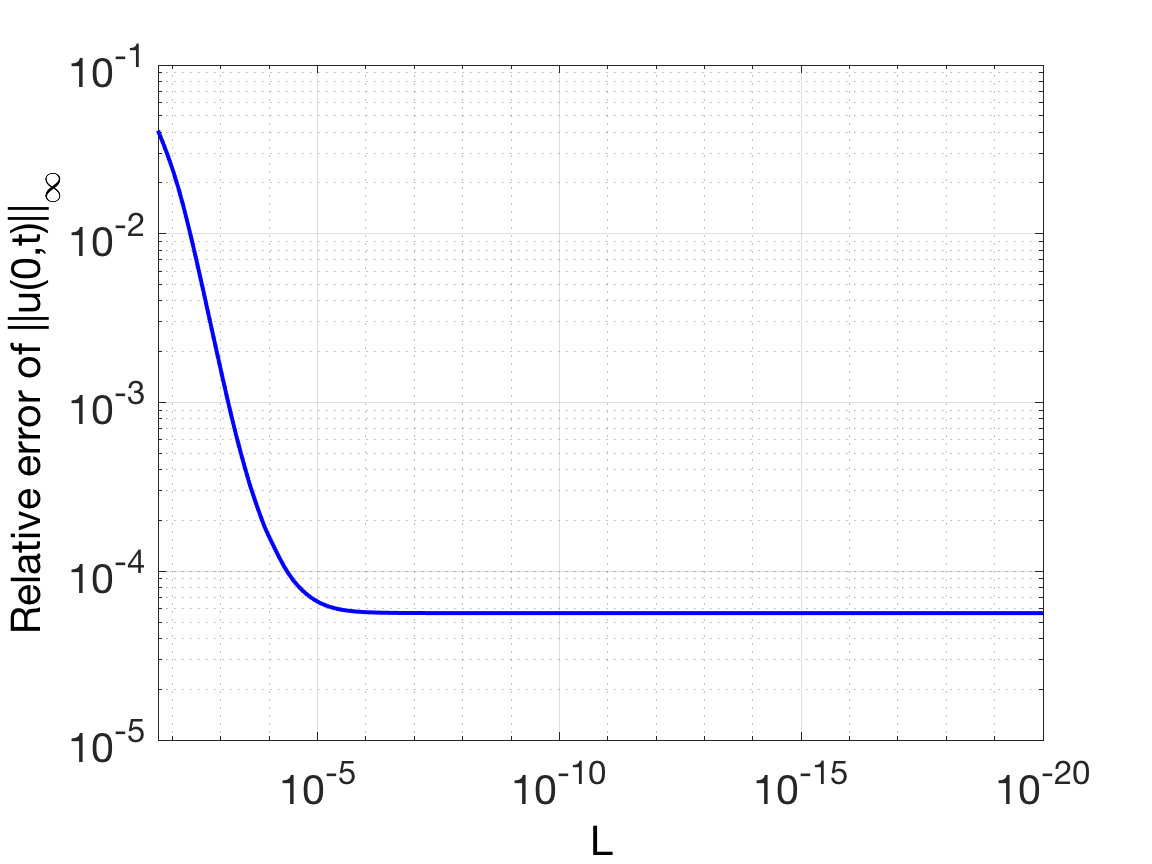}
\includegraphics[width=0.41\textwidth]{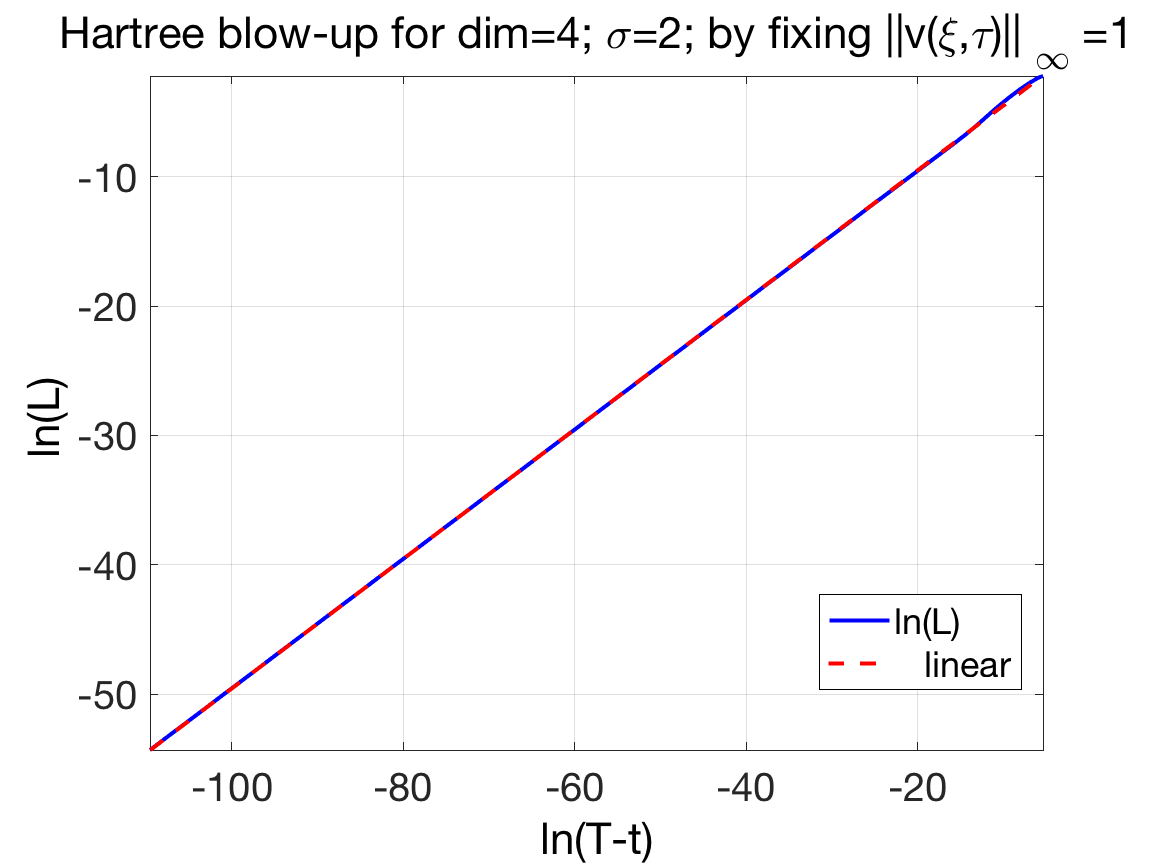}
\includegraphics[width=0.41\textwidth]{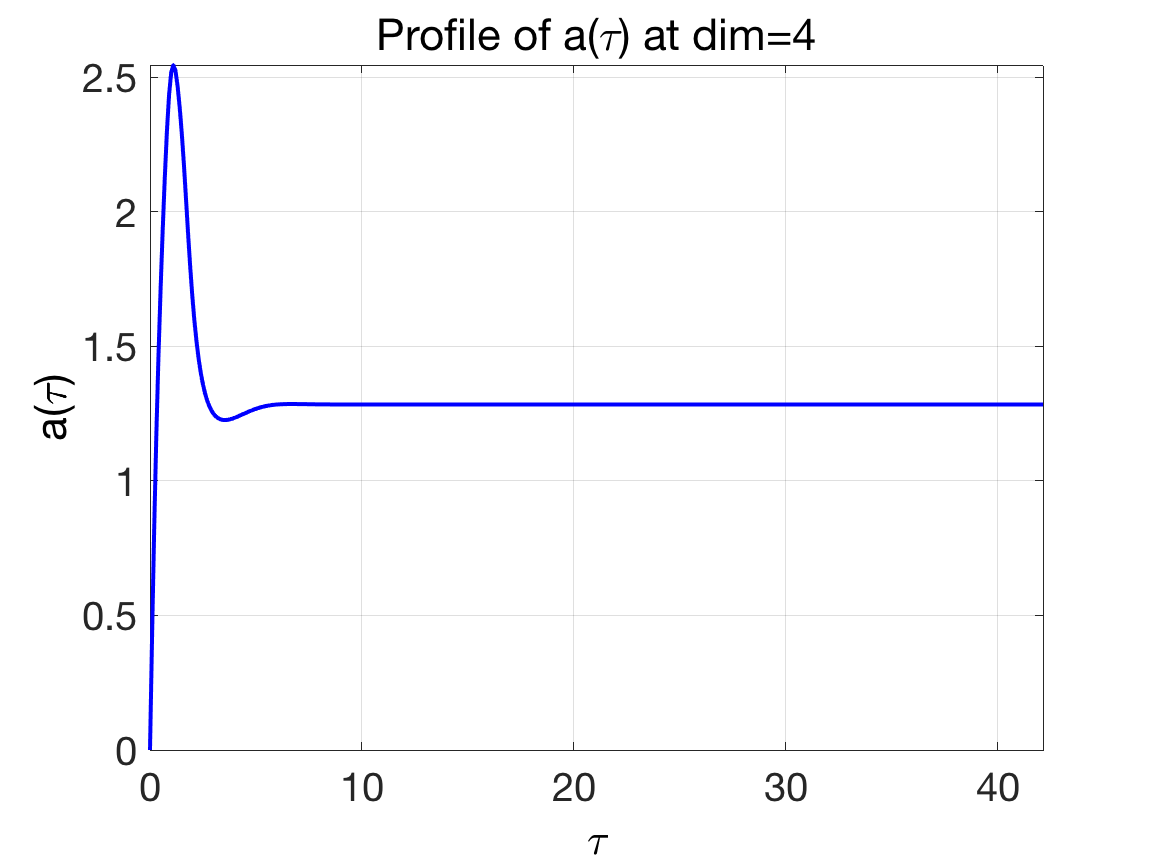}
\includegraphics[width=0.41\textwidth]{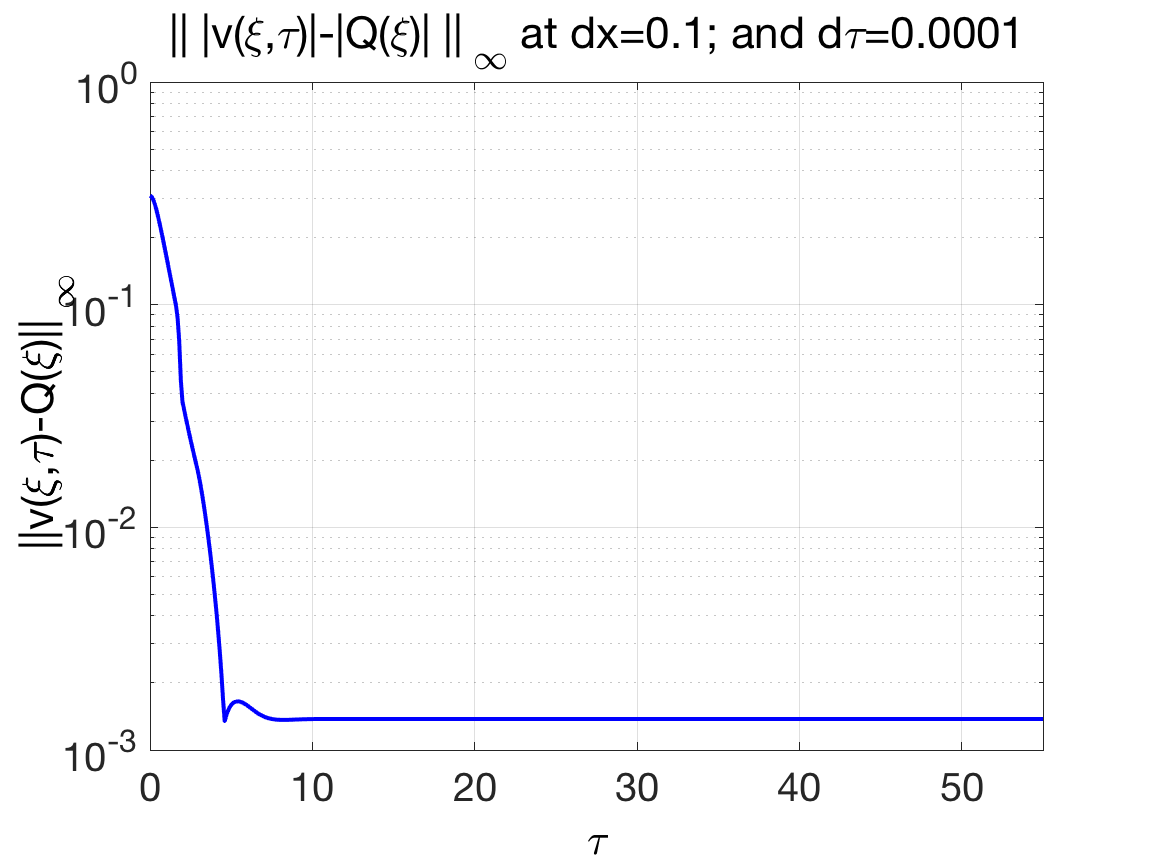}
\includegraphics[width=0.41\textwidth]{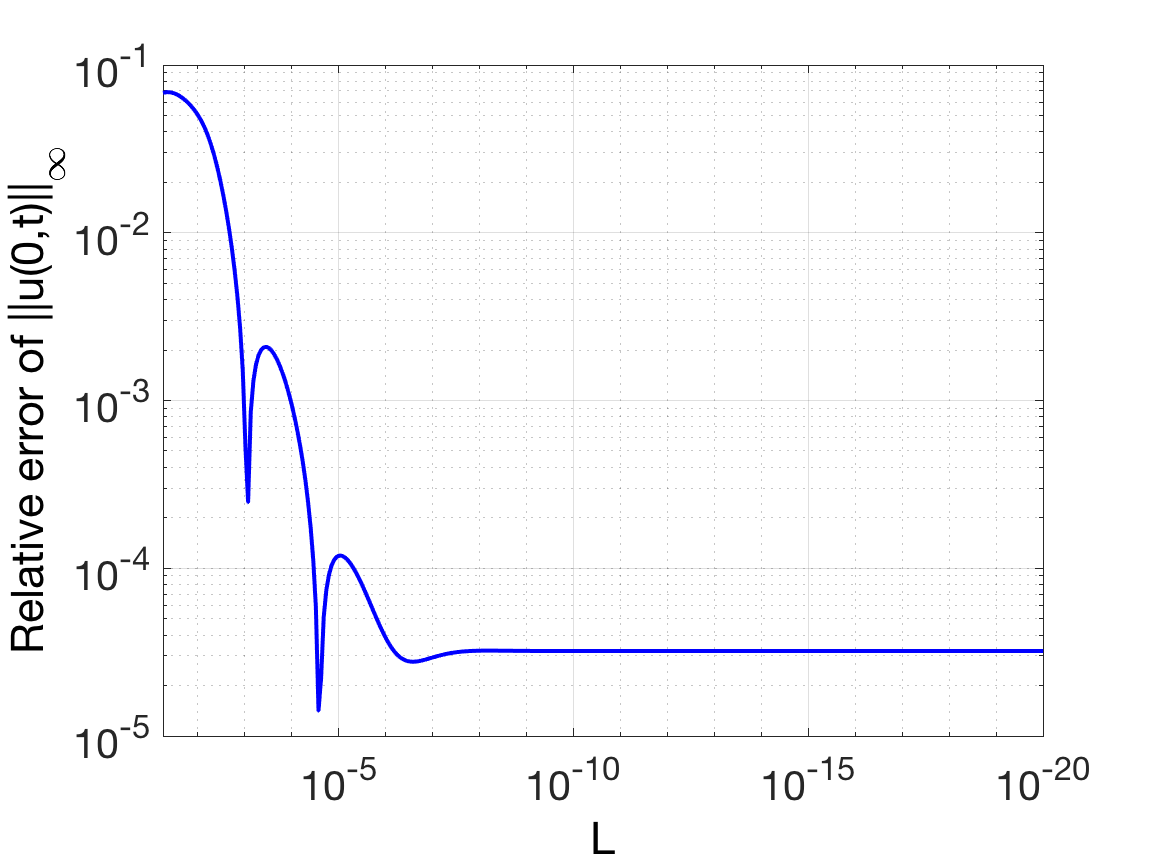}
\caption{ Blow-up data for the 4d cubic (top half) and 4d quintic (bottom half) case: $\ln(T-t)$ vs. $\ln(L)$ (upper left), the quantity $a(\tau)$ (upper right), the distance between $Q$ and $v$ on time $\tau$ ($\| v(\tau) -Q \|_{L^{\infty}_{\xi}}$) (lower left), the relative error with respect to the predicted blow-up rate (lower right).  }
\label{4d5p data}
\end{center}
\end{figure}

\newpage

\clearpage

\section{Conclusions}
This work is the {\it first} attempt to study stable blow-up solutions in the standard and generalized Hartree equations in both $L^2$-critical and $L^2$-supercritical cases, from asymptotical analysis approach and via numerical simulations. 

We are able to obtain rates and blow-up profiles in the cases considered, and observe that the stable blow-up dynamics in the nonlocal gHartree equation is very similar to the NLS case. Such modification of nonlinearity  does not affect the dynamics of the stable blow-up singularity formation. It would be interesting to investigate further this work rigorously as well as to understand whether there are modifications of nonlinearity or potentials (local or nonlocal, or a certain combination) such that the stable formation of singularity would change the dynamics in the singularity formation from the known NLS-type blow-up dynamics. 


\section{Appendix}

Here, we compute the ground state $Q$ via the renormalization method from \cite[Chapter 28]{F2015}, \cite{PS2004}. We rewrite the equation (\ref{GS}) as
\begin{align}
(-\Delta +1)Q= \mathcal{N}(Q),
\end{align}
where $\mathcal{N}(Q)=\left((-\Delta)^{-1}|Q|^{p} \right)|Q|^{p-2} Q$ is the nonlinear part. Multiplying the $Q$ and integrating on both sides, we have
\begin{align}
SL(Q):=\int_{\mathbb{R}^d} Q^2 = \int_{{\mathbb{R}^d}} Q (-\Delta +1)^{-1}\mathcal{N}(Q)=:SR(Q).
\end{align}

To prevent the fixed point iteration from going to $0$ or $\infty$, we multiply (7.2) by a constant $c_i$ in each iteration, i.e.,
$$SL(c_iQ^{(i)})=SR(c_iQ^{(i)}).$$
From above, we immediately have
\begin{align}
c_i=\left(\frac{SR(Q^{(i)})}{SL(Q^{(i)})}\right)^{\frac{1}{2p-2}}.
\end{align}

Now, we can apply the fixed point iteration as follows
\begin{align}
Q^{(i+1)}&=({-\Delta_N+I_N})^{-1} \mathcal{N}(c_iQ^{(i)})
= \left(\frac{SR(Q^{(i)})}{SL(Q^{(i)})}\right)^{\frac{2p-1}{2p-2}}({- \Delta_N+I_N})^{-1} \mathcal{N}(Q^{(i)})
\end{align}
until we reach the desired accuracy, say $\|Q^{(i+1)}-Q^{(i)}\|_{\infty}<10^{-12}$ in our calculation. Here $-\Delta_N$ is the discretized Laplacian operator of size $N+1$ described in Section 2, and the $I_N$ is the identity matrix of size $N+1$.

\begin{remark}\label{Remark: Q unique}
We tried different non-trivial initial guesses for $Q^{(0)}$ (including different $Q^{(0)}(0)$), the algorithm always converges to the same profile $Q^{(\infty)}$. It is due to the convergence property of this algorithm, see \cite{PS2004}. While it does not answer the uniqueness of the profile question, numerically it suggests the uniqueness of the ground state.
\end{remark}


\bibliography{Kai_bib_Hartree}
\bibliographystyle{abbrv}

\end{document}